\documentclass[12pt]{amsart}
\usepackage[utf8]{inputenc}
\usepackage[initials]{amsrefs}
\usepackage[
        colorlinks,
]{hyperref}

\usepackage{fullpage,amssymb,amsmath,dsfont}

\numberwithin{equation}{section}

\makeatletter
\newcommand*{\rom}[1]{\expandafter\@slowromancap\romannumeral #1@}
\makeatother

\newcommand{\Gal}{{\rm Gal}}
\newcommand{\Ord}{{\rm Ord}}

\newcommand{\GL}{{\rm GL}}
\newcommand{\Id}{{\rm Id}}
\newcommand{\Aut}{{\rm Aut}}

\newcommand{\Q}{\mathbb Q}
\newcommand{\R}{\mathbb R}
\newcommand{\F}{\mathbb F}
\newcommand{\N}{\mathbb N}
\newcommand{\C}{\mathbb C}
\newcommand{\Z}{\mathbb Z}
\newcommand{\Irr}{{\rm Irr}}

\renewcommand{\Re}{{\mathfrak R}{\rm e}}

\newtheorem{X}{X}[section]
\newtheorem{cor}[X]{Corollary}
\newtheorem{lem}[X]{Lemma}
\newtheorem{prop}[X]{Proposition}
\newtheorem{thm}[X]{Theorem}

\theoremstyle{definition}
\newtheorem{defi}[X]{Definition}

\theoremstyle{remark}
\newtheorem{rk}{Remark}[section]
\newtheorem*{exe}{Example}
\newtheorem*{examples}{Examples}

\begin{document}

\title{On the role of higher roots in prime ideal races}
\author{Mounir Hayani}
\address{Max Planck Institute For Mathematics, Vivatsgasse 7, 53111, Bonn, Germany}
\email{hayani@mpim-bonn.mpg.de}
\date{}

\begin{abstract}
Let $L/K$ be a Galois extension of number fields, and let $C_1, C_2$ be conjugacy classes of $\Gal(L/K)$. By introducing algebraic parameters related to $C_1$ and $C_2$, we provide a conditional characterization for Chebyshev's bias logarithmic density $\delta_{L/K}(C_1,C_2)$ to lie in the interval $(1/2, 1)$, meaning that the prime ideal race is biased. Unlike existing examples in the literature, where the bias comes from a difference in the number of square roots, the order of vanishing of Artin $L$-functions at $s=1/2$, or a combination of both, we use this criterion to construct Galois extensions where neither of these aspects plays a role. In our constructions, the bias arises entirely from a difference in the number of $2p$-th roots for an odd prime $p$. We prove that the minimal Galois group order required for this phenomenon is $96$ for $p=3$ and $320$ for $p=5$, where in both cases the group structure is a direct product of two generalized quaternion groups. Furthermore, we provide generalized constructions for all odd primes. Finally, these same algebraic parameters enable us to prove that an estimate established by Aoki and Koyama~\cite{Ao-Ko} under the Deep Riemann Hypothesis holds unconditionally in certain cases.
\end{abstract}


\maketitle

\tableofcontents

\section{Introduction}
\subsection{Background and main results}
In a letter~\cite{Chebyshev} sent to Fuss in 1853, Chebyshev observed that primes congruent to $3\pmod{4}$ seem to appear more often than those congruent to $1\pmod{4}$. Let $q\ge3$ and let $a\in \Z$ be coprime to $q$, and define $\pi(x;q,a)=\#\{ p\le x\, \colon\,   p\equiv a\pmod{q}\}$. 
In 1914, Littlewood~\cite{Littlewood} proved that the function $x\mapsto\pi(x;4,1)-\pi(x;4,3)$ changes sign infinitely often, which naturally raised the question of the existence of the natural density of the set $\mathcal{P}_q(a,b):=\{x\ge 2\, \colon\,  \pi(x;q,a)>\pi(x;q,b)\}$ when both $a$ and $b$ are coprime to $q$. This question was addressed by Kaczorowski~\cite{Kac95}, who proved under the Generalized Riemann Hypothesis ($\mathrm{GRH}$) that the natural density of $\mathcal{P}_4(3,1)$ does not exist. 

In 1994, Rubinstein and Sarnak~\cite{RS94} proved, under $\mathrm{GRH}$ and a linear independence hypothesis on the non-negative imaginary parts of the zeros of Dirichlet $L$-functions ($\mathrm{LI}$), that the logarithmic density $\delta_q(a,b)$ of the set $\mathcal{P}_q(a,b)$ exists. Here, the logarithmic density of a Borel set $A\subset\R$ is given by 
\[ \lim_{X\to \infty} \frac{1}{\log X}\int_{2}^X \mathds{1}_{A}(u)\frac{\mathrm{d}u}{u}\, .\]
Furthermore, Rubinstein and Sarnak proved, under the same hypotheses, that $\delta_q(a,b)>1/2$ if and only if $b$ has more square roots than $a$ modulo $q$.

This setting was generalized to number fields by Ng~\cite{Ng} in his PhD thesis and studied by several authors \cites{FJ, Dev, Bail, Hay, BaHa}. Let $L/K$ be a Galois extension of number fields with Galois group $G$. For a class function $t:G\to \R$ and $x\ge2$, we define the prime ideal counting function
\begin{equation}
    \pi(x;L/K;t)=\sum_{N\mathfrak{p}\le x} t\bigl(\varphi_\mathfrak{p}\bigr),
\end{equation}
where the sum runs over all non-zero prime ideals $\mathfrak{p}$ of $O_K$. To unify the treatment of ramified and unramified prime ideals, we use the standard convention where $t(\varphi_\mathfrak{p})$ is defined as the average over the inertia group, namely
\[ t(\varphi_\mathfrak{p}) := \frac{1}{|I_{\mathfrak{P}/\mathfrak{p}}|} \sum_{s\in I_{\mathfrak{P}/\mathfrak{p}}}t(\varphi_{\mathfrak{P}} s)\, , \]
where $\mathfrak{P}$ is a prime of $O_L$ lying above $\mathfrak{p}$, $\varphi_{\mathfrak{P}}$ is the corresponding Frobenius element, and $I_{\mathfrak{P}/\mathfrak{p}}$ is the associated inertia group. Let $G^\#$ denote the set of conjugacy classes of the group $G$, for a conjugacy class $C\in G^\#$, we define $\pi(x;L/K;C) := \pi(x;L/K;\mathds{1}_C)$. For conjugacy classes $C_1,C_2\in G^\#$, we use the notation
\begin{equation}\label{tC_1C_2}
    t_{C_1,C_2}:=\frac{|G|}{|C_1|}\mathds{1}_{C_1}-\frac{|G|}{|C_2|}\mathds{1}_{C_2}\, .
\end{equation}
The function $\pi(x;L/K;t_{C_1,C_2})$ thus represents the normalized difference between the prime ideal counting functions associated with $C_1$ and $C_2$. Define $\mathcal{P}_{L/K}(t):=\{x\ge 2\, \colon\,  \pi(x;L/K;t)>0\}$, and denote its logarithmic density, when it exists, by $\delta_{L/K}(t)$. We define $\delta_{L/K}(C_1,C_2):=\delta_{L/K}(t_{C_1,C_2})$. Following Rubinstein and Sarnak, we say the prime ideal race related to $C_1$ and $C_2$ is unbiased if $\delta_{L/K}(C_1,C_2)=1/2$ and we say that it is biased towards $C_1$ if $\delta_{L/K}(C_1,C_2)>1/2$. 

In \cites{Ng, FJ, Bail}, the authors provide explicit examples of extensions $L/K$ of number fields, where $L$ is Galois over $\Q$, such that $\delta_{L/K}(C_1,C_2)\in (1/2, 1)$. In these cases, the bias arises from a difference in the number of square roots of elements in $C_1$ and $C_2$, as well as from the contribution of the zeros at $s=1/2$ of $L(s; L/\Q; \chi)$ for irreducible characters $\chi$ of $\Gal(L/\Q)$.

The goal of this paper is to construct examples where $\delta_{L/K}(C_1,C_2)\in (1/2,1)$, yet neither of these factors contributes to the bias. To achieve this, we proceed in two stages. First, we introduce two algebraic parameters associated with class functions (Definition~\ref{iota-gamma}), which allow us to establish a conditional characterization of this bias (Theorem~\ref{characterization}). This characterization enables us to isolate the precise cases where the bias is entirely determined by the behavior of higher-degree roots (Corollaries~\ref{maincriterion} and~\ref{E_d-criterion}). Second, and this is the central contribution of the paper, we unconditionally construct Galois extensions satisfying these conditions (Theorem~\ref{np}). As a consequence of our results, we show that the minimal Galois group order of an extension exhibiting this type of bias is $96$. Finally, we use these algebraic parameters to unconditionally establish explicit examples of a recent asymptotic estimate by Aoki and Koyama~\cite{Ao-Ko}, which was originally proved under assumptions stronger than $\mathrm{GRH}$ (Theorems~\ref{criterion-extreme-bias} and~\ref{existence-extreme-bias}).

Before moving to the statement of our main results, let us first introduce some necessary notation. When $N$ is a finite group and $g\, \colon\,  \Gamma\to \R$ is a class function on a subgroup $\Gamma$ of $N$, we define the induced class function by $g$ on $N$ at $a\in N$ as
\[ \bigl(\mathrm{Ind}_\Gamma^Ng\bigr)(a):=\frac{1}{\bigl|\Gamma\bigr|}\sum_{\substack{b\in N\\ b^{-1}ab\in \Gamma}} g\bigl(b^{-1}ab\bigr)\, . \]
For $\ell\ge 1$, let $f_\ell^\Gamma\, \colon\, \Gamma\to \Gamma$ be the $\ell$-th power map given by $f_\ell^\Gamma(x):=x^\ell$ for $x\in \Gamma$. We denote by $r_\ell^\Gamma(y):=\#\bigl\{x\in \Gamma\, \colon\, x^\ell=y\bigr\} $ the number of $\ell$-th roots of the element $y$ in $\Gamma$. For a conjugacy class $C\in \Gamma^\#$, we simply denote by $r_\ell^\Gamma(C)$ the common value $r_\ell^\Gamma(c)$ for all $c\in C$. We denote by $\langle\,\cdot\,,\,\cdot\, \rangle_\Gamma$ the usual scalar product on $\Gamma$. When the underlying group is clear from the context, we will simply write $f_\ell$, $r_\ell$, and $\langle\,\cdot\,,\,\cdot\, \rangle$. Let $\mathcal{S}$ be the set of positive squarefree integers.
To establish a precise algebraic characterization of the biased regime, we introduce two parameters associated with class functions. 
\begin{defi}\label{iota-gamma}
    Let $\Gamma $ be a subgroup of a finite group $N$, and let $g\, \colon\,  \Gamma\to \R$ be a class function of $\Gamma$. We define \begin{align*}\iota^{N}(g)&:=\min\left\{ \ell\in \mathcal{S}\, \colon\,  \mathrm{Ind}_\Gamma^N (g\circ f_\ell^\Gamma)\ne 0\right\}\in \mathcal{S}\cup \{\infty\}\, ,\\
    \gamma(g)&:=\min \{\ell\in \mathcal{S}\, \colon\,  \left\langle g,r_\ell^\Gamma\right\rangle_\Gamma\,\ne 0\}\in \mathcal{S}\cup \{\infty\}\, ,
    \end{align*}
    where we use the convention $\min(\emptyset)=\infty$.
\end{defi}
Let $\overline{L}$ be the Galois closure of $L$ over $\Q$, denote $\overline{G}:=\Gal\bigl(\overline{L}/K\bigr)$ and $G^+:=\Gal\bigl(\overline{L}/\Q\bigr)$, and let $\pi\, \colon\,  \overline{G}\to G $ be the natural projection. If $g\, \colon\,  \overline{G}\to \C$ is a class function, we denote its induced class function on $G^+$ as $g^+:=\mathrm{Ind}_{\overline{G}}^{G^+} g$. Finally, for a class function $t\, \colon\,  G\to \R$, let us denote $\iota(t):=\iota^{G^+}(t\circ \pi)$.
The following conditional theorem characterizes the values of the density $\delta_{L/K}(t)$ in terms of the parameters $\iota(t)$ and $\gamma(t)$.

\begin{thm}\label{characterization}
    Let $L/K$ be a Galois extension of number fields with Galois group $G$, and let $t : G \to \R$ be a class function satisfying $\langle t,1\rangle_G=0$. Assume that $\mathrm{GRH}(L)$, $\mathrm{AC}(L)$, and $\mathrm{LI}^-(L)$ hold (see \S~\ref{Hypotheses} for the statement of these hypotheses). Then $\delta_{L/K}(t)\in (0,1)$ if and only if $\iota(t)<\infty$ and $\gamma(t)\ge 2\iota(t)$. Furthermore, $\delta_{L/K}(t)\in (1/2,1)$ if and only if $\iota(t)<\infty$, $2\iota:=2\iota(t)\le \gamma(t)$, and
    \[ \mu(2\iota)\bigl\langle t, r_{2\iota}^G \bigr\rangle_G-\mu(\iota)\sum_{\chi\ne1} \bigl\langle \bigl(t \circ f_\iota^G\circ \pi\bigr)^+, \chi \bigr\rangle_{G^+} \mathrm{ord}_{s=1/2}L(s;\overline{L}/\Q;\chi) <0\, ,\]
    where the sum runs over non-trivial irreducible characters of $G^+$, and $\mu$ denotes the Möbius function.
\end{thm}
Theorem~\ref{characterization} generalizes the classical criterion given by Rubinstein--Sarnak that determines biased races. We note that the assumptions are the standard assumptions for the Rubinstein--Sarnak framework (as in~\cites{Ng, FJ, Bail}). 
Moreover, it characterizes, conditionally, extreme biases given in~\cites{FJ2,Hay}. More precisely, we prove in Proposition~\ref{asymptotics} (following ideas introduced in~\cites{FJ2,Hay}) that if $\gamma(t)<2\iota(t)$ and $\mathrm{GRH}(L)$ holds, then $\pi(x;L/K;t)\sim K_t x^{1/\gamma}/(\log x)$ as $x\to \infty$, where $K_t\ne 0$. Therefore, $\pi(x;L/K;t)$ has a constant sign for $x$ sufficiently large. Combining this with Theorem~\ref{characterization}, we obtain a conditional equivalence: the difference 
\[ \frac{|G|}{|C_1|}\pi(x;L/K;C_1)-\frac{|G|}{|C_2|}\pi(x;L/K;C_2)\] 
changes sign infinitely often if and only if $\delta_{L/K}(C_1,C_2)\in (0,1)$. This equivalence addresses a question posed to me by Lamzouri and Fiorilli.

We note that in the case where the Artin $L$-functions do not vanish at $s=1/2$, the inequality \[ \mu(2\iota)\bigl\langle t, r_{2\iota}^G \bigr\rangle_G-\mu(\iota)\sum_{\chi\ne1} \bigl\langle \bigl(t \circ f_\iota^G\bigr)^+, \chi \bigr\rangle_{G^+}\, \mathrm{ord}_{s=1/2}L(s;L/\Q;\chi) <0\] reduces to $\mu(2\iota)\bigl\langle t,r_{2\iota}^G\bigr\rangle_G<0$, which implies that $\gamma(t)\le 2\iota(t)$. This leads to the following corollary:

\begin{cor}\label{maincriterion}
    Let $L/\Q$ be a Galois extension with Galois group $G^+$ such that Artin $L$-functions $L(s;L/\Q;\chi)$, for all irreducible characters $\chi$ of $G^+$, do not vanish at $s=1/2$. Let $K$ be a subextension of $L/\Q$, let us denote by $G$ the Galois group of $L/K$, and let $C_1,C_2\in G^\#$. Assume $\mathrm{GRH}(L)$, $\mathrm{AC}(L)$, and $\mathrm{LI}^{-}(L)$. We have $\delta_{L/K}(C_1,C_2)\in (1/2,1) $ if and only if 
    \begin{equation}\label{eq-crit} 2\iota:=2\iota\left(t_{C_1,C_2}\right)=\gamma\left(t_{C_1,C_2}\right)\quad\text{and}\quad \mu(2\iota)\left(r_{2\iota}^G(C_1)-r_{2\iota}^G(C_2)\right)<0\, .\end{equation}
\end{cor}

Corollary~\ref{maincriterion} provides a criterion for identifying prime ideal races where the bias is determined by roots of higher degree, with neither square roots nor the order of vanishing of $L$-functions at $s=1/2$ having an influence. Our goal is to provide explicit examples of this type of bias.

Let $d\ge 3$ be an odd squarefree integer. Let $C_1,C_2$ be conjugacy classes in $G=\Gal(L/K)$, and let $(x,y)\in C_1\times C_2$. We say that the group $G=\Gal(L/K)$ satisfies the property $\mathcal{P}(2d;x,y)$ if $\gamma(t_{C_1,C_2})=2d$ (see Definition~\ref{propertyP(d)}). Let us note that the criterion given by Corollary~\ref{maincriterion} implies that $\Gal(L/K)$ satisfies $\mathcal{P}(2\iota;x,y)$ where $\iota$ is an odd squarefree integer. For a group to satisfy the property $\mathcal{P}(2d;x,y)$, the root-counting function $r_n$ must satisfy $r_\ell(x)=r_\ell(y)$ for all squarefree $\ell < 2d$, and in particular for $\ell=2$ and $\ell=d$, but $r_{2d}(x)\ne r_{2d}(y)$. In particular, $r_n$ must be non-multiplicative. This non-multiplicativity forces the group $G$ to be non-nilpotent. With such hypotheses, the order of such groups is relatively large even for small $d$: we will prove in Theorem~\ref{n3andn5} that if a group satisfies this property for $d=3$, then $|G|\ge 96$, and if a group satisfies this property for $d=5$, then $|G|\ge 320$, which makes finding such examples difficult. 

It is conjectured that if $\Gal(L/\Q)$ admits no irreducible symplectic representations, the associated Artin $L$-functions do not vanish at $s=1/2$ (see for instance~\cite{FJ}*{\S 1.2 Conjecture (LI)}).
Let $d\ge3 $ be an odd squarefree integer, and let $\mathcal{E}_d$ denote the set of Galois extensions $L/K$ with group $G$, where $L/\Q$ is a Galois extension with group $G^+$ admitting no irreducible symplectic representations, and where $G$ contains conjugacy classes $C_1,C_2$ satisfying $2\iota:=2\iota(t_{C_1,C_2})=\gamma(t_{C_1,C_2})$ and $\mu(2\iota)\langle t_{C_1,C_2}, r_{2\iota}\rangle_G<0$. As a consequence of Corollary~\ref{maincriterion}, under the stronger assumption $\mathrm{LI}(L)$ (see \S~\ref{Hypotheses}), we have:
\begin{cor}\label{E_d-criterion}
    Let $L/K\in \mathcal{E}_d$, and let $C_1,C_2\in G^\#$ be as given by the definition of $\mathcal{E}_d$. Assume that $\mathrm{GRH}(L)$, $\mathrm{AC}(L)$, and $\mathrm{LI}(L)$ hold. Then, for all squarefree $\ell<2d$, we have $r_\ell^G(C_1)=r_\ell^G(C_2)$, and for all irreducible characters $\chi$ of $\Gal(L/\Q)$, we have $\mathrm{ord}_{s=1/2}L(s;L/\Q;\chi)=0$, and $\delta_{L/K}(C_1,C_2)\in (1/2, 1)$.
    \end{cor}
Let $d\ge 3$ be an odd squarefree integer. Proving that $\mathcal{E}_d$ is non-empty is a non-trivial problem due to group-theoretic, representation-theoretic, and arithmetic difficulties. First, one must construct a group $G$ satisfying the property $\mathcal{P}(2d;x,y)$, which as explained above is itself non-trivial.
Second, on the arithmetic side, we need to construct a group $G^+$ for which a solution to the inverse Galois problem over $\Q$ exists and that has no irreducible symplectic representations. Finally, we require that $G^+$ satisfies the representation-theoretic constraint $\gamma(t_{C_1,C_2})=2\iota(t_{C_1,C_2})=2d$, which requires that for $\ell<d$, every $\ell$-th root of an element in $C_1$ is conjugate in $G^+$ to an $\ell$-th root of an element in $C_2$ and vice versa, and that there exists a $d$-th root of some element in $C_1$ that is not conjugate in $G^+$ to any $d$-th root of an element of $C_2$. 
A natural way to solve the arithmetic difficulty is to choose $G^+=\mathfrak{S}_n$ for some $n$ sufficiently large, which is realizable as a Galois group over $\Q$ and has no irreducible symplectic representation. However, we note that embedding the group $G$ into the symmetric group in a trivial way (the Cayley embedding, for example) would fail here; indeed, we show in Proposition~\ref{cayley} that in the Cayley embedding case we have $\gamma(t_{C_1,C_2})=\iota(t_{C_1,C_2})$, meaning that any $d$-th root of an element of $C_1$ is conjugate in $\mathfrak{S}_n$ to some $d$-th root of an element of $C_2$. Furthermore, choosing a random embedding where we force some $d$-th root of some element in $C_1$ to have a different cycle type compared to any cycle type of $d$-th roots of elements in $C_2$ will generally lead to some $\ell$-th roots ($\ell<d$) having a different cycle type. This means that such an embedding should be chosen carefully. 

Define \(n_d:=\min \{[L:K]\, \colon\,  L/K\in \mathcal{E}_d\}\in \N\cup\{\infty\} \). In the case where $d=p$ is an odd prime number, we are able to prove that $\mathcal{E}_p\ne \emptyset$, more precisely: 
\begin{thm}\label{np}
    Let $p\ge 3$ be a prime and let $q$ be a prime satisfying $q\equiv 1\pmod{2p(p-1)}$. There exists an extension $L/K\in \mathcal{E}_p$ such that $L/\Q$ has Galois group isomorphic to $\mathfrak{S}_{q^4}$ and 
    \[ \Gal(L/K)\cong Q_{4p} \times Q_{4(p-1)} \times C_{\mathrm{rad}(p-1)/2} \, ,\]
    where, for $n\ge2$, $Q_{4n}$ is the generalized quaternion group of order $4n$ and, for $n\ge 1$, $C_n$ is the cyclic group of order $n$, and $\mathrm{rad}(n)$ is the radical of $n\ge1$. 
    Moreover, we have the inequalities $8p^2\le n_p\le 8p(p-1)\mathrm{rad}(p-1)$, and in the special cases $p\in\{3,5\}$, we have $n_3=96$ and $n_5=320$.
\end{thm}
In fact, we prove a stronger result. As noted following Corollary~\ref{maincriterion}, if an extension $L/K$ satisfies the criterion of that corollary with $2\iota=2p$, then its Galois group must satisfy the property $\mathcal{P}(2p;x,y)$ for some $(x,y)\in C_1\times C_2$. We establish that any group $G$ satisfying $\mathcal{P}(2p;x,y)$ must have order $|G|\ge 8p^2$; more specifically, for $p=3$ and $p=5$, we obtain the optimal lower bounds $|G|\ge 96$ and $|G|\ge 320$, respectively. In these specific cases, our constructed group $Q_{4p}\times Q_{4(p-1)}\times C_{\mathrm{rad}(p-1)/2}$ has order $16p(p-1)$, which exactly matches $96$ for $p=3$ and $320$ for $p=5$. Consequently, for $p\in \{3,5\}$, Theorem~\ref{np} provides examples of extensions satisfying~\eqref{eq-crit} with strictly minimal degree $[L:K]$. In particular, this confirms that the extensions we construct in $\mathcal{E}_p$ achieve the absolute minimum $n_p=[L:K]$. 

A natural question is whether the order $8p(p-1)\mathrm{rad}(p-1)$ is always minimal. The answer to this question is negative. Indeed, if $p$ is a Sophie Germain prime, denoting $p':=2p+1$, we prove in Corollary~\ref{smallerembedding} that there exists $L/K\in \mathcal{E}_p$ such that $\Gal(L/K)\cong Q_{4p'}\times Q_{4(p'-1)}$. In particular, for $p=11$, there exists an extension $L/K\in \mathcal{E}_{11}$, where \[[L:K]=8096<8800=8p(p-1)\mathrm{rad}(p-1)\, .\]
However, we expect that $n_{11}=8096$. In \S~\ref{sec:open_questions}, we provide a more general expected value of $n_p$ for all primes $p$. 

\subsection{Weighted prime ideal races}\label{Weightedfunctions}

In order to study prime ideal races with different weights, for a real number $\alpha \ge 0$, a class function $t:G\to \R$, and $x\ge2$, we define the $\alpha$-weighted prime ideal counting function
\begin{equation}
    \pi_\alpha(x;L/K;t)=\sum_{N\mathfrak{p}\le x}\frac{t\bigl(\varphi_\mathfrak{p}\bigr)}{(N\mathfrak{p})^\alpha}\, .
\end{equation}
Evaluating at $\alpha=0$ naturally recovers the unweighted counting function studied in the previous section. Recently, Aoki and Koyama~\cite{Ao-Ko} studied Chebyshev's bias by focusing specifically on the weight $\alpha=1/2$. Under the Deep Riemann Hypothesis ($\mathrm{DRH}$), a hypothesis concerning the convergence of Euler products of $L$-functions that is stronger than $\mathrm{GRH}$, they proved (see~\cite{Ao-Ko}*{Theorem 1.1}) that there exists a constant $K\in \R$ such that
\begin{equation}\label{DRHequality}
    \frac{|G|}{|C_1|}\pi_{1/2}(x;L/K;C_1)-\frac{|G|}{|C_2|}\pi_{1/2}(x;L/K;C_2)= K \log \log x+O(1)\, .
\end{equation}
They assumed $\mathrm{DRH}$ to bypass the need for the Linear Independence hypothesis. Based on this asymptotic behavior, Aoki and Koyama defined (see~\cite{Ao-Ko}*{Definition 1.2}) the prime ideal race to be biased towards $C_1$ if $K>0$, and unbiased if $K=0$. More recently, Sheth~\cite{Sheth} weakened this assumption, proving that under $\mathrm{GRH}(L)$, the equality \eqref{DRHequality} holds outside an exceptional set of finite logarithmic measure. 

In this paper, we use the algebraic parameters $\gamma(t)$ and $\iota(t)$ to produce examples in which the equality~\eqref{DRHequality} holds unconditionally. The following theorem gives a precise criterion for the asymptotic behavior of $\pi_\alpha(x;L/K;t)$ under $\mathrm{GRH}(L)$, and specifies the exact conditions under which this behavior becomes unconditional.

\begin{thm}\label{criterion-extreme-bias}
    Let $L/K$ be a Galois extension of number fields with Galois group $G$ and let $t\, \colon\,  G\to \R$ be a non-zero class function such that $\iota(t)<\infty$ and $\gamma:=\gamma(t)<2\iota(t)$. If $\mathrm{GRH}(L)$ holds, then 
    \begin{equation}\label{pialpha-asympt}
        \pi_\alpha(x;L/K;t)=\begin{cases}
            K_{\alpha,t}\frac{x^{(1-\alpha \gamma)/\gamma}}{\log x}+O\left(x^{(1-\alpha \gamma)/\gamma}(\log x)^{-2}\right) &\text{if}\ \  \alpha<1/\gamma\\
           K'_t \log \log x+O(1) &\text{if}\ \alpha=1/\gamma\\
            O(1)  &\text{if}\ \ \alpha > 1/\gamma
            \end{cases}\, ,
    \end{equation}
    where $K_{\alpha,t}$ and  $K'_t$ are non-zero constants having the same sign as $\mu(\gamma)\langle t,r_\gamma\rangle_G$. If, moreover, $\gamma(t)=\iota(t)$, then~\eqref{pialpha-asympt} holds unconditionally.
\end{thm}

Following ideas introduced in~\cite{Hay}, we can construct unconditional examples of such extreme biases. This leads to the following theorem:

\begin{thm}\label{existence-extreme-bias}
    Let $p\ge2$ be a prime. There exist infinitely many Galois extensions of number fields $L/K$ with Galois group $G$ and conjugacy classes $C_1,C_2\in G^\#$, such that for all $x\ge2$ we have
    \[ \frac{1}{|C_1|}\pi_{\alpha}(x;L/K;C_1)-\frac{1}{|C_2|}\pi_{\alpha}(x;L/K;C_2)=\begin{cases}
            K_{\alpha,C_1,C_2}\frac{x^{(1-\alpha p)/p}}{\log x}+O\left(\frac{x^{(1-\alpha p)/p}}{(\log x)^{2}}\right) &\text{if}\ \  \alpha<1/p\\
           K'_{C_1,C_2} \log \log x+O(1) &\text{if}\ \alpha=1/p\\
            O(1)  &\text{if}\ \ \alpha > 1/p
            \end{cases}\, , \]
    where $K_{\alpha,C_1,C_2}, K'_{C_1,C_2}$ are positive constants.
\end{thm}

In particular, taking $p=2$ in Theorem~\ref{existence-extreme-bias}, we deduce that we can unconditionally produce examples where~\eqref{DRHequality} holds. Let us also note that taking $p$ to be an odd prime yields examples of extensions where there are conjugacy classes $C_1$ and $C_2$ for which \[ \frac{1}{|C_1|}\pi_{1/2}(x;L/K;C_1)-\frac{1}{|C_2|}\pi_{1/2}(x;L/K;C_2)=O(1)\]
yet we have an extreme Chebyshev's bias towards $C_1$, in the sense that for all $x$ sufficiently large we have \[ \frac{1}{|C_1|}\pi(x;L/K;C_1)>\frac{1}{|C_2|}\pi(x;L/K;C_2)\, .\]
This shows that there are unbiased races in the sense of Aoki--Koyama~\cite{Ao-Ko}*{Definition 1.2}, that are biased in the sense of Rubinstein--Sarnak for which we have $\delta_{L/K}(C_1,C_2)=1$. A natural question is whether we can still produce unbiased races in the sense of Aoki--Koyama while $\delta_{L/K}(C_1,C_2)\in (1/2,1)$.
In order to give an answer to this question, we will prove the following proposition: 
\begin{prop}~\label{Aoki-Koyamaunbiasedness}
    Let $p\ge 3$ be a prime number, and let $L/K\in \mathcal{E}_p$ and let $C_1,C_2\in G^\#$ be as given by the definition of $\mathcal{E}_p$. If $\mathrm{GRH}(L)$ holds, then for all $\alpha>1/(2p)$ we have $\pi_\alpha(x;L/K;t_{C_1,C_2})=O(1)$.
\end{prop}
In particular, since by Theorem~\ref{np} $\mathcal{E}_p\ne\emptyset$, if $L/K\in \mathcal{E}_p$ and $C_1,C_2$ are as in Proposition~\ref{Aoki-Koyamaunbiasedness}, then the corresponding race is unbiased in the sense of Aoki--Koyama. Using Corollary~\ref{E_d-criterion}, we have, conditionally to $\mathrm{GRH}(L)$, $\mathrm{LI}(L)$, and $\mathrm{AC}(L)$; that $\delta_{L/K}(C_1,C_2)\in (1/2,1)$. This provides a conditional answer to the previous question. 

\subsection{Statement of the assumptions and structure of the paper}\label{Hypotheses}

In this paper, we will follow the standard assumptions for the study of Chebyshev's bias in general number fields (see~\cite{FJ}*{\S 1.2} and~\cite{Bail}*{\S 1.3}). Let $L$ be a number field that is Galois over $\Q$, with Galois group $G^+ := \Gal(L/\Q)$. For a finite group $G$, we denote by $\mathrm{Irr}(G)$ the set of irreducible characters of $G$. Recall that for an irreducible character $\chi \in \mathrm{Irr}(G)$, the Frobenius--Schur indicator is given by
\[ \varepsilon_2(\chi) := \frac{1}{|G|}\sum_{g\in G} \chi(g^2)\, . \]
It is a classical result (see, for instance, \cite{FJ}*{Theorem 3.3}) that $\varepsilon_2(\chi) \in \{1, -1, 0\}$. Following standard terminology, we say that the character $\chi$ is orthogonal, symplectic, or unitary depending on whether $\varepsilon_2(\chi)$ takes the value $1$, $-1$, or $0$, respectively. 
We state the following hypotheses on the extension $L/\Q$:

\begin{itemize}
    \item[$\mathrm{AC}(L)$:] We assume Artin's holomorphicity conjecture, which states that for every nontrivial $\chi \in \mathrm{Irr}(G^+)$, the associated Artin $L$-function $L(s; L/\Q; \chi)$ is entire.
    
    \item[$\mathrm{GRH}(L)$:] We assume the Riemann Hypothesis for the extension $L/\Q$, that is, every nontrivial zero of $L(s; L/\Q; \chi)$ lies on the line $\Re(s) = 1/2$, for every $\chi \in \mathrm{Irr}(G^+)$.
    
    \item[$\mathrm{LI}^-(L)$:] We assume that the multiset of positive imaginary parts of the zeros of all Artin $L$-functions $L(s; L/\Q; \chi)$ in the region $\{s \in \C : \Re(s) \ge 1/2\}$, with $\chi \in \mathrm{Irr}(G^+)$, are linearly independent over the rationals.
    
    \item[$\mathrm{LI}(L)$:] We assume $\mathrm{LI}^-(L)$. Moreover, we assume that $L(1/2, L/\Q, \chi) \neq 0$ if $\chi$ is an orthogonal or unitary irreducible character of $G^+$. (see~\cite{FJ}*{\S 1.2} for a more general assumption $\mathrm{LI}$).
\end{itemize}
When the extension $L/\Q$ is not Galois, we use the notation $\mathrm{GRH}(L)=\mathrm{GRH}\bigl(\overline{L}\bigr)$, $\mathrm{AC}(L)=\mathrm{AC}\bigl(\overline{L}\bigr)$, $\mathrm{LI}^-(L)=\mathrm{LI}^-\bigl(\overline{L}\bigr)$, and $\mathrm{LI}(L)=\mathrm{LI}\bigl(\overline{L}\bigr)$, where $\overline{L}$ is the Galois closure of $L$ over $\Q$.

The remainder of the paper is organized as follows. In \S~\ref{Induction}, we prove Theorem~\ref{characterization}, Theorem~\ref{criterion-extreme-bias}, and Proposition~\ref{Aoki-Koyamaunbiasedness}. Following ideas from \cites{Hay, FJ2}, we use an inclusion-exclusion principle to relate Chebyshev-type functions in number fields (see \eqref{theta-psi} and \eqref{inc-exc}) together with the induction property of Artin $L$-functions. This leads to Proposition~\ref{asymptotics}, which gives the main asymptotics of $\pi(x;L/K;t)$ when $\gamma(t)<2\iota(t)$. Using summation by parts, we then deduce Theorem~\ref{criterion-extreme-bias}. For the regime $\gamma(t)\ge 2\iota(t)$, we combine these ideas with the standard limiting distribution machinery of Rubinstein and Sarnak to establish Theorem~\ref{characterization}.

In \S~\ref{constructions}, we reduce the existence parts of Theorems~\ref{np} and \ref{existence-extreme-bias} to group-theoretic and representation-theoretic problems. Theorem~\ref{existence-extreme-bias} follows as a consequence of Theorem~\ref{criterion-extreme-bias} combined with group-theoretic constructions from~\cite{Hay}. To prove that $\mathcal{E}_p$ is non-empty for an odd prime $p$, we introduce the representation-theoretic property $\mathcal{Q}(2p;x,y)$ (see Definition~\ref{propertyQ}). If a pair of finite groups $(G,G^+)$ satisfies $\mathcal{Q}(2p;x,y)$ and $G^+$ has no symplectic irreducible representations, then any Galois extension $L/K$ with $\Gal(L/\Q) \cong G^+$ and $\Gal(L/K) \cong G$ satisfies $L/K\in \mathcal{E}_p$. Constructing the appropriate group $G$ requires that it satisfies the property $\mathcal{P}(2p;x,y)$ (see Definition~\ref{propertyP(d)}). Using generalized quaternion groups, we construct several such examples in \S~\ref{subsec:Prop(2p)}. To build the larger group $G^+$, we embed the generalized quaternion groups into $\GL(2,\F_q)$ for a well-chosen prime $q$, using their standard faithful irreducible representations of degree $2$. Applying this to the groups in Theorem~\ref{np} yields an embedding of $G$ into $\GL(4,\F_q)$. By studying the eigenvalues of the matrices corresponding to the roots of $x$ and $y$, we show that for $\ell<p$, the $\ell$-th roots of $x$ and $y$ are conjugate in $\GL(4,\F_q)$. Then, by considering the natural faithful action of $\GL(4,\F_q)$ on $\F_q^4$, we obtain an embedding of $G$ into the symmetric group $\mathfrak{S}_{q^4}$. Finally, we prove that under this embedding, the $p$-th roots of $x$ and $y$ cannot have the same cycle type. This completes the construction of $G$ and $G^+$, thereby proving that $\mathcal{E}_p \ne \emptyset$.

In \S~\ref{lbounds}, we establish lower bounds for $n_p$: we prove that a group $G$ satisfying the property $\mathcal{P}(2p;x,y)$ has order $\ge 8p^2$. This result is technical; we start by proving some counting properties of the root functions $r_n$, $n\ge1$. Then, we reduce the problem to the case where the group $G$ has a normal $p$-Sylow $P=\langle t\rangle$ of order $p$. Proceeding by contradiction, and assuming that $|G|<8p^2$, enables us to use the Schur--Zassenhaus theorem which implies that $G\cong P \rtimes H$ for some subgroup $H$ of $G$. We prove that certain quotients of the subgroup $H\cap \mathcal{C}_G(t)$ contain a high proportion of involutions. These groups were studied by Wall~\cite{Wall}; his classification theorem plays a major role for obtaining a contradiction. Finally, by exploiting divisibility properties of the groups satisfying $\mathcal{P}(2p;x,y)$ and eliminating specific orders for the case $p=5$, we deduce that $n_3\ge 96$ and $n_5\ge 320$.

\subsection*{Acknowledgements} I would like to thank the Max Planck Institute for Mathematics in Bonn for excellent working conditions and support. I thank Zakariae Bouazzaoui for several discussions that were essential to the development of the paper, as well as for his remarks. I am also grateful to Florent Jouve and Daniel Fiorilli for their encouragement and remarks that helped improve the paper, and to Youness Lamzouri for his question that led to Theorem~\ref{characterization}.

\subsection*{Notation and conventions}
For the convenience of the reader, we collect the group-theoretic definitions and notation used throughout the paper. Let $\Gamma$ and $N$ be finite groups, let $H$ be a subgroup of $\Gamma$, and let $z \in \Gamma$. 

\begin{itemize}
    \item We denote the centralizer of $z$ in $\Gamma$ by $\mathcal{C}_\Gamma(z)$, the normalizer of $H$ in $\Gamma$ by $N_\Gamma(H)$, and the center of $\Gamma$ by $\mathcal{Z}(\Gamma)$.
    \item We let $\Id_\Gamma$ denote the identity map. The inversion map $\mathrm{Inv}_\Gamma \colon \Gamma \to \Gamma$ is given by $\mathrm{Inv}_\Gamma(x) := x^{-1}$, and the inner automorphism induced by $h \in \Gamma$ is denoted by $\mathrm{Int}_h(x) := h x h^{-1}$.
    \item If $\Phi \in \Aut(\Gamma)$ is an automorphism satisfying $\Phi(H) = H$, we let $\Phi_{|H} \in \Aut(H)$ denote the automorphism induced on $H$ by $\Phi$.
    \item For class functions $f, g \colon \Gamma \to \C$, the standard inner product on $\Gamma$ is defined as $\langle f, g \rangle_\Gamma := \frac{1}{|\Gamma|} \sum_{x \in \Gamma} f(x)\overline{g(x)}$.
    \item We let $C_n = \langle c_n \rangle$ denote the cyclic group of $n$-th roots of unity in $\C$, where $c_n := \exp(2\pi i/n)$. Note that if $n \mid m$, we have $C_n \subset C_m$ and $c_m^{m/n} = c_n$.
    \item For a group homomorphism $\phi \colon \Gamma \to \Aut(N)$, the semidirect product $N \rtimes_\phi \Gamma$ is the group defined on the set $N \times \Gamma$ with the law $(n_1, h_1) \cdot (n_2, h_2) := \bigl(n_1 \phi(h_1)(n_2), h_1 h_2\bigr)$ for all $n_1, n_2 \in N$ and $h_1, h_2 \in \Gamma$.
    \item We denote the dihedral group of order $2n$ by $D_{2n}$, the symmetric group on $n$ letters by $\mathfrak{S}_n$, and the generalized quaternion group of order $4n$ by $Q_{4n}$ (the group $Q_{4n}$ is explicitly defined in \S~\ref{subsec:Prop(2p)}).
    \item We define the root sets $\mathcal{R}_n^\Gamma(z) := \{x \in \Gamma \colon x^n = z\}$. Let us note that $r_n^\Gamma(z) = \bigl|\mathcal{R}_n^\Gamma(z)\bigr|$. Furthermore, we define $\kappa_n^\Gamma(z) := \#\{x \in \mathcal{R}_n^\Gamma(z) \colon \Ord(x) = n\Ord(z)\}$. If $\Ord(x) = p^r m$ with $\gcd(p, m) = 1$, we denote $\mathcal{F}_p^\Gamma(x) := \{t \in \mathcal{R}_p^\Gamma(x^m) \colon tx = xt\}$.
    \item For elements $x, y \in \Gamma$ with respective conjugacy classes $C_x$ and $C_y$, we use the shorthand $t_{x,y} := t_{C_x, C_y}$, where $t_{C_x,C_y}$ is defined in \eqref{tC_1C_2}.
\end{itemize}

\section{Analytic criteria for prime ideal races}\label{Induction}
\subsection{Main asymptotics}
Let $L/K$ be a Galois extension of number fields with Galois group $G$. Consider a non-zero class function $t: G \to \R$ such that $\langle t, 1 \rangle_G = 0$. 
Define the analogues of Chebyshev's $\theta$ and $\psi$ functions: 
\begin{equation}\label{theta-psi}
    \theta(x;L/K;t)=\sum_{N\mathfrak{p}\le x}t(\varphi_\mathfrak{p})\log N\mathfrak{p}\quad\text{and}\quad \psi(x;L/K;t)=\sum_{\substack{k,\,\mathfrak{p}\\ N\mathfrak{p}^k\le x}}t(\varphi_\mathfrak{p}^k)\log N\mathfrak{p}\, .
\end{equation}
 By the inclusion-exclusion principle, we have \begin{equation}\label{inc-exc}
    \theta(x;L/K;t)=\sum_{\ell\ge 1}\mu (\ell)\psi\bigl(x^{1/\ell};L/K; t\circ f_\ell^G\bigr)\, .
\end{equation}
Let $\overline{L}$ be the Galois closure of $L$ over $\Q$, denote $\overline{G}:=\Gal\bigl(\overline{L}/K\bigr)$ and $G^+:=\Gal\bigl(\overline{L}/\Q\bigr)$, and let $\pi\, \colon\,  \overline{G}\to G $ be the natural projection. 
The following lemma shows that proving Theorem~\ref{characterization} can be reduced to the case where $L/\Q$ is Galois.
\begin{lem}\label{GaloiscaseReduction}
    We have for all $x\ge 2$
    \[\pi(x;L/K;t)=\pi(x;\overline{L}/K;t\circ \pi )\quad \text{and}\quad \langle t\circ \pi, 1\rangle_{\overline{G}}=0\, .\]
\end{lem}
\begin{proof}
   We have \[ \langle t\circ \pi, 1\rangle_{\overline{G}}=\langle t, 1\rangle_G=0\, .\]
    By the inflation property of Artin $L$-functions, we have for all $x\ge 2$
    \begin{align*}
        \psi(x;L/K;t)&=-\frac{1}{2i\pi}\int_{\Re(s)=2}\frac{L'}{L}(s;L/K;t)\frac{x^s}{s}\mathrm{d}s\\
        &=-\frac{1}{2i\pi}\int_{\Re(s)=2}\frac{L'}{L}(s;\overline{L}/K;t\circ \pi)\frac{x^s}{s}\mathrm{d}s=\psi(x;\overline{L}/K;t\circ\pi)\, .
    \end{align*}
    Thus, \eqref{inc-exc} shows that $\theta(x;L/K;t)=\theta(x;\overline{L}/K;t\circ \pi)$ (because $\pi\circ f^{\overline{G}}_\ell=f_\ell^G\circ \pi$). The lemma is thus deduced by summing by parts.
\end{proof}
Lemma~\ref{GaloiscaseReduction} shows that the study of $\pi(x;L/K;t)$ can be done by studying the behavior of $\pi\bigl(x;\overline{L}/K;t\circ\pi\bigr)$. Let us also note that the lemma implies that $\mathcal{P}_{L/K}(t)=\mathcal{P}_{\overline{L}/K}(t\circ\pi)$. We may assume, from now on, that $L/\Q$ is Galois with Galois group $G^+$. The induced function of any class function $g\, \colon\,  G\to\C$ in $G^+$ will be denoted as $g^+$.
The following proposition provides the main asymptotics of $\pi(x;L/K;t)$. Its proof is similar to~\cite{Hay}*{Proposition 2.3}.

\begin{prop}\label{asymptotics}
    Let us denote $\gamma:=\gamma(t)$, then the following holds:
    \begin{enumerate}
        \item If $\mathrm{GRH}(L)$ holds for $L/K$ and if $2\iota(t) > \gamma(t)$, then for all $x\ge 2$ \begin{equation}\label{mainasympt}
            \pi(x; L/K; t) = \mu(\gamma)\langle t, r_\gamma \rangle_G \frac{x^{1/\gamma}}{\log x} + O_{L,K,t}\left(\frac{x^{1/\gamma}}{(\log x)^2}\right)\, .
        \end{equation}
        \item If $\iota(t)=\gamma(t)<\infty$, then \eqref{mainasympt} holds unconditionally.
        \item If $\iota(t) = \infty$, then for all $x \ge 2$, we have $\pi(x; L/K; t) = 0$.
    \end{enumerate}
\end{prop}
\begin{proof} 
We begin with the first point. Assume $\mathrm{GRH}(L)$ holds and $2\iota(t) > \gamma(t)$. We have for all $\ell<\gamma$, $\langle t\circ f_\ell,1\rangle_G=\langle t,r_\ell\rangle_G=0$, so that, under $\mathrm{GRH}(L)$, the effective version of Chebotarev's theorem (see \cite{LO}*{Theorem 1.1}) implies that for $\ell<\gamma$ and $x\ge 2$
\[ \psi\left(x^{1/\ell};L/K;t\circ f_\ell\right)\ll_{L,K,t} x^{1/(2\ell)}(\log x)^2\, .\]
Furthermore, for all $x\ge 2$: 
\[\psi\left(x^{1/\gamma};L/K;t\circ f_\gamma\right)=\langle t,r_\gamma\rangle_G x^{1/\gamma}+O_{L,K,t}\left(x^{1/(2\gamma)}(\log x)^2\right)\, .\]
By the invariance of Artin $L$-functions under induction, we have for all squarefree $\ell\ge 1$ and all $x\ge 2$ \[\psi\left(x^{1/\ell};L/K;t\circ f_\ell\right)=\psi\left(x^{1/\ell};L/\Q;(t\circ f_\ell)^+\right)\, .\]
Combining this with~\eqref{inc-exc}, and noting that for $\ell< \iota(t)$, we have $(t\circ f_\ell)^+=0$, we deduce that
\begin{align*}
    \theta(x;L/K;t)&=\sum_{\ell\ge \gamma} \mu(\ell)\psi\left(x^{1/\ell};L/K;t\circ f_\ell\right)+O\left(x^{1/(\gamma+1)}(\log x)^3\right)\\
    &=\mu(\gamma)\psi\left(x^{1/\gamma};L/K;t\circ f_\gamma\right)+\sum_{\substack{\ell\ge \gamma \\ \ell\ne \gamma}} \mu(\ell)\psi\left(x^{1/\ell};L/K;t\circ f_\ell\right)+O\left(x^{1/(\gamma+1)}(\log x)^3\right)\\
    &=\mu(\gamma)\langle t,r_\gamma\rangle_G x^{1/\gamma}+O\left(x^{1/(\gamma+1)}(\log x)^3\right)\, ,
\end{align*}
where we used the fact that for all $\ell>\gamma$ we have $\psi\left(x^{1/\ell};L/K;t\circ f_\ell\right)\ll x^{1/(\gamma+1)}$ and that $2\iota(t)\ge \gamma+1$. Finally, a summation by parts shows that for all $x\ge 2$
\begin{align*}
    \pi(x;L/K;t)&=\frac{\theta(x;L/K;t)}{\log x}+\int_2^x\frac{\theta(u;L/K;t)}{u(\log u)^2}\mathrm{d}u\\
    &=\mu(\gamma)\langle t,r_\gamma\rangle_G \frac{x^{1/\gamma}}{\log x}+ O_{L,K,t}\left(\frac{x^{1/\gamma}}{(\log x)^2}\right)\, ,
\end{align*}
where we used that $\theta(u;L/K;t)\ll_{L,K,t} u^{1/\gamma}$, to deduce that $\int_2^x\frac{\theta(u;L/K;t)}{u(\log u)^2}\mathrm{d}u\ll_{L,K,t} \frac{x^{1/\gamma}}{(\log x)^2} $. This establishes the first point.

Assume now that $\gamma(t)=\iota(t)<\infty$. This matches exactly the setting of \cite{Hay}*{Proposition 2.3~(1)}; indeed, since $\langle t,r_1\rangle_G=\langle t,1\rangle_G=0$, we have $\gamma\ge 2$. Thus, $\gamma$ is a squarefree integer, and for all squarefree $\ell<\gamma=\iota(t)$ we have $(t\circ f_\ell)^+=0$. The main difference with \cite{Hay}*{Proposition 2.3~(1)} is that we need a more precise error term. We use the unconditional effective Chebotarev density theorem (see~\cite{LO}*{Theorem 1.3}), which yields $\psi(y; L/K; t\circ f_\gamma) = \langle t, r_\gamma \rangle_G\, y + O_{L,K,t}\big(y/(\log y)^3\big)$ for $y \ge 2$. Evaluating at $y = x^{1/\gamma}$ and following the exact same steps of inclusion-exclusion and summation by parts as in the conditional case above yields the error term $O_{L,K,t}\big(x^{1/\gamma}/(\log x)^2\big)$ in~\eqref{mainasympt}.

Finally, if $\iota(t) = \infty$, this means that for all squarefree $\ell\ge 1$, we have $(t\circ f_\ell)^+=0$. Applying \cite{Hay}*{Proposition 2.3~(2)} directly yields $\pi(x;L/K;t)=0$, which proves the third point and finishes the proof of the proposition.
\end{proof}

\subsection{Proof of Theorem~\ref{characterization}}
The proof of Theorem~\ref{characterization} combines Proposition~\ref{asymptotics} with the usual Rubinstein--Sarnak setting as in~\cites{RS94,Dev,FJ,Ng}. We will keep the proofs short and focus on the main changes compared to the standard proofs. We use the notation $\langle\cdot\, ,\, \cdot\rangle:=\langle\cdot\, ,\, \cdot\rangle_G$. We start by combining the induction property of $L$-functions with~\eqref{inc-exc} as in the proof of Proposition~\ref{asymptotics}.
\begin{lem}\label{theta_expansion}
    Assume $\mathrm{GRH}(L)$, and that $\iota := \iota(t) < \infty$ with $2\iota \le \gamma(t)$. Then,
    \[ \theta(x; L/K; t) = \mu(\iota) \psi\left(x^{1/\iota}; L/\Q; (t \circ f_\iota)^+\right) + \mu(2\iota)\langle t, r_{2\iota} \rangle x^{1/(2\iota)} + O_{L,K,t}\left(x^{1/(2\iota+1)}\right)\, .\]
\end{lem}
\begin{proof}
    Following the proof of Proposition~\ref{asymptotics}, we see that 
    \[\left|\theta(x;L/K;t)-\mu(\iota)\psi\left(x^{\frac{1}{\iota}}; L/\Q; (t \circ f_\iota)^+\right)-\mu(2\iota)\psi\left(x^{\frac{1}{2\iota}};L/K;t\circ f_{2\iota} \right)\right|\le \sum_{\substack{\ell>\iota\\\ell\ne 2\iota}}\left|\psi\left(x^{\frac{1}{\ell}};L/K;t\circ f_\ell\right)\right|.\]
    Using the effective version of Chebotarev density theorem under $\mathrm{GRH}(L)$ (see~\cite{LO}*{Theorem 1.1}), we see that
    \begin{align*}
        \psi\left(x^{1/(2\iota)};L/K;t\circ f_{2\iota} \right)&=\langle t,r_{2\iota}\rangle x^{1/(2\iota)}+O_{L,K,t}\left(x^{1/(2\iota+1)}\right),\\
        \psi\left(x^{1/\ell};L/K;t\circ f_\ell\right)&\ll_{L,K,t} x^{1/(2\iota+2)}\qquad (\ell<\gamma(t))\, .
    \end{align*}
    The lemma is, thus, deduced by combining this with the general upper bound \[\psi\left(x^{1/\ell};L/K;t\circ f_\ell\right)\ll_t x^{1/\ell}\] for $\ell>2\iota$. (Note that one has to apply the previous upper bound on $\ell>2\iota+1$ then take the sum, then apply it in the case $\ell=2\iota+1$).
\end{proof}
We define the normalized prime counting function for $y \ge 1$ as:
\begin{equation}
    E(y;L/K;t) := \frac{y}{\exp\left(\frac{y}{2\iota(t)}\right)} \pi(e^y; L/K; t)\, .
\end{equation}
The factor $\exp(y / 2\iota(t))$ is the main change in our proof, and it is the correct normalizing factor. Normalizing by the standard $\exp(y/2)$ when $\iota(t) > 1$ would force $\lim_{y \to \infty} E_t(y) = 0$, leading the limiting distribution to be a Dirac measure at $0$. Our normalization is given by the first non-vanishing $\psi$ term in the inclusion-exclusion equality~\eqref{inc-exc}. In order to prove Theorem~\ref{characterization}, we prove that $E(y;L/K;t)$ admits a logarithmic distribution $\mu_t$ which has a rapidly decreasing Fourier transform $\widehat{\mu_t}$. We first start by writing the explicit formula for $E(y;L/K;t)$.
\begin{prop}\label{explicit-formula}
    Assume $\mathrm{GRH}(L)$, $\mathrm{AC}(L)$, and $\mathrm{LI}^-(L)$, and that $\iota:=\iota(t)<\infty$ with $2\iota(t)\le\gamma(t)$. Then,
    \[ E(y;L/K;t) = a_t - \sum_{\chi \ne 1} b_{t, \chi} \sum_{0 < |\gamma_\chi| \le X} \frac{e^{i y \gamma_\chi}}{1/2 + i\gamma_\chi} + O\left( \frac{e^{y/2 \iota} \log(X e^y)^2}{X} + \frac{1}{y} \right)\, , \]
    where the sum over $\chi$ runs over the non-trivial irreducible characters of $G^+$, and the sum over $\gamma_\chi$ runs over non-zero imaginary parts of zeros of the Artin $L$-function $L(s;L/\Q;\chi)$. The coefficients are $b_{t,\chi} = \mu(\iota)\langle (t \circ f_\iota)^+, \chi \rangle_{G^+}$, and $a_t = \mu(2\iota)\langle t, r_{2\iota} \rangle_G-\sum_{\chi\ne1} b_{t,\chi}\mathrm{ord}_{s=1/2}L(s;L/\Q;\chi)$.
\end{prop}
\begin{proof}
    This follows the same ideas as in~\cite{FJ}*{Corollary 3.17}; writing 
    \[ (t\circ f_\iota)^+=\sum_{\chi\ne 1} \bigl\langle (t\circ f_\iota)^+,\chi\bigr\rangle_{G^+}\chi\, ,\]
    where we used Frobenius reciprocity $ \bigl\langle (t\circ f_\iota)^+,1\bigr\rangle_{G^+}= \bigl\langle t\circ f_\iota,1\bigr\rangle_G=\langle t,r_\iota\rangle_G=0$, then using the explicit formula~\cite{FJ}*{Lemma 3.16} (see also~\cite{Ng}*{(5.8)}), leads to an explicit formula of the term $\psi\left(x^{1/\iota}; L/\Q; (t \circ f_\iota)^+\right)$ in terms over the sums over the zeros of Artin $L$-functions $L(s;L/\Q;\chi)$, with $\chi\in \Irr(G^+)$. The rest is standard, it suffices to use Lemma~\ref{theta_expansion}, then a summation by parts to deduce our lemma (we refer the reader to~\cite{Ng}*{Section 5.1} and~\cite{Dev}*{Subsection 4.3}).
\end{proof}

\begin{proof}[Proof of Theorem~\ref{characterization}]
    By Proposition~\ref{asymptotics}, if $\iota(t)=\infty$, then $\pi(x;L/K;t)=0$ for all $x\ge 2$. Furthermore, if $2\iota(t)>\gamma(t)$, there exists $A>0$ such that $\pi(x;L/K;t)$ is either strictly positive for all $x\ge A$ or strictly negative for all $x\ge A$. In this case $\delta_{L/K}(t)\in \{0,1\}$. Assuming now that $\iota(t)<\infty$ and $2\iota(t)\le \gamma(t)$, we will show that $\delta_{L/K}(t)\in (0,1)$. We first establish that $E(y;L/K;t)$ admits a limiting distribution $\mu_t$. This is mainly a consequence of Lemma~\ref{explicit-formula} (note that the coefficients $b_{t,\chi}$ are not all zero since $(t\circ f_\iota)^+\not \equiv 0$). The existence of $\mu_t$ can be deduced following~\cite{FJ}*{Proposition 3.18} (see also~\cite{Ng}*{Section 5.1} and~\cite{Dev}*{Theorem 2.1}) using the Kronecker--Weyl equidistribution theorem and Lévy's continuity theorem. Following~\cite{Ng}*{Theorem 5.2.1}, the Fourier transform of $\mu_t$ is given by:
    \[ \widehat{\mu_t}(\xi) = e^{-i a_t \xi} \prod_{\chi \ne 1} \prod_{\gamma_\chi > 0} J_0\left( \frac{2 |b_{t,\chi}| \xi}{\sqrt{1/4 + \gamma_\chi^2}} \right)\, ,\]
    where $J_0$ is the Bessel function of the first kind of order 0. Because $(t \circ f_\iota)^+ \not\equiv 0$, the infinite product decays rapidly, ensuring that the limiting distribution $\mu_t$ is absolutely continuous with a real analytic density function $f_t$. This shows that $\delta_{L/K}(t)=\mu_t(0,\infty)$, and the real analyticity of $f_t$ guarantees that $0 < \delta_{L/K}(t) < 1$.
    
    To deduce the last assertion of the Theorem, it suffices to use the fact that $J_0$ is an even function. By applying the inverse Fourier transform, this evenness implies that the density $f_t$ is symmetric around $x=-a_t$. Consequently, if $a_t=0$, then $\delta_{L/K}(t)=1/2$, if $a_t<0$ then $\delta_{L/K}(t)\in(1/2,1)$, and if $a_t>0$ then $\delta_{L/K}(t)\in (0,1/2)$.
\end{proof}
We deduce the Corollary~\ref{maincriterion} as follows: since under the stronger assumptions $\mathrm{LI}(L)$ hold and $\Gal(L/\Q)$ has no symplectic irreducible representation, we have $\mathrm{ord}_{s=1/2}L(s;L/\Q;\chi)=0$ for all $\chi \in \Irr(G^+)\setminus \{1\}$. This implies that the inequality $a_{t_{C_1,C_2}}<0$ reduces to $\mu(2\iota)\langle t_{C_1,C_2},r_{2\iota}\rangle_G<0$, which in particular implies that $\gamma(t_{C_1,C_2})\le 2\iota(t_{C_1,C_2})$. Applying Theorem~\ref{characterization}, we obtain exactly the equivalence of Corollary~\ref{maincriterion}.

\subsection{Weighted prime counting functions}
We start by proving Theorem~\ref{criterion-extreme-bias}.
\begin{proof}[Proof of Theorem~\ref{criterion-extreme-bias}]
Fix $\alpha > 0$. By Proposition~\ref{asymptotics}, if $\gamma(t) < 2\iota(t)$, and either $\mathrm{GRH}(L)$ holds or $\gamma(t) = \iota(t)$, we have for all $x \ge 2$:
\[\pi(x; L/K; t) = \mu(\gamma)\langle t, r_\gamma\rangle \frac{x^{1/\gamma}}{\log x} + O_{L,K,t}\left(\frac{x^{1/\gamma}}{(\log x)^2}\right)\, .\]
A summation by parts yields:
\begin{align}
\pi_\alpha(x; L/K; t) &= \frac{\pi(x; L/K; t)}{x^\alpha} + \alpha \int_2^x \frac{\pi(u; L/K; t)}{u^{\alpha+1}} \mathrm{d}u \label{Sumpartspialpha} \\
&= \mu(\gamma)\langle t, r_\gamma\rangle \frac{x^{\frac{1}{\gamma} - \alpha}}{\log x} + \alpha \int_2^x \mu(\gamma)\langle t, r_\gamma\rangle \frac{u^{\frac{1}{\gamma} - \alpha - 1}}{\log u} \mathrm{d}u + O_{L,K,t}\left(\frac{x^{\frac{1}{\gamma} - \alpha}}{(\log x)^2}\right)\, , \nonumber
\end{align}
where we used the fact that the integral $\int_2^x \frac{u^{\frac{1}{\gamma} - \alpha - 1}}{(\log u)^2} \mathrm{d}u \ll \frac{x^{\frac{1}{\gamma} - \alpha}}{(\log x)^2}$ for $\alpha < 1/\gamma$, and is bounded by $O(1)$ otherwise. We now distinguish three cases based on the value of $\alpha$: if $\alpha < 1/\gamma$, we have \[\int_2^x \frac{u^{\frac{1}{\gamma} - \alpha - 1}}{\log u} \mathrm{d}u = \frac{x^{\frac{1}{\gamma} - \alpha}}{\bigl(\frac{1}{\gamma} - \alpha\bigr)\log x} + O\left(\frac{x^{\frac{1}{\gamma}- \alpha}}{(\log x)^2}\right)\, .\] This yields:
\[\pi_\alpha(x; L/K; t) = \mu(\gamma)\langle t, r_\gamma\rangle \left( 1 + \frac{\alpha}{1/\gamma - \alpha} \right) \frac{x^{1/\gamma - \alpha}}{\log x} + O_{L,K,t}\left(\frac{x^{1/\gamma - \alpha}}{(\log x)^2}\right)\, ,\]
where the main term factor simplifies to $K_{\alpha,t} = \mu(\gamma)\langle t, r_\gamma\rangle/(1 - \alpha\gamma)$.

If $\alpha = 1/\gamma$, we have $\pi(x;L/K;t)/(x^{1/\gamma}) \ll 1/(\log x)$. The main contribution comes entirely from the integral:
\(\int_2^x \frac{\mathrm{d}u}{u \log u} = \log \log x - \log \log 2\),
while the integral of the error term converges, $\int_2^x \frac{\mathrm{d}u}{u(\log u)^2} = O(1)$. This yields the second case with $K'_t = \mu(\gamma)\alpha\langle t, r_\gamma\rangle  $.

If $\alpha > 1/\gamma$, the exponent $\frac{1}{\gamma}- \alpha$ is strictly negative. Thus, the integral $\int_2^\infty u^{\frac{1}{\gamma} - \alpha - 1}\mathrm{d}u/(\log u) $ converges absolutely. Hence, $\pi_\alpha(x; L/K; t)$ is bounded. This completes the proof.
\end{proof}
We now prove Proposition~\ref{Aoki-Koyamaunbiasedness}:
\begin{proof}[Proof of Proposition~\ref{Aoki-Koyamaunbiasedness}]
    By Lemma~\ref{theta_expansion}, using $\bigl\langle (t_{C_1,C_2}\circ f_\iota)^+, 1\bigr\rangle_{G^+}=\langle t_{C_1,C_2}\circ f_\iota,1\rangle_G=\langle t_{C_1,C_2},r_\iota\rangle_G=0$, we deduce that under $\mathrm{GRH}(L)$ we have
    \[\theta\left(x;L/K;t_{C_1,C_2}\right) \ll_{L,K,C_1,C_2} x^{1/(2\iota)}(\log x)^2\, . \]
    A summation by parts implies that 
    \[\pi\left(x;L/K;t_{C_1,C_2}\right)\ll_{L,K,C_1,C_2} x^{1/(2\iota)} \log x\, .\]
    Finally, if $\alpha>1/\gamma =1/(2\iota)$, then by~\eqref{Sumpartspialpha}, we have $\pi_\alpha\left(x;L/K;t_{C_1,C_2}\right)\ll_{L,K,C_1,C_2} 1$. This proves the proposition.
\end{proof}

\section{Explicit group-theoretic constructions}\label{constructions}

\subsection{Property $\mathcal{P}(d)$ and proof of Theorem~\ref{existence-extreme-bias}}\label{subsec:PropP}
The aim of this subsection is to prove Theorem~\ref{existence-extreme-bias}. In order to do this, we will first translate the problem into the existence of groups satisfying a group-theoretic property, namely property $\mathcal{P}(d;x,y)$. Throughout this section, if $G\subset G^+$ are finite groups and $t\, :\,G\to \R$ is a class function, we denote by $t^+:=\mathrm{Ind}_G^{G^+}(t)$ the induced class function by $t$ on $G^+$.
We now introduce the property $\mathcal{P}(d;x,y)$, which was defined in~\cite{Hay}*{Definition 3.1}, and which led to unconditional examples of extreme Chebyshev's bias.
\begin{defi}\label{propertyP(d)}
    Let $d\in \mathcal{S}$. We say that a finite group $G$ satisfies the property $\mathcal{P}(d;x,y)$ if $x,y\in G$ are elements of the same order such that for all squarefree integers $\ell<d$ we have $r_\ell(x)=r_\ell(y)$ and $r_d(x)<r_d(y)$.
\end{defi}
The following proposition shows the relevance of this property for the proof of Theorem~\ref{existence-extreme-bias}:
\begin{prop}\label{cayley}
    Let $G$ be a finite group of order $n$ satisfying the property $\mathcal{P}(d;x,y)$, where $d\ge 2$ is a squarefree integer. We identify $G$ with a subgroup of $\mathfrak{S}_n$ through the Cayley embedding, and let $C_x$ and $C_y$ denote the respective conjugacy classes of $x$ and $y$ in $G$. We have \[\gamma\left( t_{C_x,C_y}\right)=\iota^{\mathfrak{S}_n}\left( t_{C_x,C_y}\right)\, .\]
\end{prop}
\begin{proof}
    By~\cite{Hay}*{Proposition 2.1}, we have the inequality $\gamma\left( t_{C_x,C_y}\right)\ge\iota^{\mathfrak{S}_n}\left( t_{C_x,C_y}\right)$. By~\cite{Hay}*{Lemma 2.8}, we have $\gamma\left( t_{C_x,C_y}\right)\le\iota^{\mathfrak{S}_n}\left( t_{C_x,C_y}\right)$ provided that all elements $a,b\in G$ of the same order are conjugate in $\mathfrak{S}_n$, which holds because we considered the Cayley embedding (this is a well-known result; we refer the reader to~\cite{Hay}*{Proof of Theorem 1.2, \S~2} where it was reproved).
\end{proof}
Proposition~\ref{cayley} combined with Theorem~\ref{criterion-extreme-bias} shows that in order to prove Theorem~\ref{existence-extreme-bias} it suffices to provide examples of groups $G$ satisfying the property $\mathcal{P}(p;x,y)$. Such examples are given in~\cite{Hay}*{\S 3}; for instance, \cite{Hay}*{Theorem 3.17} characterizes abelian groups satisfying $\mathcal{P}(p;x,y)$. We summarize here some consequences of~\cite{Hay}:
\begin{prop}\label{examplespropertyP}
    Let $p\ge2$ be a prime number. In each of the following cases, the group $G$ satisfies the property $\mathcal{P}(p;x,y)$:
    \begin{enumerate}
        \item $G=C_{p^n}\times C_{p^m}\times H$, where $H$ is a finite abelian group, $x=\left(c_{p^n},1\right)$, $y=\left(1,c_{p^n}\right)$, and $1\le n<m$.
        \item $G=Q_8\times C_4$, where $x=(-1,1)$ and $y=(1,c_2)$.
    \end{enumerate}
\end{prop}
\begin{proof}
    For the first statement, see~\cite{Hay}*{Theorem 3.17}, and for the second one, see~\cite{Hay}*{Corollary 1.4}.
\end{proof}
\begin{proof}[Proof of Theorem~\ref{existence-extreme-bias}]
    Let $p\ge2$ be a prime number. By Proposition~\ref{examplespropertyP}, if $G=C_{p^n}\times C_{p^m}\times H$, where $H$ is a finite abelian group, $x=\left(c^{p^n},1\right)$, $y=\left(1,c^{p^n}\right)$, and $1\le n<m$, then $G$ satisfies the property $\mathcal{P}(p;x,y)$. Define $n:=|G|$, and let us consider $L/\Q$ to be a Galois extension of $\Q$ with Galois group isomorphic to $\mathfrak{S}_n$. Identifying $G$ with a subgroup of $\mathfrak{S}_n$ through the Cayley embedding allows us to consider the subextension $K=L^G$. By Galois theory, we have $\Gal(L/K)\cong G$. Theorem~\ref{existence-extreme-bias} is thus a direct consequence of Proposition~\ref{cayley} and Theorem~\ref{criterion-extreme-bias}.
\end{proof}

\subsection{Basic properties of the root-counting function}
Before we move to the construction of groups satisfying the property $\mathcal{P}(2p;x,y)$, we state and prove some basic properties of the function $r_n$. We use the following notation: If $G$ is a group and $A\subset G$ is a subset, we denote, for $n\ge1$, $\mathcal{R}_n^G(A):=\{x\in G\, \colon\,  x^n\in A\}$. When $A=\{g\}$, we define $\mathcal{R}_n^G(g):=\mathcal{R}_n^G(\{g\})$. We simply write $\mathcal{R}_n$ instead of $\mathcal{R}_n^G$ when the underlying group is clear from the context. Let us note that we have $r_n^G(g)=|\mathcal{R}_n^G(g)|$. 
It is clear that the functions $r_n$, $n\ge1$, are invariant on conjugacy classes; in fact they are invariant on larger sets: rational classes.
\begin{defi}[Rational classes]
    Let $G$ be a finite group. A rational class in $G$ is a class of an element in $G$ with respect to the equivalence relation $\mathcal{T}$ defined as follows: if $x,y\in G$
    \[x\, \mathcal{T}\,y \quad\text{ if and only if}\quad \exists\, g\in G\, ; \quad \langle x\rangle=g\langle y\rangle g^{-1}\, .\]
    We call the rational class of $x$ the class of $x$ with respect to $\mathcal{T}$. In other words, the rational class of $x$ is the union of conjugacy classes of $x^k $, for $k$ coprime to $\Ord(x)$.
\end{defi}
\begin{lem}\label{rational-classes}
    Let $G$ be a finite group. For all $n\ge 1$, the function $r_n$ is invariant on rational classes.
\end{lem}
\begin{proof}
    Let $x,g\in G$. The bijection $\mathrm{Int}_g\, \colon\,  \mathcal{R}_n(x)\to \mathcal{R}_n(g x g^{-1})$, given by restricting the inner automorphism of $G$ associated to $g$ shows that $r_n$ is invariant on conjugacy classes. Let $k\in \Z$ be coprime to $\Ord(x)$. It suffices to prove that $r_n(x)=r_n\bigl(x^k\bigr)$. By Dirichlet's theorem on arithmetic progressions, there exists a prime $p>n|G|$ such that $p\equiv k\pmod{\Ord(x)}$. The map $\Psi\, \colon\,  \mathcal{R}_n(x)\to\mathcal{R}_n(x^k)$, defined by $\Psi(h)=h^p$, is a bijection. Indeed by Bézout's theorem, there exist $u,v\in \Z$ such that $up+vn\Ord(x)=1$, and the map $\Psi'\, \colon\,  \mathcal{R}_n(x^k)\to \mathcal{R}_n(x)$ defined by $\Psi'(h)=h^u$ satisfies $\Psi\circ\Psi'=\Id_{\mathcal{R}_n(x^k)}$ and $\Psi'\circ \Psi=\Id_{\mathcal{R}_n(x)}$. This proves the result. 
\end{proof}
We now state two consequences of Lemma~\ref{rational-classes}. The first one allows us to replace $x$ and $y$ by some specific elements $x'$ and $y'$ such that $x\mathcal{T}x'$ and $y\mathcal{T}y'$. This will be very useful in \S~\ref{lbounds}.
\begin{cor}\label{invarianceofpropertyP}
    Let $G$ be a finite group satisfying $\mathcal{P}(d;x,y)$. If $x',y'\in G$ are such that $x'\, \mathcal{T}\,x$ and $y'\,\mathcal{T}\, y$, then $G$ satisfies $\mathcal{P}(d;x',y')$.
\end{cor}
The second consequence of Lemma~\ref{rational-classes} will allow us to reduce the study of $r_n$ to the cases where $\mathrm{rad}(n)$ divides $|G|$.
\begin{lem}\label{rnd=rn}
    Let $G$ be a finite group and let $d\ge 2$ be an integer coprime to $|G|$. Then, for all $n\ge 1$ and all $x\in G$, we have \[r_{nd}(x)=r_n(x)\, .\]
\end{lem}
\begin{proof}
    By Bézout's theorem, there exists $u,v\in \Z$ such that $ud+v|G|=1$. Let $x'=x^u$; we have $x\mathcal{T}x'$. Moreover, $\mathcal{R}_{nd}(x)=\mathcal{R}_{n}(x')$; indeed, if $z\in \mathcal{R}_{nd}(x)$, then \[z^n=z^{n(1-v|G|)}=z^{ndu}=x^u=x'\, .\]
    Conversely, if $z\in \mathcal{R}_{n}(x')$, then \[z^{nd}=x'^d=x^{ud}=x^{1-v|G|}=x\, .\]
    This shows that $r_{nd}(x)=r_n(x')$. By Lemma~\ref{rational-classes}, we deduce that $r_{nd}(x)=r_n(x)$.
\end{proof}

\begin{lem}\label{rootorder}
    Let $G$ be a finite group and let $n\ge 1$. If $g\in G$ is such that $d=\gcd(\Ord(g),n)$, then for all $x\in \mathcal{R}_n(g)$, we have \[ d\Ord(g) \mid \Ord(x)\mid n\Ord(g)\, .\]
    In other words, $\Ord(x)$ is a multiple of $d\Ord(g)$ that divides $n\Ord(g)$.
\end{lem}
\begin{proof}
    Since $x^n=g$, we have $\Ord(g)=\Ord(x)/\gcd(\Ord(x),n)$. In particular, $\Ord(x)$ is a multiple of $\Ord(g)$, thus $\gcd(\Ord(x),n)$ is a multiple of $d$. This proves that $d\Ord(g)\mid \Ord(x)$. Since $\gcd(\Ord(x),n)$ divides $n$, we also have $\Ord(x)\mid n\Ord(g)$.
\end{proof}

\begin{defi}\label{def:primitiveroots}
    Let $G$ be a finite group and let $g\in G$. We say that $x\in \mathcal{R}_n(g)$ is a trivial $n$-th root of $g$ if $\Ord(x)=\Ord(g)$, and we say that $x\in \mathcal{R}_n(g)$ is a primitive $n$-th root of $g$ if $\Ord(x)=n\Ord(g)$. We denote $\kappa_n^G(g)$ the number of primitive $n$-th roots of $g$ in $G$, that is: 
    \[\kappa_n^G(g):=\bigl| \{ x\in \mathcal{R}_n(g)\, \colon\, \Ord(x)=n\Ord(g)\}\bigr|\, .\]
\end{defi}
\begin{rk}\label{trivialroot}
    We note that $g\in G$ admits a trivial $n$-th root if and only if $\gcd(\Ord(g),n)=1$, and in this case, this trivial $n$-th root is unique. We also note that, by Lemma~\ref{rootorder}, if $p$ is a prime number, then any non-trivial $p$-th root is primitive. 
\end{rk}

\subsection{Explicit constructions of groups satisfying the property $\mathcal{P}(2p;x,y)$}\label{subsec:Prop(2p)}
In this subsection we construct groups satisfying the property $\mathcal{P}(2p;x,y)$. The generalized quaternion groups will play a major role in these constructions. Let us recall the definition of a generalized quaternion group: Let $n\ge 2$ and consider $\Phi: C_4\to \Aut(C_{2n})$, where $\Phi(c_4)(x)=x^{-1}$ for all $x\in C_{2n}$. We define the Quaternion group of order $4n$ as:
\begin{equation}
    Q_{4n}:= (C_{2n} \rtimes_{\Phi}C_4)/\langle(c_2,c_2)\rangle\, .
\end{equation}
The Quaternion group of order $4n$ has a unique element of order $2$, denoted as $-1$. If $\omega\, \colon\,  C_{2n} \rtimes_{\Phi}C_4\to Q_{4n}$ is the natural projection, let us denote $u:=\omega(1,c_4)$ and identify $C_{2n}=\omega\bigl( C_{2n}\times\{1\}\bigr)$ (which is possible because the restriction of $\omega$ to $C_{2n}\times \{1\}$ is injective). We have $Q_{4n}=\langle c_{2n}, u\rangle$. Moreover, $Q_{4n}$ admits the following presentation \[ Q_{4n}=\bigl\langle c_{2n},u \ |\ c_{2n}^{2n}=1,\  c_{2n}^{n}=u^2,\ uc_{2n}u^{-1}=c_{2n}^{-1}\bigr\rangle\, .\]
\begin{lem}\label{rootsnumberinQ4n}
    We have for all $n\ge 2$: 
    \[ r_2^{Q_{4n}}(-1)=\begin{cases}
        2n\quad&\text{if } n\text{ is odd}\\
        2(n+1)\quad &\text{if }n\text{ is even}
    \end{cases}\]
\end{lem}
\begin{proof}
    Let $n\ge2$. We have a disjoint union $Q_{4n}=C_{2n} \cup (C_{2n} u)$. For all $g\in C_{2n}$, we have \[(gu)^2=g\cdot (ugu^{-1}) \cdot u^2=g\cdot g^{-1} \cdot (-1)=-1\, .\]
    Thus $r_2^{Q_{4n}}(-1)=r_{2}^{C_{2n}}(c_2)+2n$. If $n$ is odd, we have $r_2^{C_{2n}}(c_{2})=0$, and if $n$ is even we have $r_2^{C_{2n}}(c_{2})=2$. This proves the lemma.
\end{proof}
\noindent The following proposition provides some examples of groups satisfying the property $\mathcal{P}(2p;x,y)$.
\begin{prop}\label{general-construction}
    Let $n\ge 3$ be an odd integer, and write $n(n-1)=2^{\alpha_0} p_1^{\alpha_1}\dots p_k^{\alpha_k}$, where $p_1<\dots<p_k$ are the odd prime factors of $n(n-1)$. For any integer $1\le j\le k$, let 
    \[ G=Q_{4n}\times Q_{4(n-1)}\times C_{p_1\dots p_{j-1}}\, , \]
    and define the elements $x,y\in G$ as $x=(1,-1,c_{p_1\dots p_{j-1}})$ and $y=(-1,1,c_{p_1\dots p_{j-1}})$. Then, the group $G$ satisfies the property $\mathcal{P}(2p_j;x,y)$ if $p_j \mid n-1$, and it satisfies the property $\mathcal{P}(2p_j;y,x)$ if $p_j \mid n$. In particular, the group $Q_{4n}\times Q_{4(n-1)}$ satisfies the property $\mathcal{P}(2p_1;(1,-1),(-1,1))$ if $p_1 \mid n-1$, and it satisfies the property $\mathcal{P}(2p_1;(-1,1),(1,-1))$ if $p_1 \mid n$.
\end{prop}
\begin{proof}
    Let us denote $c := c_{p_1\dots p_{j-1}}$, $C := C_{p_1\dots p_{j-1}}$, and $r_n:=r_n^G$. Because the order of $c$ is odd, the orders of $x=(1,-1,c)$ and $y=(-1,1,c)$ in $G$ are both equal to $2\Ord(c)$. We evaluate the $\ell$-th roots for any squarefree $\ell < 2p_j$. By Lemma~\ref{rnd=rn}, we may restrict to the case where $\ell$ divides $|G|$. Moreover, if $\ell$ is divisible by any $p_m$ for $1 \le m < j$, since $C$ contains no $p_m$-th roots of $c$, $r_\ell(x) = r_\ell(y) = 0$. 
    
    Thus, we only need to verify the equality $r_\ell(x) = r_\ell(y)$ for $\ell=2$ and $\ell=p_m$ for $j \le m \le k$. Notice that for all these values of $\ell$, $\gcd(\ell, |C|) = 1$, which implies $r_\ell^C(c) = 1$. This allows us to simplify the computations to just the first two coordinates. For $\ell=2$, since $n$ is odd, we have by Lemma~\ref{rootsnumberinQ4n}, $r_2(x) = r_2^{Q_{4n}}(1) r_2^{Q_{4(n-1)}}(-1) = 2\times 2n = 4n$. Similarly, $r_2(y) = 2n\times 2= 4n$. Thus, $r_2(x) = r_2(y)$. For $\ell=p_m$ (with $j \le m \le k$), the prime $p_m$ divides exactly one of $n$ or $n-1$. Assuming $p_m \mid n$ (the case $p_m \mid n-1$ is identical), we have $r_{p_m}(x) = r_{p_m}^{Q_{4n}}(1)r_{p_m}^{Q_{4(n-1)}}(-1)= p_m\times 1 = p_m$, and similarly $r_{p_m}(y) = p_m\times 1 = p_m$. Thus, $r_{p_m}(x) = r_{p_m}(y)$. Finally, for $\ell=2p_j$, the prime $p_j$ divides exactly one of $n$ or $n-1$. \\
    Case 1: $p_j \mid n-1$. 
    Since $\gcd(p_j, n) = 1$, we have, by Lemma~\ref{rnd=rn}, $r_{2p_j}^{Q_{4n}}(1) = r_2^{Q_{4n}}(1) = 2$. In $Q_{4(n-1)}$, the $2p_j$-th roots of $-1$ include the $2(n-1)$ elements outside the cyclic subgroup $C_{2(n-1)}$, in addition to $r_{2p_j}^{C_{2(n-1)}}(c_2)=2p_j$ (which is a consequence of the Chinese Remainder Theorem, see also \eqref{2p-rootsnumb2} which gives a general form for this type of equality). Thus, $r_{2p_j}^{Q_{4(n-1)}}(-1) = 2(n-1) + 2p_j$. This gives:
    \[ r_{2p_j}(x) = 4(n-1) + 4p_j\, . \]
    For $y$, we have $r_{2p_j}(y) = r_{2p_j}^{Q_{4n}}(-1) r_{2p_j}^{Q_{4(n-1)}}(1) = r_{2}^{Q_{4n}}(-1) r_{2p_j}^{Q_{4(n-1)}}(1)=2n\times 2p_j = 4np_j$. Since $n, p_j \ge 3$, we have $4np_j > 4(n-1) + 4p_j \ge r_{2p_j}(x)$, which proves $r_{2p_j}(x) < r_{2p_j}(y)$.\\
    Case 2: $p_j \mid n$. 
    Here, $r_{2p_j}(x) = r_{2p_j}^{Q_{4n}}(1) r_{2p_j}^{Q_{4(n-1)}}(-1) = 2p_j\times 2n = 4np_j$. For $y$, because $n$ is odd, the cyclic subgroup $C_{2n}$ contains no square root of $-1$. Only the $2n$ elements outside the cyclic subgroup satisfy $g^2=-1$ and thus $(g^2)^{p_j}=-1$. Therefore, $r_{2p_j}^{Q_{4n}}(-1) = 2n$. This yields:
    \[ r_{2p_j}(y) = 2n\times 2 = 4n\, . \]
    Since $4np_j > 4n$, we have $r_{2p_j}(y) < r_{2p_j}(x)$. This finishes the proof of the proposition.
\end{proof}
\begin{exe}
    Taking $n=3$ (resp. $n=5$), we see that the group $Q_{12}\times Q_8$ (resp. $Q_{20}\times Q_{16}$) satisfies the property $\mathcal{P}(6;(-1,1),(1,-1))$ (resp. $\mathcal{P}(10;(-1,1),(1,-1))$). This group has order $96$ (resp. $320$).
\end{exe}
\begin{cor}\label{corP(2p)}
    Let $p\ge 3$ be a prime number. If $p-1$ is a power of $2$, then $Q_{4p}\times Q_{4(p-1)}$ satisfies the property $\mathcal{P}(2p;(-1,1),(1,-1))$. If $p$ is a Sophie Germain prime, then $Q_{4(2p+1)}\times Q_{8p}$ satisfies the property $\mathcal{P}(2p;(1,-1),(-1,1))$. More generally, the group $Q_{4p}\times Q_{4(p-1)}\times C_{\mathrm{rad}(p-1)/2}$ satisfies the property $\mathcal{P}(2p; (-1,1,c_{\mathrm{rad}(p-1)/2}),(1,-1,c_{\mathrm{rad}(p-1)/2}))$.
\end{cor}

\begin{proof}
    This is a direct application of Proposition~\ref{general-construction}.
\end{proof}

\subsection{Matrix embeddings and property $\mathcal{Q}(2p;x,y)$}\label{subsec:propQ}
We start this subsection with the definition of the property $\mathcal{Q}(2d;x,y)$:
\begin{defi}\label{propertyQ}
    Let $d\ge 3$ be an odd squarefree integer. We say that a pair $(G,G^+)$ of finite groups $G\subset G^+$ satisfies the property $\mathcal{Q}(2d;x,y)$, if $G$ satisfies the property $\mathcal{P}(2d;x,y)$ and, denoting by $C_x$ and $C_y$ the respective conjugacy classes of $x$ and $y$ in $G$, we have \[\gamma\left(t_{C_x,C_y}\right)=2\iota^{G^+}\left(t_{C_x,C_y}\right)\, .\]
\end{defi}
It is clear that with this definition, if $L/\Q$ is a Galois extension with group $G^+$ having no symplectic irreducible representation, and if $K$ is a subextension such that, denoting $G=\Gal(L/K)$, we have that if $(G,G^+)$ satisfies the property $\mathcal{Q}(2d;x,y)$, then $L/K\in \mathcal{E}_d$. Thus, by Corollary~\ref{maincriterion}, $\delta_{L/K}(C_y,C_x)\in (1/2,1)$, where $C_x$ and $C_y$ are the respective conjugacy classes of $x$ and $y$. This reduces the problem of proving $\mathcal{E}_d\ne \emptyset$ to proving the existence of finite groups $G,G^+$ such that $(G,G^+)$ satisfies the property $\mathcal{Q}(2d;x,y)$ and such that $G^+$ has no symplectic irreducible representation. In what follows, if $G$ is a finite group and $x,y\in G$ with respective conjugacy classes $C_x,C_y$, we simply denote $t_{x,y}:=t_{C_x,C_y}$ (see~\eqref{tC_1C_2}). If $G$ is a subgroup of a finite group $G^+$, for a conjugacy class $C\in G^\#$ that is contained in a conjugacy class $C^+\in (G^+)^\#$, we have the following elementary identity: \begin{equation}\label{inductionindicator}
    \frac{|G|}{|C|}\mathrm{Ind}_G^{G^+} \mathds{1}_C =\frac{|G^+|}{|C^+|}\mathds{1}_{C^+}\, .
\end{equation}
The previous identity can be generalized as follows: 
\begin{lem}\label{invconjset-induction}
    Let $G$ be a subgroup of a finite group $G^+$. Let $D\subset G$ be a non-empty conjugacy invariant set, meaning a subset that is invariant under inner automorphisms of $G$. If $D\subset C^+$ where $C^+$ is a conjugacy class of $G^+$, then \[\mathrm{Ind}_G^{G^+} \mathds{1}_D=\frac{|D|\, |G^+|}{|G|\, |C^+|}\mathds{1}_{C^+}\, .\]
\end{lem}
\begin{proof}
    Let us write $D=C_1\cup\dots\cup C_r$, where $C_1,\dots,C_r$ are distinct conjugacy classes in $G$. We have $\mathds{1}_D=\sum_{i=1}^r\mathds{1}_{C_i}$. Applying~\eqref{inductionindicator}, we deduce: 
    \[\mathrm{Ind}_G^{G^+} \mathds{1}_D=\sum_{i=1}^r \frac{|C_i|\, |G^+|}{|G|\, |C^+|}\mathds{1}_{C^+}=\frac{|D|\, |G^+|}{|G|\, |C^+|}\mathds{1}_{C^+}\, .\]
\end{proof}

In \S~\ref{subsec:Prop(2p)}, we gave a general group-theoretic construction of the generalized quaternion group of order $4n$. In fact, we can also see $Q_{4n}$ as a subgroup of the general linear groups over finite fields. Let $n\ge 3$ be an integer, and let $q\equiv 1 \pmod{2n}$ be a prime. We can view $Q_{4n}$ as a subgroup of $\GL(2, \F_q)$ by considering a primitive $2n$-th root of unity $\xi_{2n}\in \F_q^\times$. Let 
\[ a=\begin{pmatrix}\xi_{2n}&0\\0&\xi_{2n}^{-1}\end{pmatrix}\quad\text{and}\quad b=\begin{pmatrix}0&1\\-1&0\end{pmatrix}\, . \]
Note that $a^n=b^2=-I_2$ and $bab^{-1}=a^{-1}$. Thus, $\langle a,b\rangle\cong Q_{4n}$. It is worth noting that this explicit construction provides an irreducible faithful representation of degree $2$ of $Q_{4n}$ on $\F_q$. We start by proving the following lemma: 
\begin{lem}\label{reducedinduction}
    Let $N,\Gamma_1,\Gamma_2$ be finite groups such that $\Gamma:=\Gamma_1\times \Gamma_2$ is a subgroup of $N$. Let $((a,b),c)\in \Gamma_1^2\times \Gamma_2$ be such that $(1,c)\in \mathcal{Z}(N)$, and let $C_1,C_2,C_a,$ and $C_b$ be the respective conjugacy classes of $(a,c),(b,c),(a,1),$ and $(b,1)$ in $\Gamma$. Let $\ell\ge 1$ be coprime to $\Ord(c)$, and let $u\in\Z$ such that $\ell u\equiv 1\pmod{\Ord(c)}$. Then, for all $y\in N$:
    \[ \mathrm{Ind}_\Gamma^N \bigl(t_{C_1,C_2}\circ f_\ell^\Gamma\bigr) (y)=\mathrm{Ind}_\Gamma^N \bigl(t_{C_a,C_b}\circ f_\ell^\Gamma\bigr) \bigl(y\cdot(1,c^{-u})\bigr)\, .\]
    In particular, we have $\mathrm{Ind}_\Gamma^N \bigl(t_{C_1,C_2}\circ f_\ell^\Gamma\bigr)=0$ if and only if $\mathrm{Ind}_\Gamma^N \bigl(t_{C_a,C_b}\circ f_\ell^\Gamma\bigr)=0$.
\end{lem}
\begin{proof}
    We first prove that for all $z\in \Gamma$ we have $t_{C_1,C_2}(z)=t_{C_a,C_b}(z(1,c^{-1}))$. Let $z=(g_1,g_2)\in \Gamma$. First assume that $g_2\ne c$. Since $(1,c)$ is central in $N$ (and in particular in $\Gamma$), we have $z\notin C_1\cup C_2$, which implies that $t_{C_1,C_2}(z)=0$. We also have $g_2c^{-1}\ne 1$, which means that $(g_1,g_2c^{-1})\notin C_a\cup C_b$, implying $t_{C_a,C_b}(z(1,c^{-1}))=0$. Assume that $g_2=c$. Since $\mathds{1}_{C_1}(g_1,c)=\mathds{1}_{C_a}(g_1,1)$ and $\mathds{1}_{C_2}(g_1,c)=\mathds{1}_{C_b}(g_1,1)$, and since $|C_1|=|C_a|$ and $|C_2|=|C_b|$, we have $t_{C_1,C_2}(z)=t_{C_a,C_b}(z(1,c^{-1}))$, which proves the claim. Let $y\in N$. We have
    \begin{align*}
        \mathrm{Ind}_\Gamma^N&\left(t_{C_1,C_2}\circ f_\ell^\Gamma \right)(y)=\frac{1}{|\Gamma|}\sum_{\substack{g\in N\\ gyg^{-1}\in \Gamma}} t_{C_1,C_2}\bigl( (gyg^{-1})^\ell\bigr)=\frac{1}{|\Gamma|}\sum_{\substack{g\in N\\ gyg^{-1}\in \Gamma}} t_{C_a,C_b}\bigl( (gyg^{-1})^\ell (1,c^{-1})\bigr)\\
        &=\frac{1}{|\Gamma|}\sum_{\substack{g\in N\\ gyg^{-1}\in \Gamma}} t_{C_a,C_b}\bigl( (gyg^{-1})^\ell (1,c^{-u})^\ell\bigr)=
        \frac{1}{|\Gamma|}\sum_{\substack{g\in N\\ gy(1,c^{-u})g^{-1}\in \Gamma}} t_{C_a,C_b}\bigl( (gy(1,c^{-u})g^{-1})^\ell \bigr)\\
        &=\mathrm{Ind}_\Gamma^N\left(t_{C_a,C_b}\circ f_\ell^\Gamma \right)\bigl(y(1,c^{-u})\bigr)\, ,
    \end{align*}
    where we used $(1,c^{-u})\in\mathcal{Z}(N)\cap \Gamma$ to deduce that $(gyg^{-1})^\ell (1,c^{-u})^\ell=(gy(1,c^{-u})g^{-1})^\ell$, and that $gy(1,c^{-u})g^{-1}\in \Gamma\iff gyg^{-1}\in \Gamma $. 
 \end{proof}
 The following lemma provides a criterion for proving that the induced function of $t_{x,y}\circ f_\ell$, $\ell\ge 1$, in a larger group is non-zero:
 \begin{lem}\label{indtx,yne0}
     Let $G$ be a subgroup of a finite group $N$, and let $x,y\in \mathcal{Z}(G)$ and $\ell\in \mathcal{S}$. If there exists $s\in \mathcal{R}_\ell^G(x)$ that is non-conjugate to any element in $\mathcal{R}_\ell^G(y)$ in $N$, then \[\mathrm{Ind}_G^N\, t_{x,y}\circ f_\ell \ne 0\, .\]
 \end{lem}
 \begin{proof}
     Since $x,y\in \mathcal{Z}(G)$, we have $\mathcal{R}_\ell^G(x)$ and $\mathcal{R}_\ell^G(y)$ are conjugacy invariant sets. We write $\mathcal{R}_\ell^G(x)=\cup_{i=1}^{r_x} C_{i,x}$ and $\mathcal{R}_\ell^G(y)=\cup_{j=1}^{r_y} C_{j,y} $  where $C_{i,x},C_{j,y}$, $1\le i\le r_x$, $1\le j\le r_y$, are conjugacy classes in $G$. Thus, by~\eqref{inductionindicator}, we have
     \begin{align*}
         \mathrm{Ind}_G^N \mathds{1}_{\{x\}}\circ f_\ell=\mathrm{Ind}_G^N \mathds{1}_{\mathcal{R}_\ell^G(x)}=\sum_{i=1}^{r_x} \alpha_{i,x}\mathds{1}_{C_{i,x}^+}\ \text{ and }\ 
        \mathrm{Ind}_G^N \mathds{1}_{\{y\}}\circ f_\ell=\mathrm{Ind}_G^N \mathds{1}_{\mathcal{R}_\ell^G(y)}=\sum_{j=1}^{r_y} \alpha_{j,y}\mathds{1}_{C_{j,y}^+}\, ,
     \end{align*}
     where $\alpha_{i,x},\alpha_{j,y}>0$, $C_{i,x}^+$ is the conjugacy class of $N$ containing $C_{i,x}$, and $C_{j,y}^+$ is the conjugacy class of $N$ containing $C_{j,y}$. There exists $1\le i_0\le r_x$ such that $s\in C_{i_0,x}$. Therefore, $\mathrm{Ind}_G^N \mathds{1}_{x}\circ f_\ell (s)\ge \alpha_{i_0,x}>0$. By the assumption, we have $s\notin \cup_{j=1}^{r_y}C_{j,y}^+$, which implies that $\mathrm{Ind}_G^N \mathds{1}_{y}\circ f_\ell (s)=0$. Hence \[\bigl(\mathrm{Ind}_G^N t_{x,y} \circ f_\ell\bigr) (s)=|G|\bigl(\mathrm{Ind}_G^N \mathds{1}_{x}\circ f_\ell\bigr) (s)>0\, .\]
     This finishes the proof.
 \end{proof}
 
 In the following proposition, we use the matrix representation of $Q_{4n}$ above to embed the groups constructed in \S~\ref{subsec:Prop(2p)} into symmetric groups. This is our main tool to prove that $\mathcal{E}_p\ne \emptyset$ for primes $p\ge3$.
\begin{prop}\label{generalembedding}
    Let $n\ge 3$ be an odd integer, and write $n(n-1)=2^{\alpha_0}p_1^{\alpha_1}\dots p_k^{\alpha_k}$, where $p_1<\dots<p_k$ are the odd prime factors. Let $G=Q_{4n}\times Q_{4(n-1)}\times C_{p_1\cdots p_{j-1}}$ and $x,y\in G$ be defined as $x=(1,-1,c_{p_1\dots p_{j-1}})$ and $y=(-1,1,c_{p_1\cdots p_{j-1}})$. 
    For any prime $q\equiv 1 \pmod{2n(n-1)}$, there exists an embedding of $G$ into $\mathfrak{S}_{q^4}$ such that $\bigl(G,\mathfrak{S}_{q^4}\bigr)$ satisfies the property $\mathcal{Q}(2p_j;x,y)$ if $p_j\mid (n-1)$ and satisfies $\mathcal{Q}(2p_j;y,x)$ if $p_j\mid n$.
\end{prop}
Since by Proposition~\ref{general-construction} the group $G$ satisfies the property $\mathcal{P}(2p_j;x,y)$ if $p_j \mid n-1$, and it satisfies the property $\mathcal{P}(2p_j;y,x)$ if $p_j \mid n$, it remains to prove that $\gamma(t_{x,y})=2\iota^{G^+}(t_{x,y})$. We proceed in 5 main steps: \begin{enumerate}
    \item We use the matrix representation of the generalized Quaternion group to embed $G$ in $N:=\GL(4,\F_q)$. 
    \item We reduce the proof of $\mathrm{Ind}_G^N t_{x,y}\circ f_\ell=0$ for all squarefree $\ell< p_j$ to the case $\ell=2$.
    \item We prove that $\mathrm{Ind}_G^N t_{x,y}\circ f_2=0$ by studying the eigenvalues of square roots of $x$ and $y$ in $G$.
    \item We embed $\GL(4,\F_q)$ naturally in $G^+:=\mathfrak{S}_{q^4}$, which gives an embedding of $G$ in $G^+$.
    \item We prove that $\mathrm{Ind}_G^N t_{x,y}\circ f_{p_j}\ne0$.
\end{enumerate}
\begin{proof}
    Let $q\equiv 1 \pmod{2n(n-1)}$ be a prime. Thanks to the matrix construction discussed above, we have natural embeddings of $Q_{4n}$ and $Q_{4(n-1)}$ into $\GL(2, \F_q)$. This yields a natural embedding by blocks $\eta\, \colon\,  Q_{4n}\times Q_{4(n-1)}\to \GL(4,\F_q)$. Letting $\xi\in \F_q^{\times}$ be a primitive $p_1\cdots p_{j-1}$-th root of unity in $\F_q^\times$ (which exists because $q\equiv 1\pmod{p_1\dots p_{j-1}}$), yields an embedding $\beta \, \colon\,  C_{p_1\cdots p_{j-1}}\to \GL(4,\F_q)$ given by $\beta(c_{p_1\cdots p_{j-1}})=\xi I$, where $I$ is the identity matrix in $\GL(4,\F_q)$. Since $\xi I\in \mathcal{Z}(\GL(4,\F_q))$, we have $\beta( C_{p_1\cdots p_{j-1}})\cap \eta (Q_{4n}\times Q_{4(n-1)})\subset \mathcal{Z}(\eta(Q_{4n}\times Q_{4(n-1)}))$. Since $\mathcal{Z}(\eta(Q_{4n}\times Q_{4(n-1)}))$ has order $4$ and there is no subgroup of $\beta( C_{p_1\cdots p_{j-1}})$ of order $2$, we deduce that $\beta( C_{p_1\cdots p_{j-1}})\cap \eta (Q_{4n}\times Q_{4(n-1)})=\{I\}$. This yields an embedding $G$ into $\GL(4, \F_q)$, which allows us to identify elements of $G$ with elements of $\GL(4, \F_q)$. Furthermore, because $q\equiv 1 \pmod{2n(n-1)}$, and because every element of $G$ has order dividing $2n(n-1)$, then every element of $G$ is diagonalizable over $\F_q$ (the corresponding matrix has a separable split annihilator polynomial in $\F_q$). This implies that elements $g_1,g_2 \in G$ are conjugate in $\GL(4, \F_q)$ if and only if the multi-sets of their eigenvalues are the same; $\mathrm{Spec}(g_1)=\mathrm{Spec}(g_2)$. (Here we mean by a multi-set, a set where the multiplicity is taken into account). In what follows, we denote $N:=\GL(4,\F_q)$.

    We first consider the $\ell$-th roots for any squarefree $\ell<p_j$. If $\ell$ is divisible by some prime $p_m\in \{p_1, \dots, p_{j-1}\}$, then neither $x$ nor $y$ possesses any $\ell$-th roots in $G$. This is because $c_{p_1\dots p_{j-1}}$ has no $p_m$-th root in $C_{p_1\dots p_{j-1}}$. Assume that $\ell$ is coprime to $p_1\cdots p_{j-1}$. By Lemma~\ref{reducedinduction}, in order to prove that $\mathrm{Ind}_G^N(t_{x,y}\circ f_\ell)=0$, it is sufficient to prove that $\mathrm{Ind}_G^N(t_{x',y'}\circ f_\ell)=0$ where $x'=(1,-1,1)$ and $y'=(-1,1,1)$. Writing $\ell=2^{\varepsilon}\ell'$, where $\varepsilon\in\{0,1\}$ and $\ell'$ is an odd squarefree integer, and noting that since $\ell<p_j$ and coprime to $p_1\cdots p_{j-1}$, we have $\ell'$ is coprime to $|G|$, shows that $\mathds{1}_{\{x'\}}\circ f_{\ell'}=\mathds{1}_{\{x'\}}$ and that $\mathds{1}_{\{y'\}} \circ f_{\ell'}=\mathds{1}_{\{y'\}}$, this is because $(x')^{\ell'}=x'$ and $(y')^{\ell'}=y'$ and $f_{\ell'}\, \colon\,  G\to G$ is a bijection. Hence, $t_{x',y'}\circ f_\ell=t_{x',y'}\circ f_{2^\varepsilon}$, which reduces the problem to proving that $\mathrm{Ind}_G^N(t_{x',y'}\circ f_{2^\varepsilon})=0$. The case $\varepsilon=0$ is trivial, since $x'$ and $y'$ have eigenvalues $1$ and $-1$ each with multiplicity $2$, they are conjugate in $N$.
    
    We now move to proving the case $\varepsilon=1$. We have \begin{align}\label{x'f2-y'f2}
        \mathds{1}_{\{x'\}}\circ f_2=\mathds{1}_{ \{1,-1\}\times \left(\mathcal{R}_2^{Q_{4(n-1)}}(-1)\right)\times\{1\} }\quad\text{and}\quad
        \mathds{1}_{\{y'\}}\circ f_2=\mathds{1}_{ \left(\mathcal{R}_2^{Q_{4n}}(-1)\right)\times \{1,-1\}\times\{1\} }\, .
    \end{align}
    Let $s\in \mathcal{R}_2^{Q_{4(n-1)}}(-1)$. Considering the embedding $Q_{4(n-1)} \to \GL(2,\F_q)$, $s$ admits $X^2+1$ as an annihilating polynomial. Thus, the only possible eigenvalues of $s$ are $-i$ and $i$, where $i\in \F_q$ is a primitive $4$-th root of unity (which exists because $4$ divides $q-1$). Since $s$ is not central in $Q_{4(n-1)}$, $s$ cannot have $-i$ or $i$ as an eigenvalue of multiplicity $2$, thus $s$ has $i,-i$ as eigenvalues with multiplicity $1$. Let us denote $D_{1,x}:=\{1\}\times \mathcal{R}_2^{Q_{4(n-1)}}(-1)\times \{1\}$. $D_{1,x}$ is a non-empty conjugacy invariant set in $G$. Moreover, denoting $C_1^+\subset N$ the conjugacy class of all matrices that are diagonalizable having $1$ as an eigenvalue with multiplicity $2$, and both $i,-i$ as eigenvalues with multiplicity $1$, we have $D_{1,x}\subset C_1^{+}$. Denoting the conjugacy invariant set $D_{1,y}=\mathcal{R}_2^{Q_{4n}}(-1)\times \{1\}\times\{1\}$, we show in the same way that $D_{1,y}\subset C_1^+$. Because $n$ is odd, we have, by Lemma~\ref{rootsnumberinQ4n}, $|D_{1,x}|=|D_{1,y}|$. Applying Lemma~\ref{invconjset-induction}, we deduce that \[ \mathrm{Ind}_G^N \bigl( \mathds{1}_{D_{1,x}}-\mathds{1}_{D_{1,y}}\bigr)=0\, .\]
    Defining $C_2^+$ as the conjugacy class in $N$ of all matrices that are diagonalizable, having $-1$ as an eigenvalue with multiplicity $2$, and both $i,-i$ as eigenvalues with multiplicity $1$, and defining $D_{-1,x}:=\{-1\}\times \mathcal{R}_2^{Q_{4(n-1)}}(-1)\times \{1\}$ and $D_{-1,y}:=\mathcal{R}_2^{Q_{4n}}(-1)\times \{-1\}\times\{1\}$, we deduce as above that $D_{-1,x}\cup D_{-1,y}\subset C_2^+$, which implies  \[ \mathrm{Ind}_G^N \bigl( \mathds{1}_{D_{-1,x}}-\mathds{1}_{D_{-1,y}}\bigr)=0\, .\]
    Finally, by~\eqref{x'f2-y'f2}, we have \[\mathds{1}_{\{x'\}}\circ f_2=\mathds{1}_{D_{1,x'}}+\mathds{1}_{D_{-1,x'}}\ \text{ and }\ \mathds{1}_{\{y'\}}\circ f_2=\mathds{1}_{D_{1,y'}}+\mathds{1}_{D_{-1,y'}} \, .\]
    We conclude that $\mathrm{Ind}_G^N (\mathds{1}_{\{x'\}}\circ f_2-\mathds{1}_{\{y'\}}\circ f_2)=0$.

    Since $N=\GL(4, \F_q)$ acts faithfully on the vector space $\F_q^4$, we obtain a natural embedding of $G$ into the symmetric group $G^+:=\mathfrak{S}_{q^4}$. By the above, we have for all squarefree $\ell<p_j$, \[\mathrm{Ind}_G^{G^+} t_{x,y}\circ f_\ell=\mathrm{Ind}_N^{G^+}\mathrm{Ind}_G^N t_{x,y}\circ f_\ell=0\, .\]

    It remains to prove that $\mathrm{Ind}_G^{G^+}t_{x,y}\circ f_{p_j}\ne 0$. The prime $p_j$ divides exactly one of $n$ or $n-1$. Assume that $p_j\mid n$ (the case $p_j\mid n-1$ is treated in a similar way). Any $p_j$-th root of $y'=(-1, 1,1)$ requires a $p_j$-th root of $1$ in $Q_{4(n-1)}$. Since $\gcd(p_j, 2(n-1))=1$, the only such element is $1$. The corresponding block matrix is $I_2$, which means any primitive $p_j$-th root of $y'$ has $1$ as an eigenvalue. Thus, $y'$ is embedded in $\mathfrak{S}_{q^4}$ into a permutation which fixes at least $q$ elements (the eigenspace associated with the eigenvalue $1$ is not trivial). On the other hand, a primitive $p_j$-th root of $x'=(1, -1, 1)$ requires a primitive $p_j$-th root of $1$ in $Q_{4n}$. In our matrix representation, this element is of the form $\mathrm{diag}(\xi_0, \xi_0^{-1})$ for some $\xi_0\ne 1$. The eigenvalue for the $-I_2$ block is $-1$ with multiplicity $2$. Thus, $1$ is not an eigenvalue for any primitive $p_j$-th root of $x'$. Hence, a primitive $p_j$-th root of $y'$ is not conjugate in $G^+$ to any element in $\mathcal{R}_{p_j}^G(x')$. Let $s=\bigl(c_{2p_j},1,(c_{p_1\cdots p_{j-1}})^u\bigr)\in Q_{4n}\times Q_{4(n-1)}\times C_{p_1\cdots p_{j-1}}$, where $u\in\Z$ satisfies $u p_j\equiv 1\pmod{ p_1\cdots p_{j-1}}$. We have $s\in \mathcal{R}_{p_j}^G(y)$, and $\Ord(s)=2p_1\cdots p_j=p_j \Ord(y)$. Thus, $s$ is a primitive $p_j$-th root of $y$. If $s$ in $G^+$ is conjugate to an element $z\in \mathcal{R}_{p_j}^G(x)$, then $s^{p_1\cdots p_{j-1}}$ is conjugate to $z^{p_1\cdots p_{j-1}}$. Moreover $z^{p_1\cdots p_{j-1} \cdot p_j}=x^{p_1\cdots p_{j-1}}=x'$ and similarly $s^{p_1\cdots p_{j-1}\cdot p_j}=y'$, meaning that $z^{p_1\cdots p_{j-1}}\in \mathcal{R}_{p_j}^G(x')$ and $s^{p_1\cdots p_{j-1}}\in \mathcal{R}_{p_j}^G(y')$. Since $\Ord\bigl(s^{p_1\cdots p_{j-1}}\bigr)=2p_j=p_j\Ord(y')$, we have that $s^{p_1\cdots p_{j-1}}$ is a primitive $p_j$-th root of $y'$, which is conjugate to an element of $\mathcal{R}_{p_j}^G(x')$. Since we proved that such an element cannot exist, this gives a contradiction. Thus, $s\in \mathcal{R}_{p_j}^G(y)$ is not conjugate to any element in $\mathcal{R}_{p_j}^G(x)$. By Lemma~\ref{indtx,yne0}, we have that $\mathrm{Ind}_G^{G^+}t_{x,y}\circ f_{p_j} \ne 0$.
\end{proof}

As in \S~\ref{subsec:Prop(2p)}, taking the particular case where $n=p$ is a prime number we deduce the following:

\begin{cor}\label{smallerembedding}
    Let $p\ge 3$ be a prime number. Let $q$ be a prime satisfying $q\equiv 1 \pmod{2p(p-1)}$. Let \[G=Q_{4p}\times Q_{4(p-1)}\times C_{\mathrm{rad}(p-1)/2}\, ,\] and define the elements $x=(-1, 1,c_{\mathrm{rad}(p-1)/2})$ and $y=(1, -1,c_{\mathrm{rad}(p-1)/2})$ in $G$. Then there exists an embedding of $G$ into the symmetric group $\mathfrak{S}_{q^4}$ such that $\bigl(G,\mathfrak{S}_{q^4}\bigr)$ satisfies the property $\mathcal{Q}(2p;x,y)$. If $p$ is a Sophie Germain prime and $q'$ is a prime satisfying $q'\equiv 1 \pmod{4p(2p+1)}$, then there exists an embedding of $G':=Q_{4(2p+1)}\times Q_{8p}$ into $\mathfrak{S}_{(q')^4}$ such that $\bigl(G',\mathfrak{S}_{(q')^4}\bigr)$ satisfies the property $\mathcal{Q}(2p;(1,-1),(-1,1))$.
\end{cor}

\begin{proof}
    This follows from combining Proposition~\ref{generalembedding} with Corollary~\ref{corP(2p)}.
\end{proof}

\begin{examples}
    \begin{enumerate}
        \item Taking $p=3$, we have $p-1=2$, and $13\equiv 1\pmod{2p(p-1)}$. Thus, the pair $(Q_{12}\times Q_8,\mathfrak{S}_{28561})$ satisfies the property $\mathcal{Q}(6;x,y)$.
        \item Taking $p=5$, we have $p-1=4$, and $41\equiv 1 \pmod{2p(p-1)}$. Thus, the pair $(Q_{20}\times Q_{16},\mathfrak{S}_{2825761})$ satisfies the property $\mathcal{Q}(10;x,y)$.
    \end{enumerate}
\end{examples}

\section{Lower bounds for the order of groups satisfying $\mathcal{P}(2p;x,y)$}\label{lbounds}
The goal of this section is to prove the following theorems. The first one provides a general lower bound for groups satisfying the property $\mathcal{P}(2p;x,y)$. 
\begin{thm}\label{lowerbound}
    Let $G$ be a finite group satisfying the property $\mathcal{P}(2p;x,y)$. Then $|G|\ge 8p^2$. 
\end{thm}
The second theorem provides the best lower bounds in the case $p\in \{3,5\}$.
\begin{thm}\label{n3andn5}
    Let $G$ be a finite group satisfying the property $\mathcal{P}(2p;x,y)$. If $p=3$ then $|G|\ge 96$, and if $p=5$ then $|G|\ge 320$.
\end{thm}
Let us note that Theorem~\ref{np} follows by combining Corollary~\ref{smallerembedding} with Theorems~\ref{lowerbound} and~\ref{n3andn5}. 
\subsection{Root-counting formulas}
We first establish useful formulas for $r_p$, $r_{2p}$, and $\kappa_{2p}$ (see Definition~\ref{def:primitiveroots}). Let us recall also that $\mathcal{R}_n^G(g)$ denotes the set of $n$-th roots of $g$ in $G$, and that we simply write $\mathcal{R}_n(g)$ when the underlying group is clear from the context.
\begin{lem}
    Let $G$ be a finite group and let $p\ge2$ be a prime number. For all $g\in G $ with order $\Ord(g)=p^rm$, such that $\gcd(p,m)=1$ and $r\ge0$, we have \begin{equation}\label{p-rootsnumber}
        r_p(g)=\#\{t\in \mathcal{R}_p\bigl(g^m\bigr)\, \colon\,  tg=gt\}\, .
    \end{equation}
    If moreover $p\ge3$, we have 
    \begin{align}
        r_{2p}(g)=\#\left\{(s,t)\in \mathcal{R}_2\bigl(g\bigr) \times\mathcal{R}_p\bigl(g^m\bigr)\, \colon\,  st=ts\right\}\label{2p-rootsnumb2}\, ,
    \end{align}
    and \begin{equation}\label{2p-primrootsnumb}
        \kappa_{2p}(g)=\#\Bigl\{(s,t)\in \bigl(\mathcal{R}_2(g)\setminus \langle g\rangle\bigr) \times \bigl(\mathcal{R}_p(g^m)\setminus\{1\}\bigr)\, \colon\,  st=ts\Bigr\}\, .
    \end{equation}
\end{lem}
\begin{proof}
    By Bézout's theorem, there exist $\alpha,\beta \in \Z$ such that $\alpha p^{r+1}+\beta m=1$. We define the map $\varphi_1\, \colon\, \left\{ t\in \mathcal{R}_p\bigl(g^m\bigr)\, \colon\,  tg=gt\right\}\to \mathcal{R}_p(g)$, by \[\varphi_1(t):=g^{\alpha p^{r}}t^\beta\, ,\]
    which is well defined thanks to Bézout's relation. Moreover, if $t_1,t_2\in \mathcal{R}_p\bigl(g^m\bigr)$, by Lemma~\ref{rootorder}, we have $\Ord(t_i)\, | \, p\Ord(g^m)=p^{r+1}$, for $i\in \{1,2\}$. Thus, if $\varphi_1(t_1)=\varphi_1(t_2)$, we have $t_1^\beta=t_2^\beta$, which implies that \[t_1=t_1^{1-\alpha p^{r+1}}=t_1^{\beta m}=t_2^{\beta m}=t_2^{1-\alpha p^{r+1}}=t_2\, .\]
    Hence, $\varphi_1$ is injective. Moreover, if $h\in \mathcal{R}_p(g)$, we have $\varphi_1(h^m)=g^{\alpha p^r}h^{\beta m}=h^{\alpha p^{r+1} +\beta m}=h$, and $\bigl(h^m\bigr)^p=g^m$, which proves that $\varphi_1$ is bijective.
    
    We now move to the second part of the lemma. We assume that $p\ge 3$; by Bézout's theorem, there exist $u,v\in \Z$ such that $up+2vm=1$. Define \[\mathcal{E}=\left\{(s,t)\in \mathcal{R}_2(g) \times\mathcal{R}_p\bigl(g^m\bigr)\, \colon\,  st=ts\right\}\, ,\]
    and let $\Phi:\mathcal{R}_{2p}(g)\to \mathcal{E}$ be the map defined by $\Phi(h)=\left(h^p,h^{2m}\right)$. The map $\Phi$ is well defined since $(h^{p})^2=g$ and $h^{2mp}=g^m$. Since $\left(a^{p},a^{2m}\right)=\left(b^{p},b^{2m}\right)$ implies that $a=a^{up} a^{2vm}=b^{up}b^{2vm}=b$, the map $\Phi$ is injective. Moreover, if $(s,t)\in \mathcal{E}$, define $h=s^ut^v$. We have $h^{2p}=(s^2)^{up} (t^p)^{2v}=g^{up+2mv}=g$. Using $s^2=g$, $t^p=g^m$, and $st=ts$, we deduce that:
    \[h^{2m}=s^{2mu}t^{2mv}=g^{mu}t^{1-up}=t g^{mu-mu}=t\quad\text{and}\quad h^{p}=s^{up}t^{vp}=s^{1-2mv}g^{mv}=s g^{mv-mv}=s\, .\]
    Thus, $\Phi(h)=(s,t)$. This proves that $\Phi$ is bijective.
    
    To deduce~\eqref{2p-primrootsnumb}, it suffices to prove that $h\in \mathcal{R}_{2p}(g)$ is primitive if and only if $\Phi(h)\in \bigl(\mathcal{R}_2(g)\setminus \langle g\rangle\bigr) \times  \bigl(\mathcal{R}_p(g^m)\setminus\{1\}\bigr)$. Let $h\in \mathcal{R}_{2p}(g)$, and denote $(s,t):=\Phi(h)$. If $h$ is not a primitive $2p$-th root, then by Lemma~\ref{rootorder}, two cases arise. First case: $\Ord(h) \mid p\Ord(g)$. In this case, $s=h^p$ is a trivial square root of $g$; indeed $\Ord(s)\mid \gcd(2\Ord(g),p\Ord(g))=\Ord(g)$. This implies that $s\in \langle g\rangle$. Second case: $\Ord(h)\mid 2\Ord(g)$. In this case, $h^2$ is a trivial $p$-th root of $g$. Thus $\gcd(p,\Ord(g))=1$, which means that $g^m=1$. Hence $h^{2m}\in \langle g^m\rangle =\{1\}$. Conversely, since $ts=st$ we have $\Ord(h)=\Ord(s^u t^v)\mid \gcd (\Ord(s),\Ord(t))$. If $s\in \langle g\rangle$, we have $\gcd(\Ord(s),\Ord(t))\mid p\Ord(g)$, and if $t=1$, we have $\gcd(\Ord(s),\Ord(t))=\Ord(s)\mid 2\Ord(g)$. This finishes the proof.
\end{proof}
With these counting formulas in hand, we can now prove that if $r_2(x)=r_2(y)$ and $r_p(x)=r_p(y)$, then $r_{2p}(y)-r_{2p}(x)$ is either a multiple of $p(p-1)$ or a multiple of $p^2$: 
\begin{lem}\label{difference2proots}
    Let $G$ be a finite group, let $p\ge 3$ be a prime number, and let $(x,y)\in G^2$ such that $\Ord(x)=\Ord(y)$. If $r_2(x)=r_2(y)$ and $r_p(x)=r_p(y)$, then $r_{2p}(x)-r_{2p}(y) \in p(p-1)\Z \cup p^2\Z$. If moreover $\Ord(x)$ and $\Ord(y)$ are even, then $r_{2p}(x)-r_{2p}(y) \in 2p(p-1)\Z\cup 2p^2\Z$.
\end{lem}
\begin{proof}
    For all $g\in G$ with order $\Ord(g)=p^\alpha \ell$ such that $\gcd(p,\ell)=1$, we define \[\mathcal{F}_p(g):=\left\{ t\in \mathcal{R}_p\bigl(g^\ell\bigr)\, \colon\,  tg=gt\right\}\, .\]
    From the equality $r_2(x)r_p(x)=r_2(y)r_p(y)$ and~\eqref{p-rootsnumber}, we deduce 
    \[\sum_{t\in \mathcal{F}_p(x)}\left| \mathcal{R}_2(x)\right|=\sum_{t\in\mathcal{F}_p(y)}\left|\mathcal{R}_2(y)\right|\, .\]
    This, combined with~\eqref{2p-rootsnumb2}, implies that
    \begin{align*}
        r_{2p}(x)-r_{2p}(y)&=\sum_{t\in \mathcal{F}_p(x)}\#\left\{s\in \mathcal{R}_2(x)\, \colon\,  st=ts\right\}-\sum_{t\in \mathcal{F}_p(y)}\#\left\{s\in \mathcal{R}_2(y)\, \colon\,  st=ts\right\}\\
        &=\sum_{t\in \mathcal{F}_p(y)}\#\left\{s\in \mathcal{R}_2(y)\, \colon\,  st\ne ts\right\}-\sum_{t\in \mathcal{F}_p(x)}\#\left\{s\in \mathcal{R}_2(x)\, \colon\,  st\ne ts\right\}
    \end{align*}
    Let $t\in \mathcal{F}_p(x)\setminus\{1\}$. We prove that we have $\#\left\{s\in \mathcal{R}_2(x)\, \colon\,  st\ne ts\right\}\in p\Z$ and if moreover $\Ord(x)$ is even, then $\#\left\{s\in \mathcal{R}_2(x)\, \colon\,  st\ne ts\right\}\in 2p\Z$.
    Assume that $\left\{s\in \mathcal{R}_2(x)\, \colon\,  st\ne ts\right\}\ne \emptyset$. Writing $\Ord(x)=p^rm$ with $\gcd(m,p)=1$, we have $t\in \mathcal{R}_{p}(x^m)$. Since $ts\ne st$ we have that $t\notin \langle s\rangle$, and thus $t\notin\langle x^m \rangle$. This means, by Lemma~\ref{rootorder}, that $\Ord(t)$ is a strict multiple of $p^{r}$ that divides $p^{r+1}$. Hence, $\Ord(t)=p^{r+1}$. We consider the action by conjugation of $\langle t\rangle$ on the set $\left\{s\in \mathcal{R}_2(x)\, \colon\,  st\ne ts\right\}$. This action has no trivial orbits, and the cardinality of each orbit divides $p^{r+1}$. Thus, $p$ divides $\#\left\{s\in \mathcal{R}_2(x)\, \colon\,  st\ne ts\right\}$. If $\Ord(x)$ is even, then by Lemma~\ref{rootorder}, for all $s\in \mathcal{R}_2(x)$, we have $\Ord(s)=2\Ord(x)$. Since $s^{\Ord(x)+1}\in \mathcal{R}_2(x)\setminus\{s\}$, we can write $\{s\in \mathcal{R}_2(x)\, \colon\,  st\ne ts\}$ as a disjoint union of subsets of the form $\left\{s,s^{\Ord(x)+1}\right\}$. This implies that $\#\{s\in \mathcal{R}_2(x)\, \colon\,  st\ne ts\}$ is a multiple of $2p$.
    
    Since $t\in \mathcal{R}_p(x^m)$, two cases arise: if $r=0$ (i.e., $\gcd(\Ord(x),p)=1$), then for all $h\in\langle t\rangle\setminus \{1\} $ we have $h\in \mathcal{F}_p(x)\setminus\{1\}$ and \[\left\{s\in \mathcal{R}_2(x)\, \colon\,  st\ne ts\right\}=\left\{s\in \mathcal{R}_2(x)\, \colon\,  sh\ne hs\right\}.\]
    Since there are $p-1$ elements in $\langle t\rangle\setminus\{1\}$, we have \[\sum_{h\in \langle t\rangle\setminus\{1\}} \#\{s\in \mathcal{R}_2(y)\, \colon\,  sh\ne hs\}\equiv\begin{cases}
        0\pmod{ p(p-1)}\quad&\text{if } \Ord(x)\text{ is odd}\\
        0\pmod{ 2p(p-1)} &\text{otherwise}.
    \end{cases}\]
    Since the latter holds for all $t\in \mathcal{F}_p(x)\setminus\{1\}$, we have that if $r=0$, then \[\sum_{t\in \mathcal{F}_p(x)}\#\left\{s\in \mathcal{R}_2(x)\, \colon\,  st\ne ts\right\}\equiv\begin{cases}
        0\pmod{ p(p-1)}\quad&\text{if } \Ord(x)\text{ is odd}\\
        0\pmod{ 2p(p-1)} &\text{otherwise}.
    \end{cases} \] 
    If $r\ge 1$, the set $\langle t\rangle\cap\mathcal{F}_p(x)=\langle t\rangle\cap\mathcal{R}_p(x^m)$ has cardinality $p$ (more generally, any element in a cyclic $p$-group has either no $p$-th roots or has exactly $p$ $p$-th roots). Similarly as above, for all these $p$ elements $h\in \langle t\rangle\cap\mathcal{F}_p(x) $ we have
    \[\left\{s\in \mathcal{R}_2(x)\, \colon\,  st\ne ts\right\}=\left\{s\in \mathcal{R}_2(x)\, \colon\,  sh\ne hs\right\}\, .\]
    Thus, if $r\ge1$, then
    \[\sum_{t\in \mathcal{F}_p(x)}\#\left\{s\in \mathcal{R}_2(x)\, \colon\,  st\ne ts\right\}=\begin{cases}
        0\pmod{ p^2}\quad&\text{if } \Ord(x)\text{ is odd}\\
        0\pmod{ 2p^2} &\text{otherwise}.
    \end{cases} \]  
    
    Similarly,  \[\sum_{t\in \mathcal{F}_p(y)}\#\left\{s\in \mathcal{R}_2(y)\, \colon\,  st\ne ts\right\}\in\begin{cases}
        p(p-1)\Z\cup p^2\Z\quad&\text{if } \Ord(y)\text{ is odd}\\
        2p(p-1)\Z\cup 2p^2\Z &\text{otherwise}.
    \end{cases} \]  
    This proves the lemma.
\end{proof}
We now use Lemma~\ref{difference2proots} to deduce properties of groups satisfying the property $\mathcal{P}(2p;x,y)$.
\begin{lem}\label{primitive2proots}
    Let $G$ be a finite group. If $G$ satisfies the property $\mathcal{P}(2p;x,y)$, then $\kappa_{2p}(y)\ge p(p-1)$, and if moreover $\Ord(y)$ is even, then $\kappa_{2p}(y)\ge 2p(p-1)$. In particular, $2p\Ord(y)$ divides the order of $G$. Finally, writing $\Ord(x)=p^r m$ with $\gcd(m,p)=1$, there exists $(s,t)\in \mathcal{R}_2(x)\times \mathcal{R}_p(x^{m})$ such that $st\ne ts$.
\end{lem}
\begin{proof}
    By Lemma~\ref{rootorder}, if $g\in \mathcal{R}_{2p}(x)$ we have $\Ord(g)\in\{\Ord(x),2\Ord(x),p\Ord(x),2p\Ord(x)\}$. Since $\Ord(x)=\Ord(y)$, Remark~\ref{trivialroot} shows that for all $n\ge 1$ \[\#\{ g\in \mathcal{R}_n(x)\, \colon\,  \Ord(g)=\Ord(x)\}=\#\{ g\in \mathcal{R}_n(y)\, \colon\,  \Ord(g)=\Ord(y)\}\, .\]
    Furthermore, if $g\in \mathcal{R}_{2p}(x)\setminus \langle x\rangle$ then $\Ord(g)=2\Ord(x)$ if and only if $g^2$ is the trivial $p$-th root of $x$. Thus, by Lemma~\ref{rational-classes}, the number of $2p$-th roots of order $2\Ord(x)$ is either $0$ if $p\mid \Ord(x)$ or is exactly the number of non-trivial square roots of $x$. Similarly, the number of $2p$-th roots of order $p\Ord(x)$ is either $0$ if $\Ord(x)$ is even, or is exactly the number of non-trivial $p$-th roots of $x$. This shows that \begin{align*}
        \#\{ g\in \mathcal{R}_{2p}(x)\, \colon\,  \Ord(g)=2\Ord(x)\}&=\#\{ g\in \mathcal{R}_{2p}(y)\, \colon\,  \Ord(g)=2\Ord(y)\}\\
        \#\{ g\in \mathcal{R}_{2p}(x)\, \colon\,  \Ord(g)=p\Ord(x)\}&=\#\{ g\in \mathcal{R}_{2p}(y)\, \colon\,  \Ord(g)=p\Ord(y)\}\, .
    \end{align*}
    This implies that $\kappa_{2p}(y)-\kappa_{2p}(x)=r_{2p}(y)-r_{2p}(x)$. The lemma is thus a consequence of Lemma~\ref{difference2proots}. For the last statement, by~\eqref{2p-rootsnumb2} and~\eqref{p-rootsnumber}, defining $\mathcal{F}_p(x):=\{t\in \mathcal{R}_p(x^m)\, :\, tx=xt\}$ (and defining similarly $\mathcal{F}_p(y)$), we have
    \[r_{2p}(y)\le r_2(y)\bigl| \mathcal{F}_p(y)\bigr|=  r_2(y)r_p(y)\, ,\]
    where we used the trivial fact that if $t$ commutes with $s\in \mathcal{R}_2(y)$, then it commutes with $s^2=y$. Assume for the sake of a contradiction that for all $(s,t)\in \mathcal{R}_2(x)\times \mathcal{R}_p(x^{m})$ we have $st=ts$. Then, by $\eqref{2p-rootsnumb2}$, $r_{2p}(x)=r_2(x)\bigl| \mathcal{F}_p(x)\bigr|=r_2(x)r_p(x)$. This contradicts the fact that $r_{2p}(x)<r_{2p}(y)$, $r_2(x)=r_2(y)$ and $r_{p}(x)=r_{p}(y)$.
\end{proof}
\begin{rk}\label{noncommutativity-normalpsylow}
    If $G$ has a normal $p$-Sylow $P=\langle t\rangle$ of order $p$, Lemma~\ref{primitive2proots} implies that $\Ord(x)$ is coprime to $p$, that $x$ commutes with $t$, and that there exists a square root $s\in\mathcal{R}_2(x)$ that does not commute with $t$. Indeed, since $2p\Ord(y)$ divides $|G|$, $p$ is coprime to $\Ord(y)=\Ord(x)$. Moreover, since $y$ admits a primitive $2p$-th root, there exists a cyclic group containing an element of order $p$ and $y$, meaning that $y$ commutes with $t$, since the only subgroup of $G$ containing an element of order $p$ is $P$. This implies that $y$ has a non-trivial $p$-th root, and thus $x$ has a non-trivial $p$-th root, $x$ commutes with an element of order $p$, and thus it commutes with $t$. Finally, since there exists $(s,t')\in \mathcal{R}_2(x)\times \mathcal{R}_p(1)$ such that $st'\ne t's$, we have $t'\ne 1$.  Since $t'\in \langle t\rangle$, we have $st\ne ts$.
\end{rk}

\subsection{Reducing to the normal $p$-Sylow case}
Let us start with stating the following useful elementary lemma (see~\cite{Berkovich}*{Introduction, exercice 10}):
\begin{lem}\label{groupproductorder}
    Let $H_1,H_2$ be subgroups of a finite group $G$. We have \[|H_1H_2|= \frac{|H_1|\cdot |H_2|}{|H_1\cap H_2|}\, .\]
\end{lem}
We now use Lemma~\ref{groupproductorder} to deduce some properties on $|G|$, when $G$ satisfies $\mathcal{P}(2p;x,y)$, with $\Ord(x)$ being a prime power:
\begin{lem}\label{q-Sylow-divisibility}
    Let $G$ be a finite group satisfying the property $\mathcal{P}(2p;x,y)$. Assume that $\Ord(x)$ is a power of a prime $q$. Then $q^2$ divides $|G|$, and if moreover $q\in\{2,p\}$, then $q^4$ divides $|G|$.
\end{lem}
\begin{proof}
    Let $Q\subset G$ be a $q$-Sylow of $G$. $Q$ contains a conjugate of $x$ and a conjugate of $y$. Up to replacing $x$ and $y$ with these conjugates (see Corollary~\ref{invarianceofpropertyP}), we may assume that $x,y\in Q$. Since $\langle x\rangle \ne \langle y \rangle $, we have, by Lemma~\ref{groupproductorder}, $|\langle x\rangle \langle y\rangle|\ge q^2$. Hence $q^2$ divides $|Q|$, thus $q^2$ divides $|G|$. If moreover $q\in\{2,p\}$, then by Lemma~\ref{primitive2proots}, $x$ and $y$ have non-trivial $q$-th roots in $G$, say $t_x$ and $t_y$. Moreover, $Q$ contains conjugates of $t_x$ and $t_y$. Up to replacing $t_x$ and $t_y$ with these conjugates, we may assume $t_x,t_y\in Q$. As $x\notin \langle t_y\rangle$, we have $\langle x\rangle \not \subset \langle t_x \rangle \cap \langle t_y\rangle $, and since both $\langle x\rangle$ and $\langle t_x \rangle \cap \langle t_y\rangle$ are subgroups of the cyclic $q$-group $\langle t_x\rangle$, we have $\langle t_x \rangle \cap \langle t_y\rangle \subset \langle x\rangle$ (here we used the fact that subgroups of a cyclic $q$-group form a chain). Similarly, $\langle t_x \rangle \cap \langle t_y\rangle \subset \langle y\rangle$. Thus, $\langle t_x \rangle \cap \langle t_y\rangle = \langle x\rangle\cap \langle y \rangle$. Hence, by Lemma~\ref{groupproductorder}, we have \[|\langle t_x\rangle\langle t_y\rangle|\ge \frac{|\langle t_x\rangle|\cdot|\langle t_y\rangle|}{|\langle x\rangle\cap\langle y\rangle|}=q^2 \frac{|\langle x\rangle|\cdot|\langle y\rangle|}{|\langle x\rangle\cap \langle y\rangle|}\ge q^4\, .\]
    Hence $q^4$ divides $|Q|$, thus it divides $|G|$.
\end{proof}

The following lemma is useful when the prime $p$ divides $\Ord(x)$ or when $p$ is small.
    \begin{lem}\label{lowerboundsmallp}
        Let $G$ be a group satisfying the property $\mathcal{P}(2p;x,y)$. We have $|G|\ge 8p\Ord(y)$, and if moreover $\langle x\rangle \cap\langle y\rangle=\{ 1\}$ (in particular, when $\Ord(y)$ is a prime), we have $|G|\ge 4p\Ord(y)^2$.
    \end{lem}
    \begin{proof}
        Let $g_1,g_2\in G$ be primitive $2p$-th roots of $y$ such that $\langle g_1\rangle\ne \langle g_2\rangle$. We have that $\frac{|\langle g_1\rangle \cap\langle g_2\rangle |}{\Ord(y)}$ is not divisible by $2p$, otherwise $|\langle g_1\rangle \cap\langle g_2\rangle |\ge 2p\Ord(y)=\Ord(g_1)=\Ord(g_2)$ which would imply $\langle g_1\rangle=\langle g_1\rangle \cap\langle g_2\rangle =\langle g_2\rangle$. 
        
        Let $q\in \{2,p\}$ be such that $q$ does not divide $\frac{|\langle g_1\rangle \cap\langle g_2\rangle |}{\Ord(y)}$, and let us write $\Ord(x)=q^r \ell$ where $r\ge 0$ and $\gcd(q,\ell)=1$. Let $v\in \{2,p\}\setminus\{q\}$ be the second prime and let $c$ be a non-trivial $q$-th root of $x$. We have that $g_1^{v\ell}$ and $g_2^{v\ell}$ are of order $q^{r+1}$ and are non-trivial $q$-th roots of $y^\ell$. Moreover, since $\langle g_1^{v\ell}\rangle\cap \langle g_2^{v\ell}\rangle \subset \langle g_1\rangle\cap \langle g_2\rangle$, $q^{r+1}$ does not divide $| \langle g_1^{v\ell}\rangle\cap \langle g_2^{v\ell}\rangle  |$. Thus, $\langle g_1^{v\ell}\rangle\ne \langle g_2^{v\ell}\rangle$. Hence, we have $\langle c\rangle \ne \langle g_i^{v\ell}\rangle$ for some $i\in \{1,2\}$. This proves that $|\langle c\rangle \cap \langle g_i\rangle|$ is not divisible by $q^{r+1}$, and since it divides $\Ord(c)=q^{r+1}\ell$, we deduce that $|\langle c\rangle \cap \langle g_i\rangle|$ divides $q^r\ell=\Ord(x)=\Ord(y)$. Combining this together with the fact that any subgroup of $\langle c\rangle $ of order dividing $\Ord(x)$ is contained in $\langle x\rangle$, and any subgroup of $\langle g_i\rangle$ of order dividing $\Ord(y)$ is contained in $\langle y\rangle$, we conclude that $\langle c\rangle \cap \langle g_i\rangle=\langle x\rangle \cap \langle y\rangle$. This, combined with Lemma~\ref{groupproductorder}, proves that \[|G|\ge|\langle c\rangle\langle g_i\rangle|\ge \frac{\Ord(c)\Ord(g_i)}{|\langle x\rangle\cap \langle y\rangle|}=2p q \frac{\Ord(x)}{|\langle x\rangle\cap \langle y\rangle|}\Ord(y)\ge 8p\Ord(y).\]
        If $\langle x\rangle \cap \langle y\rangle=\{1\}$, then $\Ord(x)/|\langle x\rangle\cap \langle y\rangle|=\Ord(x)$, which yields $|G|\ge 2pq \Ord(y)^2\ge 4p\Ord(y)^2$.
        This proves the lemma.

    \end{proof}

The following lemma reduces the proof of Theorem~\ref{lowerbound} to the case where the $p$-Sylow is normal and has order $p$:
\begin{lem}\label{nonnormalpsylow}
    Let $G$ be a finite group satisfying $\mathcal{P}(2p;x,y)$. If $p^2$ divides $|G|$ or if $G$ has non-normal $p$-Sylow subgroups, then $|G|\ge \min\bigl(16p^2, 4p^2\Ord(y)\bigr)$. 
\end{lem}
\begin{proof}
     Let $t_y$ be a primitive $2p$-th root of $y$. If $\Ord(y)\ge 2p$, we have, by Lemma~\ref{lowerboundsmallp}, $|G|\ge 8p\Ord(y)\ge 16p^2$. If $\Ord(y)=p$, then Lemmas~\ref{q-Sylow-divisibility} and~\ref{primitive2proots} imply that $|G|\ge 2p^4\ge 18p^2$. We may assume that $\gcd(p,\Ord(y))=1$.
     
     If $p^2$ divides $|G|$, we have, by Lemma~\ref{primitive2proots}, $2\Ord(y)p^2$ divides $|G|$. We now prove that we cannot have $|G|=2p^2\Ord(y)$, which proves that $|G|\ge 4p^2\Ord(y)$. Assume for the sake of a contradiction that $|G|=2p^2\Ord(y)$. Let $P$ be a $p$-Sylow of $G$ and let $s_y$ be a non-trivial square root of $y$. We have, because of the equality of cardinalities, $G=P\langle s_y\rangle$. Since $P$ has order $p^2$, it is isomorphic to either $C_{p^2}$ or $C_p\times C_p$. In both cases any element of order $p$ that commutes with $s_y$ is central and thus commutes with all square roots of $x$. Since this holds replacing $s_y$ by any other non-trivial square root of $y$, by~\eqref{2p-rootsnumb2}, this implies that $r_{2p}(x)\ge r_{2p}(y)$, which contradicts the fact that $G$ satisfies the property $\mathcal{P}(2p;x,y)$.
    
    Let us assume that $G$ has a non-normal $p$-Sylow of order $p$. Since $\Ord(x)=\Ord(y)$ and $r_p(x)=r_p(y)$, $x$ admits a primitive $p$-th root. By~\eqref{p-rootsnumber}, there exists $t_x\in G$ of order $p$ such that $t_x x=xt_x$. Let $s\in G$ be a primitive $2p$-th root of $y$. If $t_x\notin \langle s\rangle$, then $\langle t_x x\rangle\cap \langle s\rangle=\langle x\rangle\cap\langle y\rangle$, indeed any subgroup of $\langle t_x x\rangle$ (resp. $\langle s\rangle$) whose order is coprime to $p$ is a subgroup of $\langle x\rangle$ (resp. $\langle s^p\rangle$); moreover, if the order of a subgroup of $\langle s^p\rangle$ divides the order of $y$, then it is a subgroup of $\langle y\rangle$. Thus, by Lemma~\ref{groupproductorder}, we have \[|G|\ge\bigl| \langle s, t_x x\rangle\bigr|\ge \bigl|\langle s\rangle \langle t_x x\rangle\bigr|\ge \frac{2p^2\Ord(x)^2}{|\langle x\rangle\cap\langle y\rangle|}\ge4p^2\Ord(x)\, . \]\end{proof}

\subsection{The Schur--Zassenhaus theorem}
    In what follows, the Schur--Zassenhaus theorem will play a major role. We state it in the following particular form:
    \begin{thm}[Schur--Zassenhaus]
    Let $G$ be a finite group with a normal $p$-Sylow $P$. Then $P$ admits a complement $H$ in $G$; that is, there exists a subgroup $H$ in $G$ such that $P\cap H=\{1\}$ and $PH=G$. Moreover, a subgroup $H'$ of $G$ is a complement of $P$ in $G$ if and only if it is a conjugate of $H$ in $G$ by an element of $P$ (in other words, there exists $g\in P$ such that $H'=gHg^{-1}$). 
    \end{thm}
    We have the following useful consequence:
    \begin{cor}\label{SZapplication}
        Let $G$ be a finite group with a normal $p$-Sylow $P\subset G$. If $N\subset G$ is a subgroup of $G$ such that $\gcd(|N|,p)=1$, then there exists a complement $H$ of $P$ in $G$ containing $N$. In particular, if $G$ satisfies $\mathcal{P}(2p;x,y)$, where $\gcd(p,\Ord(x))=1$, then $G$ satisfies $\mathcal{P}(2p;x',y')$ where $x',y'$ are both elements of a complement $H$ of $P$ in $G$.
    \end{cor}
    \begin{proof}
        Let $H$ be a complement of $P$ in $G$. Define $H'=H\cap PN$; we prove that $H'$ is a complement of $P$ in $PN$. It is clear that $H'\cap P=\{1\}$ and that $PH'\subset PN$. Let $g=tn$, where $(t,n)\in P\times N$. There exists $(t',h)\in P\times H$ such that $n=t'h$. Note that $h=(t')^{-1}n\in PN$, which implies that $h\in H'$. Thus, $g=t t' h\in PH'$. Hence $PN=PH'$, which proves the claim. By the Schur--Zassenhaus theorem, there exists $a\in PN$ such that $N=aH'a^{-1}$. Thus, $aHa^{-1}$ is a complement of $P$ in $G$ that contains $N$.
        
        Let us assume that $G$ satisfies $\mathcal{P}(2p;x,y)$ and let $H$ be a complement of $P$ in $G$. We have that $x$ and $y$ each belong to a complement of $P$ in $G$. Since, by the Schur--Zassenhaus theorem, all these complements are conjugate, there exist elements $x',y'\in H$ that are conjugate to $x$ and $y$, respectively. By Corollary~\ref{invarianceofpropertyP}, we deduce that $G$ satisfies the property $\mathcal{P}(2p;x',y')$.
    \end{proof}
   The following proposition transfers the root count from $G$ to $H$:
    \begin{prop}\label{rootcount-in-H}
        Let $G$ be a finite group with a normal $p$-Sylow $P=\langle t\rangle$ of order $p$, and let $H$ be a complement of $P$ in $G$. Let $g\in H\cap\mathcal{C}_G(t)$, where $\mathcal{C}_G(t)$ is the centralizer of $t$ in $G$. We have 
        \begin{equation}\label{square-rootcountinH}
            r_2^G(g)=\bigl|\mathcal{R}_2^H(g)\cap \mathcal{C}_G(t)\bigr|+p\cdot\bigl| \mathcal{R}_2^H(g)\setminus\mathcal{C}_G(t)  \bigr|  \, ,
        \end{equation}
        and we also have \begin{equation}\label{2prootcountinH}
            r_{2p}^G(g)=p\cdot r_2^H(g)\quad\text{and}\quad \kappa_{2p}^G(g)=(p-1)\kappa_2^{H\cap \mathcal{C}_G(t)}(g) .
        \end{equation}
    \end{prop}
    \begin{proof}
        Let $g\in H\cap\mathcal{C}_G(t)$. If $\ell=t^k h\in \mathcal{R}_2^G(g)$, where $h\in H$ and $0\le k<p$, we have $g=\ell^2=(t^k\cdot ht^kh^{-1})h^2\in H$, and since $t^k\cdot ht^kh^{-1}\in P$, then $ht^kh^{-1}=t^{-k}$ and $h^2=g$. This shows that \begin{equation}\label{setsquare-rootsG-H}
            \mathcal{R}_2^G(g)=\{t^k h\, \colon\,  0\le k<p,\  ht^{k}h^{-1}=t^{-k},\text{ and } h\in \mathcal{R}_2^H(g)\}\, .
        \end{equation}
        Thus, \begin{align*}
            r_2^G(g)&=\sum_{h\in \mathcal{R}_2^H(g)}\bigl|\{0\le k<p \colon ht^{k}h^{-1}=t^{-k}\}\bigr|\\
            &=\sum_{h\in \mathcal{R}_2^H(g)\cap \mathcal{C}_G(t)}\bigl|\{0\le k<p \colon  ht^{k}h^{-1}=t^{-k}\}\bigr|+\sum_{h\in \mathcal{R}_2^H(g)\setminus\mathcal{C}_G(t) }\bigl|\{0\le k<p \colon ht^{k}h^{-1}=t^{-k}\}\bigr|\, .
        \end{align*}
        If $h\in \mathcal{R}_2^H(g)\cap \mathcal{C}_G(t)$, since $t^k=t^{-k}\iff k=0$, we have $\bigl|\{0\le k<p\, \colon\,  ht^{k}h^{-1}=t^{-k}\}\bigr|=1$. If $h\in \mathcal{R}_2^H(g)\setminus\mathcal{C}_G(t) $, since $ht\ne th$ and $h^2t=th^2$, $(\mathrm{Int}_h)_{|P}$ is an automorphism of $P$ of order $2$, thus $(\mathrm{Int}_h)_{|P}=\mathrm{Inv}_P$, hence $\bigl|\{0\le k<p\, \colon\,  ht^{k}h^{-1}=t^{-k}\}\bigr|=p$. This proves the proposition. We now move to the second equality. We use~\eqref{2p-rootsnumb2} together with~\eqref{setsquare-rootsG-H} to deduce that \[r_{2p}^G(g)=\bigl|\bigl\{ (k,k',h)\, \colon\,  0\le k,k'<p,\ h\in \mathcal{R}_{2}^H(g), \ ht^kh^{-1}=t^k,\ ht^{k'}h^{-1}=t^{-k'}\bigr\}\bigr|\, .\]
        If $h\in \mathcal{R}_2^H(g)\cap \mathcal{C}_G(t)$, then $\bigl|\{(k,k')\, \colon\,  ht^kh^{-1}=t^k,\ ht^{k'}h^{-1}=t^{-k'}\}\bigr|=p$, and if $h\in \mathcal{R}_2^H(g)\setminus \mathcal{C}_G(t) $, then by the above $(\mathrm{Int}_h)_{|P}=\mathrm{Inv}_P$, thus $\bigl|\{(k,k')\, \colon\,  ht^kh^{-1}=t^k,\ ht^{k'}h^{-1}=t^{-k'}\}\bigr|=p$. This proves that \[r_{2p}^G(g)=p \cdot|\mathcal{R}_2^H(g)\cap \mathcal{C}_G(t)|+p\cdot |\mathcal{R}_2^H(g)\setminus \mathcal{C}_G(t)|=p\cdot r_2^H(g)\, . \]
        For the last equality, we start by noting that if $s=t^k h\in G$, with $0\le k<p$ and $h\in H$, then $s\in \langle g\rangle$ if and only if $h\in \langle g\rangle$.  Combining this with~\eqref{setsquare-rootsG-H} and~\eqref{2p-primrootsnumb}, we deduce that
        \[\kappa_{2p}^G(g)=\bigl|\bigl\{(k,k',h)\, \colon\, 0\le k,k'<p,\, k\ne 0,\,  h\in \mathcal{R}_2^H(g)\setminus\langle g\rangle, ht^k=t^kh,\, ht^{k'}=t^{-k'}h\bigr\}\bigr|\, . \]
        The condition $ht^{k}=t^kh$ for $0<k<p$ forces $k'=0$ and $h\in H\cap \mathcal{C}_G(t)$. Hence, 
        \[  \kappa_{2p}^G(g)=\bigl|\bigl\{ (k,h)\, \colon \, 1\le k\le p-1,\, h\in \mathcal{R}_{2}^{H\cap \mathcal{C}_G(t)}(g)\setminus\langle g\rangle\bigr\}\bigr|=(p-1)\kappa_2^{H\cap \mathcal{C}_G(t)}(g)\, .  \]
        \end{proof}

\subsection{Wall's theorem for groups with high proportion of involutions}
It is a classical result that a finite group in which all non-trivial elements are involutions is elementary abelian. There is a more general elementary result of the same type regarding groups with a high proportion of involutions:

\begin{lem}\label{3-4thinvolutions}
    Let $G$ be a finite group such that $r_2^G(1) > 3|G|/4$. Then $G$ is an elementary abelian $2$-group.
\end{lem}
\begin{proof}
    Although this is a well-known result, we recall the standard argument for completeness. Let $I := \mathcal{R}_2^G(1)$ and let $x \in I$. Since $|I| > 3|G|/4$, we have
    \[ |I \cap xI| = |I| + |xI| - |I \cup xI| > \frac{6|G|}{4} - |G| = \frac{|G|}{2}\, .\]
    The centralizer $\mathcal{C}_G(x)$ of $x$ contains the intersection $I \cap xI$; indeed, if $y \in I \cap xI$, we can write $y = xz$ with $y^2 = z^2 = 1$, yielding $[x,y] = xyxy = (xy)^2 = z^2 = 1$, which means $x$ and $y$ commute. Thus, $|\mathcal{C}_G(x)| > |G|/2$. Since the order of a subgroup must divide the order of the group, we deduce that $\mathcal{C}_G(x) = G$. This proves that any element of order $2$ is central, implying that $I$ is a subgroup of $G$. Since it contains strictly more than half the elements of $G$, we must have $I = G$, meaning $G$ is an elementary abelian $2$-group.
\end{proof}
A natural generalization is to ask what can be said about finite groups satisfying the weaker bound $r_2^G(1) > |G|/2$. Groups satisfying this property were studied by Wall~\cite{Wall}, and in this paper, they are referred to as \emph{Wall groups}. The following theorem classifies all isomorphism classes of such groups.
\begin{thm}[Wall's Classification Theorem~\cite{Wall}, 1970]\label{Wallthm}
    Let $G$ be a finite group satisfying $r_2^G(1) > |G|/2$. Then $G \cong E \times G_0$, where $E$ is an elementary abelian $2$-group and $G_0$ is one of the following groups:
    \begin{enumerate}
        \item \emph{Type \rom{1}:} $G_0 = A_0 \rtimes \langle \tau \rangle$, where $A_0$ is an abelian group and $\tau$ is an element of order $2$ acting on $A_0$ by inversion (yielding the generalized dihedral group).
        \item \emph{Type \rom{2}:} $G_0 = D_8 \times D_8$. In this case, we have $\mathcal{Z}(G_0) = \langle(b,1), (1,b)\rangle$, where $b$ is the central involution of $D_8$.
        \item \emph{Type \rom{3}:} $G_0 = D_8^r/N$, where $N = \langle (b,b,1,\dots,1), (b,1,b,1,\dots,1), \dots, (b,1,\dots,1,b) \rangle$. This is the product of $r$ copies of $D_8$ amalgamated at the center. In this case, $\mathcal{Z}(G_0) = \{1,c\}$ with $\Ord(c)=2$, and $G_0$ can be presented as follows:
        \[ G_0 = \bigl\langle c, x_1, y_1, \dots, x_r, y_r\ :\ [x_i,y_i]=c \bigr\rangle\, . \]
        \item \emph{Type \rom{4}:} $G_0 = A_0 \rtimes \langle c \rangle$ where $A_0 = E_1 \times E_2 \cong (C_2)^{2r}$. Here, $E_1 \cap E_2 = \{1\}$, $E_1 \cong (C_2)^r$ admits $(x_1,\dots,x_r)$ as a basis, and $E_2 \cong (C_2)^r$ admits $(y_1,\dots,y_r)$ as a basis, all satisfying the relations $[c,x_i]=y_i$. In this case, $\mathcal{Z}(G_0) = E_2$, and $G_0$ can be presented as:
        \[ G_0 = \bigl\langle c, x_1, y_1, \dots, x_r, y_r\ :\ [c,x_i]=y_i \bigr\rangle\, . \]
    \end{enumerate}
    If $G_0$ is of type $i$, for $i \in \{$\rom{1}$,\dots,$\rom{4}$\}$, we say that $G$ is a Wall group of type $i$.
\end{thm}
In what follows, we will use the notation $r_2(x):=r_2^G(x)$, for $x\in G$. The following lemma enables us to reduce the case $\frac{r_2(1)}{|G|}>\frac{5}{8}$ to $G$ being a Wall group either of type~\rom{1} that is not a $2$-group, or of type~\rom{3}:
    \begin{lem}\label{wallsquare-rootratio}
        Assume that $G$ is a Wall group satisfying one of the following conditions: \begin{enumerate}
            \item $G$ is of type \rom{1} and $|G|$ is not a power of $2$.
          \item $G$ is of type \rom{2}.
          \item $G$ is of type \rom{4} and $r\ge 2$.
        \end{enumerate}
        Then, \[\frac{r_2(1)}{|G|}\le \frac{5}{8}\, .\]
    \end{lem}
    \begin{proof}
        Let us first note that if $G\cong E\times G_0$, where $E$ is an elementary abelian group, then $\frac{r_2^G(1)}{|G|}=\frac{r_2^{G_0}(1)}{|G_0|}$. Therefore, we may assume that $E=\{1\}$ and that $G$ has no direct factor of order $2$. If $G$ is of type \rom{1} and $|G|$ is not a power of $2$, then $G\cong A_0\rtimes\langle \tau\rangle=(A_0\times\{1\})\cup (A_0\times\{\tau\})$, and for all $a\in A_0\times\{\tau\}$ we have $a^2=1$, thus $r_2^G(1)=r_2^{A_0}(1)+|A_0|$. Since $A_0$ is abelian, we have $r_2^{A_0}(1)=\bigl|A_0[2]\bigr|$, where $A_0[2]$ is the $2$-torsion of $A_0$. If this $2$-torsion is trivial, then $r_2^G(1)=|A_0|$. In this case, $\frac{r_2^G(1)}{|G|}=\frac{1}{2}<\frac{5}{8}$. If $A_0[2]$ is non-trivial, then the $2$-Sylow $S$ of $A_0$ is non-trivial; moreover, $S\cong \prod_{i=1}^n C_{2^{k_i}}$ for some $1\le k_1\le\dots\le k_n $. Since $G$ has no direct factor of order $2$, for all $1\le i\le n$ we have $k_i\ge 2$. Moreover, $A_0[2]=S[2]\cong \prod_{i=1}^n C_2=C_2^n$. Thus, $\bigl| A_0[2]\bigr|/ \bigl| S\bigr| \le 1/2$, and since $G$ is not a $2$-group, we have $|A_0|\ge 3|S|$. We deduce that 
        \[\frac{r_2^G(1)}{|G|}=\frac{\bigl|A_0[2]\bigr|+|A_0|}{2|A_0|}\le \frac{1}{2}+\frac{\bigl|A_0[2]\bigr|}{6|S|}\le \frac{1}{2}+\frac{1}{12}=\frac{7}{12}<\frac{5}{8}\, .\]
        If $G$ is of type \rom{2}, then $G\cong D_8\times D_8$; thus $r_2(1)=\bigl(r_2^{D_8}(1)\bigr)^2=36$. Hence \[\frac{r_2^G(1)}{|G|}=\frac{36}{64}=\frac{9}{16}<\frac{5}{8}\, .\]
        If $G$ is of type \rom{4} and $r\ge 2$, we have $G\cong A_0\rtimes \langle c\rangle$ where $A_0\cong (C_2)^{2r}=E_1E_2$, with $E_1\cap E_2=\{1\}$, $E_1\cong (C_2)^r$ admits $(x_1,\dots,x_r)$ as a basis, and $E_2\cong (C_2)^r$ admits $(y_1,\dots,y_r)$ as a basis, with $[c,x_i]=y_i$ (where, for convenience, we see $G= A_0\rtimes \langle c\rangle$ as an inner semi-direct product). Let $a=\prod_{i=1}^r x_i^{\alpha_i} y_i^{\beta_i}\in A_0$, where $\alpha_i,\beta_i\in \{0,1\}$, such that $(ca)^2=1$. We have $a=cac=\prod_{i=1}^r (cx_ic)^{\alpha_i}y_i^{\beta_i}=\prod_{i=1}^r x_i^{\alpha_i}y_i^{\alpha_i+\beta_i}=a\prod_{i=1}^ry_i^{\alpha_i}$. Thus, $\prod_{i=1}^ry_i^{\alpha_i}=1$. Hence, for all $1\le i\le r$ we have $\alpha_i=0$, which means $a\in E_2$. Conversely, if $a\in E_2$, then it commutes with $c$. Thus, $(ca)^2=1$. Hence $r_2^G(1)=|A_0|+|E_2|=2^{2r}+2^r$. Thus, since $r\ge2$, we have \[\frac{r_2^G(1)}{|G|}=\frac{1}{2}+\frac{1}{2^{r+1}}\le \frac{1}{2}+\frac{1}{8}=\frac{5}{8}\, .\]
        This proves the lemma.
    \end{proof}
    \begin{rk}
        Let us note that if $G$ is a Wall group of type~\rom{4} with $r=1$, then $G$ is also of type~\rom{1}. Indeed, in this case, $G_0\cong D_8$.
    \end{rk}
    \begin{lem}\label{Wallcentralelement}
        Let $G$ be a Wall group. If $x\in G$ is an element of order $2$ that is the square of some element in $G$, then $x\in \mathcal{Z}(G)$.
    \end{lem}
    \begin{proof}
        Let $G\cong E\times G_0$ be a Wall group where $E$ is an elementary abelian $2$-group. If $(e,h)\in E\times G_0$, the equality $x=(e,h)^2$ holds if and only if $x=(1,h)^2$. Thus, since $E\times \{1\}\subset \mathcal{Z}(G)$, we may assume that $G=G_0$ has no direct factor of order $2$. We will distinguish between all different types of Wall groups.
        
        If $G$ is of type~\rom{1}, then $G=A_0\rtimes\langle \tau\rangle$. If $a^2=x\ne 1$, since all elements of $A_0\times\{\tau\}$ are of order $2$, we must have $a\in A_0\times\{1\}$. Hence $x\in A_0\times\{1\}$ commutes with all elements of the abelian group $A_0\times\{1\}$, and since $\Ord(x)=2$, $x$ commutes with $(1,\tau)$. This proves that $x\in \mathcal{Z}(G)$.
        
        If $G$ is of type~\rom{2}, meaning $G=D_8\times D_8$, then $x=(\alpha,\beta)^2=(\alpha^2,\beta^2)$, for some $(\alpha,\beta)\in D_8^2$. Since $(\alpha^2,\beta^2)\in \mathcal{Z}(D_8)\times\mathcal{Z}(D_8)$, we have $x\in \mathcal{Z}(G)$.
        
        If $G$ is of type~\rom{3}, then $\mathcal{Z}(G)=\{1,c\}$. Moreover, the unique element of order $2$ that is a square is $c$. Indeed, if $h\in G$, we can write $h=c^\gamma\prod_{i=1}^r x_i^{\alpha_i}y_i^{\beta_i}$, with $\alpha_i,\beta_i,\gamma\in \{0,1\}$. Thus, $h^2=\bigl(\prod_{i=1}^r x_i^{\alpha_i}y_i^{\beta_i}\bigr)^2$. For all $1\le i\le r$, we have \[\bigl(x_i^{\alpha_i} y_i^{\beta_i}\bigr)^2=\begin{cases}
            1\quad &\text{if}\quad \alpha_i=0\ \text{or}\ \beta_i=0\\
            c   &\text{if}\quad \alpha_i=\beta_i=1
        \end{cases}\, .\]
        Conversely, we have $(x_1y_1)^2=c$. This proves the claim. We now assume that $G$ is of type~\rom{4}. Suppose $x=g^2$ with $g=\bigl(\prod_{i=1}^r x_i^{\alpha_i}y_i^{\beta_i}\bigr)c^\gamma$, where $\alpha_i,\beta_i,\gamma\in \{0,1\}$. We see that $\gamma\ne 0$ (otherwise $g^2=1$), hence $\gamma=1$. Moreover, if $a=\prod_{i=1}^r x_i^{\alpha_i} y_i^{\beta_i}$, we have $cac=\prod_{i=1}^r (cx_ic)^{\alpha_i}y_i^{\beta_i}=\prod_{i=1}^r x_i^{\alpha_i}y_i^{\alpha_i+\beta_i}=a\prod_{i=1}^ry_i^{\alpha_i}$. Thus, $g^2=a\cdot(cac)=\prod_{i=1}^r y_i^{\alpha_i}\in \mathcal{Z}(G)$. This proves the lemma.
    \end{proof}

\subsection{Proof of Theorem~\ref{lowerbound}}\label{ptlb}
In this subsection, we prove Theorem~\ref{lowerbound}. More precisely, we will prove by contradiction that there exists no group $G$ of order $|G|<8p^2$ and satisfying the property $\mathcal{P}(2p;x,y)$.
\begin{lem}\label{square-rootsinN}
    Let $G$ be a finite group of order $|G|<8p^2$ satisfying the property $\mathcal{P}(2p;x,y)$. Then, the $p$-Sylow $P$ of $G$ is normal and has order $p$. Furthermore, there exists a complement $H$ of $P$ in $G$ and conjugates $x',y'\in H$ of $x$ and $y$ respectively such that $G$ satisfies $\mathcal{P}(2p;x',y')$. Let $t$ be a generator of $P$ and define $N:=\mathcal{C}_G(t)\cap H$. We have $x',y'\in N$ and 
    \[\bigl| \left(\mathcal{R}_2^N\bigl( y'\bigr)\cup \mathcal{R}_2^N\bigl( (y')^{-1}\bigr) \right)\setminus \langle y'\rangle\bigr| \ge 2p\, . \]
\end{lem}
\begin{proof}
    By Lemma~\ref{nonnormalpsylow}, the $p$-Sylow $P$ of $G$ is normal and has order $p$. We write $P=\langle t\rangle$. Let $H$ be a complement of $P$ in $G$. By Corollary~\ref{SZapplication}, there exist conjugates $x'$ and $y'$ of $x$ and $y$, respectively, such that $x',y'\in H$. Without loss of generality, we may assume that $x=x'$ and $y=y'$. Since $G$ satisfies $\mathcal{P}(2p;x,y)$, $x$ and $y$ admit primitive $p$-th roots, which implies they commute with $t$. Thus, $x,y\in N = \mathcal{C}_G(t)\cap H$.

    By Proposition~\ref{rootcount-in-H}, we have $\kappa_{2p}^G(y) = (p-1)\kappa_2^N(y)$. If $\Ord(y)=2$, Lemma~\ref{primitive2proots} implies that $\kappa_{2p}^G(y) \ge 2p(p-1)$. Consequently, $(p-1)\kappa_2^N(y) \ge 2p(p-1)$, yielding $\kappa_2^N(y) \ge 2p$. Since $\kappa_2^N(y) = \bigl| \mathcal{R}_2^N(y) \setminus \langle y \rangle \bigr|$, the desired inequality immediately follows. If $\Ord(y)\ge 3$, we have $y \ne y^{-1}$, meaning the sets $\mathcal{R}_2^N(y)$ and $\mathcal{R}_2^N(y^{-1})$ are disjoint.  Lemma~\ref{primitive2proots} gives $\kappa_{2p}^G(y) \ge p(p-1)$, which implies $\kappa_2^N(y) \ge p$. Similarly, we also have  $\kappa_2^N(y^{-1}) \ge p$. Thus, \[\Bigl|\bigl(\mathcal{R}_2^N(y)\setminus\langle y\rangle\bigr)\cup \bigl(\mathcal{R}_2^N(y^{-1})\setminus\langle y\rangle \bigr)\Bigr|=\kappa_2^N(y) + \kappa_2^N(y^{-1}) \ge 2p\, .\]
    This concludes the proof.
\end{proof}
The following lemma provides more information on the structure of $N=\mathcal{C}_G(t)\cap H$:
\begin{lem}\label{StructureofH}
    With the same notation as in Lemma~\ref{square-rootsinN}, there exists a square root $s_0$ of $x'$ such that $s_0\in H\setminus N$. Let $y_1,\dots,y_{2p}$ be distinct elements in $\mathcal{R}_2^N\bigl( y'\bigr)\cup \mathcal{R}_2^N\bigl( (y')^{-1}\bigr)$. We have $N=\langle y_1,\dots,y_{2p}\rangle$, $(H:N)=2$, and $H=N\langle s_0\rangle$.
\end{lem}
\begin{proof}
    By Remark~\ref{noncommutativity-normalpsylow}, there exists a square root $z$ of $x'$ that does not commute with $t$. By~\eqref{setsquare-rootsG-H}, there exists $0\le k<p$ such that $s_0=t^k z t^{-k}\in H$. We have $s_0\in \mathcal{R}_2^G(x')\cap H$, and $s_0$ does not commute with $t$. Since $|G|<8p^2$, we have $|H|<8p$, and thus $|N|<4p$. Since $\langle y_1,\dots,y_{2p}\rangle$ contains more than half of the elements of $N$, it must be equal to $N$. Let $\varphi\, \colon\,  H\to \Aut(P)$ be the map defined by $\varphi(h)=(\mathrm{Int}_h)_{|P}$. We have $N=\ker \varphi$, and $\varphi(s_0)$ has order $2$. Thus, $|H|/|N|$ is a multiple of $2$. Since $|N|>2p$ and $|H|<8p$, this forces $(H:N)=2$, and thus $H=N\langle s_0\rangle$.
\end{proof}
We now reduce to the case $\Ord(y)=2$:
\begin{lem}\label{orderreduction}
     With the same notation as in Lemma~\ref{square-rootsinN}, we have $\Ord(x')=\Ord(y')=2$. Moreover, $y'$ is central in $H$.
\end{lem}
\begin{proof}
    We first prove that $\langle y'\rangle$ is normal in $H$. If $\langle y'\rangle$ is not normal in $H$, then there exists $g\in H$ such that $z=gy'g^{-1}\notin \langle y'\rangle$. Hence, $G$ satisfies the property $\mathcal{P}(2p;x',z)$. By Lemma~\ref{square-rootsinN}, we have $\left| \mathcal{R}_2^N\bigl(z\bigr)\cup \mathcal{R}_2^N\bigl( z^{-1}\bigr) \right| \ge 2p$, but this means that $|N|>4p $ (because it contains square roots of elements of $\langle z\rangle$ and $\langle y'\rangle$); this contradicts the assumption that $|G|<8p^2$. Thus, $\langle y'\rangle$ is normal in $H$.
    
    We now prove that $\Ord(y')=2$, which, combined with the normality of $\langle y'\rangle $ in $H$, proves that $y'$ is central in $H$. If $\Ord(y')$ is odd, let $y_1,\dots,y_p$ be distinct primitive square roots of $y'$, and let $\pi:N\to N/\langle y'\rangle$ be the natural projection. We have, for $i\ne j$, $\pi(y_i)\ne \pi(y_j)$. Indeed, if $y_i=(y')^k y_j $, then $y'=(y_i)^2=(y')^{2k}(y_j)^2=(y')^{2k}y'$, which implies that $(y')^{k}=1$ (that is because $(y')^2$ has the same order as $y'$). Thus, $|N|\ge p \Ord(y')$. Since $|N|<4p$, we have necessarily $\Ord(y')=3$. By Lemma~\ref{3-4thinvolutions}, since $\bigl |N/\langle y'\rangle\bigr|< 4p/3$ and contains at least $p$ involutions $\pi(y_1),\dots,\pi(y_p)$, the quotient $N/\langle y'\rangle$ is an elementary $2$-group. But $\pi(x')\ne \pi(1)$; thus $\Ord(\pi(x'))=3$. This provides the desired contradiction.
    Similarly, if $\Ord(y')$ is even, considering $y_1,\dots,y_{2p}$ to be distinct square roots, we see that $\bigl| \pi\bigl(\{y_1,\dots,y_{2p}\}\bigr)\bigr|\ge p$, which means that $|N|\ge p\Ord(y')$. This implies that $\Ord(y')=2$.
\end{proof}
\begin{rk}
The following elementary remark will be useful: let $y'\in N$ be a central element such that $\Ord(y')=2$. If $\pi\, \colon\, N\to N/\langle y'\rangle$ and $K\subset N$, then $|\pi(K)|\ge |K|/2$. Indeed, any element of $\pi(K)$ has at most $2$ pre-images in $K$, which implies that \[ |\pi(K)|=\sum_{f\in \pi(K)} 1\ge \sum_{f\in\pi(K)}\frac{\bigl|\pi^{-1}\{f\}\cap K\bigr|}{2}=\frac{|K|}{2}\, . \]
\end{rk}
In what follows, we fix $G$ to be a group of order $|G|<8p^2$ satisfying the property $\mathcal{P}(2p;x,y)$. Its $p$-Sylow $P=\langle t\rangle$ has order $p$, and we fix $H$ to be a complement of $P$ in $G$. By Lemma~\ref{square-rootsinN}, up to replacing $x,y$ with their respective conjugates $x',y'\in H$, we may assume that $x,y\in H$. We will use the notation of Lemma~\ref{square-rootsinN}. By Lemma~\ref{orderreduction} we have $\Ord(y)=\Ord(x)=2$. Thus, combining Proposition~\ref{rootcount-in-H} with Lemma~\ref{primitive2proots}, as in the proof of Lemma~\ref{square-rootsinN}, we deduce that $y$ admits $2p$ non-trivial square roots $y_1,\dots,y_{2p}\in N$. 
\begin{lem}\label{divisibilityby4}
    We have  $\bigl| \mathcal{R}_2^H(x) \setminus N\bigr|\in 4\Z$ and $\bigl| \mathcal{R}_2^H(y) \setminus N\bigr|\notin4\Z$. 
\end{lem}
\begin{proof}
    By Lemma~\ref{StructureofH} we have $H=N\langle s_0\rangle$ where $s_0\in H\setminus N$ is a square root of $x$. 
    If $s\in \mathcal{R}_2^H(x)\setminus N$, then since $y$ is central in $H$, we have $\{ s, s^{-1}, sy, s^{-1}y\}\subset  \mathcal{R}_2^H(x)\setminus N $. Moreover, the relation $\mathcal{L}$ given by: $s\, \mathcal{L}\, s'$ if and only if $s'\in \{ s, s^{-1}, sy, s^{-1}y\}$, is an equivalence relation. Thus, $\mathcal{R}_2^H(x)\setminus N$ is a disjoint union of sets of the form $\{s,s^{-1},sy,s^{-1}y\}$. Hence $\bigl| \mathcal{R}_2^H(x) \setminus N\bigr|\in 4\Z$. 
    We now prove that $\bigl| \mathcal{R}_2^H(y) \setminus N\bigr|\notin4\Z$. Assume for the sake of a contradiction that $\bigl| \mathcal{R}_2^H(y) \setminus N\bigr|\in4\Z$. There exists $k\in \Z$ such that \[ \bigl(r_2^H(x)-r_2^N(x)\bigr)-\bigl(r_2^H(y)-r_2^N(y)\bigr)=4k\, .\]
    By \eqref{square-rootcountinH}, we have for all $g\in N$, \[r_2^G(g)=r_2^N(g)+p \bigl(r_2^H(g)-r_2^N(g)\bigr)\, .\]
    Using the assumption $r_2^G(x)=r_2^G(y)$, and since $x,y\in N$, we have \[r_2^N(y)-r_2^N(x)=4pk\, .\]
    We also have that $r_2^N(y)\ne r_2^N(x)$. Otherwise, using the fact that $r_2^G(x)=r_2^G(y)$, we would have $r_2^H(y)-r_2^N(y)=r_2^H(x)-r_2^N(x)$, which implies $r_2^H(y)=r_2^H(x)$. This is a contradiction, because by~\eqref{2prootcountinH} with the assumption $r_{2p}^G(y)>r_{2p}^G(x)$ yields $r_2^H(y)> r_2^H(x)$. Thus, $k\ne 0$, meaning $N$ contains at least $4p$ elements, which contradicts $|G|<8p^2$. This finishes the proof of the lemma.
\end{proof}
\begin{cor}\label{finalstructureofH}
    We have $x\notin \mathcal{Z}(N)$ and if $N$ is a $2$-group, then there exists $s\in \mathcal{R}_2^H(y)\setminus N$ such that $\langle s\rangle $ is normal in $H$. This shows that if $N$ is a $2$-group, then \[ H= N\langle s \rangle,\quad s^2=y,\quad\text{and}\quad \langle s\rangle \lhd H\, . \]
\end{cor}
\begin{proof}
    We proceed by contradiction. If $x\in \mathcal{Z}(N)$, then, using the fact that $H=N\langle s_0\rangle$ (where $s_0$ is a square root of $x$), we deduce that $x\in \mathcal{Z}(H)$. Thus, as in the proof of Lemma~\ref{divisibilityby4}, the set $\mathcal{R}_2^H(y)\setminus N$ is a union of sets of the form $\{s,s^{-1},sx,s^{-1}x\}$ which implies that $\bigl| \mathcal{R}_2^H(y) \setminus N\bigr|\in4\Z$. This contradicts Lemma~\ref{divisibilityby4}. Now suppose that $N$ is a $2$-group and that for all $s\in \mathcal{R}_2^H(y) \setminus N$, the subgroup $\langle s \rangle$ is not normal in $H$. Since $|H|=2|N|$, we have that $H$ is a $2$-group. Thus, the orbits of the conjugation action of $H$ on the set \[\bigl\{ \langle s\rangle \, \colon\,  s\in \mathcal{R}_2^H(y) \setminus N \bigr\} \]
    are non-trivial, and their sizes are powers of $2$. Moreover, each element of the orbit gives rise to two square roots of $y$. (Note that since $\Ord(y)=2$, Lemma~\ref{rootorder} implies that for $s\in \mathcal{R}_2^H(y) \setminus N$, we have $\Ord(s)=4$, and $s,s^{-1}$ are both square roots of $y$. Furthermore, because $y$ is central in $H$, if $g\in H$ and $s^2=y$, then $(gsg^{-1})^2=y$.) This implies that $\bigl| \mathcal{R}_2^H(y) \setminus N\bigr|\in4\Z$ which contradicts Lemma~\ref{divisibilityby4}.
\end{proof}
We now use Wall's classification theorem to prove that $H$ is indeed a $2$-group.
\begin{prop}
    The subgroup $H$ is a $2$-group. 
\end{prop}
\begin{proof}
    Since $|H|=2|N|$, it suffices to prove that $N$ is a $2$-group. As the argument is quite technical, we first outline the main strategy. We proceed by contradiction, assuming that $N$ is not a $2$-group. We proceed in three steps:
    \begin{enumerate}
        \item We apply Wall's classification theorem to $N/\langle y\rangle$ to isolate its odd part. Using the Schur--Zassenhaus theorem, we lift this to a non-trivial abelian normal subgroup $B\lhd N$ of odd order.
        \item Using the property $\mathcal{P}(2p;x,y)$, we prove that $x$ acts trivially on $B$, which forces $B\subset \mathcal{C}_N(x)$.
        \item We distinguish between two cases. If $x$ commutes with at least $\lfloor 5p/4 \rfloor$ square roots of $y$, the centralizer $\mathcal{C}_N(x)$ must be a $2$-group, contradicting $B\subset \mathcal{C}_N(x)$. If $x$ commutes with less than $\lfloor 5p/4 \rfloor$ square roots, it creates enough extra involutions in $N/\langle y\rangle$ to force the entire quotient to be a $2$-group, contradicting that $N$ is not a $2$-group.
    \end{enumerate}
     Let $\pi\, \colon\, N\to N/\langle y\rangle$ be the natural projection, and let $y_1,\dots,y_{2p}$ be distinct non-trivial square roots of $y$ in $N$. We have that $\pi \bigl(\{y_1,\dots,y_{2p}\}\bigr)$ has a cardinality of at least $p$. Thus, $N/\langle  y\rangle$ has order $<2p$ and contains at least $p$ involutions. Moreover, its order is not a power of $2$. By Wall's classification theorem, $N/\langle y\rangle$ is a Wall group of type~\rom{1}. Thus, we can write $N/\langle y\rangle \cong E\times (A\rtimes\langle \tau \rangle)$, where $E$ is an elementary abelian $2$-group and $\tau $ is an element of order $2$, and $A$ is an abelian group such that for all $a\in A$ we have $\tau a \tau=a^{-1}$. Thus, considering the odd part of $A$, we see that we can write $N/\langle y\rangle=B' C$ where $|B'|$ is odd and $|C|$ is a power of $2$. Let us note that for all $g\in N/\langle y\rangle $, we have $(\mathrm{Int}_g)_{|B'}\in \bigl\{\Id_{B'},\mathrm{Inv}_{B'}\bigr\}$. We have that $\pi^{-1}(B')$ has order $2|B'|$, and $y\in \pi^{-1}(B')$ is central. Thus, $\langle y\rangle $ is a normal $2$-Sylow of $\pi^{-1}(B')$. By the Schur--Zassenhaus theorem, it admits a complement $B$ of order $|B'|$. Since $B$ has index $2$ in $\pi^{-1}(B')$, it is a normal subgroup of $\pi^{-1}(B')$.
    
    We now prove that for all $g\in N$, $(\mathrm{Int}_{\pi(g)})_{|B'}= \mathrm{Id}_{B'}\implies (\mathrm{Int}_{g})_{|B}=\Id_B$ and $(\mathrm{Int}_{\pi(g)})_{|B'}=\mathrm{Inv}_{B'}\implies (\mathrm{Int}_g)_{|B}=\mathrm{Inv}_B$. Let $g\in N$ and assume that $(\mathrm{Int}_{\pi(g)})_{|B'}= \mathrm{Id}_{B'}$. If $b\in B\setminus\{1\} $, we have $\pi(g bg^{-1})=b$ (because $\pi(B)=B'$). Thus $gbg^{-1}\in \{b,by\}$. Moreover, $gbg^{-1}$ has odd order, because it has the same order as $b$, which is an element of $B$. Also, the element $b y$ has order $2\Ord(b)$ (because $y\in \mathcal{Z}(H)$). Hence $gbg^{-1}=b$. This proves that $\mathrm{Int}_g$ restricts to the identity of $B$. If we assume that $(\mathrm{Int}_{\pi(g)})_{|B'}=\mathrm{Inv}_{B'}$, we prove similarly that for all $b\in B$, $gbg^{-1}=b^{-1}$, which proves that $(\mathrm{Int}_g)_{|B}=\mathrm{Inv}_B$. Due to the structure of $N/\langle y\rangle$, any element outside of $B'$ has even order. This proves that any element of odd order of $N$ projects onto an element of $B'$, and thus must be an element of $\pi^{-1}(B')$. However, in $\pi^{-1}(B')$ the only elements of odd order are those inside $B$. This proves that $B$ contains all elements of odd order in $N$, and $B\lhd N$, and $N/B$ is a $2$-group. Moreover, since $B\cong B'$, $B$ is abelian.
    
    We now prove that $(\mathrm{Int}_x)_{|B}=\Id_B$ and that $\pi(x)\in \mathcal{Z}\bigl(N/\langle y\rangle\bigr)$. Let $b\in B$ be an element of prime order $q\ge 3$. Since $y$ commutes with $b$, $y$ admits a non-trivial $q$-th root (namely $by$). Since $G$ satisfies $\mathcal{P}(2p;x,y)$ and $|B|=|B'|<|N/\langle y\rangle |<2p$, we have that $q<2p$ and $r_q(x)=r_q(y)$. Therefore, $x$ admits a non-trivial $q$-th root. By~\eqref{p-rootsnumber}, this implies that there exists an element $z\in G$ of order $q$ that commutes with $x$. Up to replacing $z$ with a conjugate of the form $t^k z t^{-k} \in H$, we may assume that $z\in H$. Let $Q$ be the $q$-Sylow of $B$. Since $B\lhd N$ and $N/B$ is a $2$-group, $Q$ is also a $q$-Sylow of $N$. Moreover, $Q$ is normal in $N$ (because if $g\in N$ we have $gQg^{-1}\subset B$ and thus $gQg^{-1}=Q$). Finally $Q$ is also a $q$-Sylow of $H$ that is normal in $H$ (because again, if $g\in H$, we have $gQg^{-1}\subset N$, thus $gQg^{-1}=Q$). Hence, $z\in Q\subset B$. This proves that $x$ commutes with a non-trivial element of $B$ of odd order. In particular, $z\ne z^{-1}$, which implies $(\mathrm{Int}_x)_{|B}\ne \mathrm{Inv}_B$. Hence, $(\mathrm{Int}_x)_{|B}=\Id_B$. This proves that $B\subset \mathcal{C}_N(x)$. Again, because $N/\langle y\rangle \cong E\times (A\rtimes \langle \tau\rangle)$, any element of order $2$ in $N/\langle y\rangle$ that does not act by inversion on the odd part of $A$ is central in $N/\langle y\rangle$. Thus $\pi(x)\in \mathcal{Z}\bigl(N/\langle y\rangle\bigr)$.  
    
    Two cases arise:\\
    Case 1: Assume that $x$ commutes with at least $\left \lfloor \frac{5p}{4} \right \rfloor $ elements of the set $\{y_1,\dots,y_{2p}\}$. Thus, $(N\, \colon\,  \mathcal{C}_N(x))\le 3$. Moreover, $B\subset \mathcal{C}_N(x)$. Thus, $(N\, \colon\,  \mathcal{C}_N(x))=2$. Hence $\mathcal{C}_N(x)$ is a group of order $<2p$ and contains at least $\left \lfloor \frac{5p}{4} \right \rfloor $ square roots of $y$. Thus $\mathcal{C}_N(x)/\langle y\rangle$ is a group of order $<p$ containing more than $5p/8$ square roots of $1$. By Lemma~\ref{wallsquare-rootratio} $\mathcal{C}_N(x)$ is a $2$-group. This contradicts the fact that $B\subset \mathcal{C}_N(x)$. \\
    Case 2: Assume that $x$ commutes with less than $\left \lfloor \frac{5p}{4} \right \rfloor $ elements of the set $\{y_1,\dots,y_{2p}\}$. This implies that $x$ does not commute with at least $\left \lfloor \frac{3p}{4} \right \rfloor+1 $ elements of the set $\{y_1,\dots,y_{2p}\}$. If $xy_i\ne y_i x$, then since $\pi(xy_i x)=\pi(y_i)$, we have $x y_i x=y_i\, y=y_i^{-1}$. Thus, $(xy_i)^2=1$. The set $E=\{y_1\dots,y_{2p}\}\cup \{ xy_i\, :xy_i\ne y_ix,\ 1\le i\le 2p\}$ contains at least $2p+ \lfloor3p/4\rfloor+1>11p/4$ elements. Moreover, all elements of $E$ project onto involutions in $N/\langle y\rangle$. Thus $N/\langle y\rangle$ contains more than $11p/8$ square roots of $1$. Thus \[\frac{r_2^{N/\langle y\rangle}(1)}{|N/\langle y\rangle|}>\frac{11p/8}{2p}=\frac{11}{16}>\frac{5}{8}\, .\]
    Hence by Lemma~\ref{wallsquare-rootratio} $N/\langle y\rangle$ is a $2$-group. This contradicts that $B'\subset N/\langle y\rangle$ is a non-trivial subgroup. 
    This finishes the proof of the proposition.
\end{proof}
By Corollary~\ref{finalstructureofH}, there exists $s\in \mathcal{R}_2^H(y)\setminus N$ such that $\langle s\rangle \lhd H$. Let $s_0\in \mathcal{R}_2^H(x)\setminus N$ be an element given by Lemma~\ref{StructureofH}. Then, there exists $h\in N$ such that $s_0=sh$. Thus, $shsh=x$, which implies that $s^{-1}hsh^{-1}=yx h^{-2}\in N$. We also have $s^{-1}\cdot (hsh^{-1})\in \langle s\rangle$. Thus $yxh^{-2}=s^{-1}\cdot (hsh^{-1})\in H\cap \langle s\rangle =\langle y\rangle$. Thus $h^2\in \{x,xy\}$. Denote $z:=h^2\in \{x,xy\}$. Let us note that $z\notin \mathcal{Z}(N)$, since by Corollary~\ref{finalstructureofH}, $x\notin \mathcal{Z}(N)$.
\begin{lem}
    There are at least $\lfloor p/2\rfloor+1$ elements in $\{y_1,\dots,y_{2p}\}$ that do not commute with $z$.
\end{lem}
\begin{proof}
    Assume for the sake of a contradiction that the statement of the lemma does not hold. Since $\lfloor p/2\rfloor+\lfloor 3p/2\rfloor+1=2p$, this means that $z$ commutes with at least $\lfloor 3p/2\rfloor+1$ elements in $\{y_1,\dots,y_{2p}\}$. Since $z\notin\mathcal{Z}(N)$, we have $|\mathcal{C}_N(z)|<2p$. Hence, $\mathcal{C}_N(z)/\langle y\rangle$ has order $<p$, and it contains at least $\lfloor 3p/4\rfloor+1$ square roots of $1$. By Lemma~\ref{3-4thinvolutions}, we deduce that $\mathcal{C}_N(z)/\langle y\rangle$ is an elementary abelian $2$-group. This contradicts the fact that there exists a square root $h\in N$ of $z$. Indeed, we have $h\in \mathcal{C}_N(z)$, and since $z\ne y$, $h$ projects onto an element of order $4$ in $\mathcal{C}_N(z)/\langle y\rangle$.
\end{proof}
\begin{cor}\label{FinalstructureN/y}
    Let $\pi\, \colon\,  N\to N/\langle y\rangle$ be the natural projection. We have $\pi(z)\in \mathcal{Z}(N/\langle y\rangle)$. Moreover, $N/\langle y\rangle$ is a Wall group of type \rom{1} or \rom{3}.
\end{cor}
\begin{proof}
    Since $\pi(z)$ is an element of order $2$ that is a square in the Wall group $N/\langle y\rangle$, Lemma~\ref{Wallcentralelement} shows that $\pi(z)\in \mathcal{Z}(N/\langle y\rangle )$. For each square root $y_i$ of $y$ that does not commute with $z$, since $\pi(zy_iz)=\pi(y_i)$, we have $zy_iz=y_i\cdot y=y_i^{-1}$. Thus, $\Ord(zy_i)=2$. Since there are at least $\lfloor p/2\rfloor+1$ such $y_i$, the group $N/\langle y \rangle $ contains at least $\lfloor 5p/4\rfloor+1$ square roots of $1$. Thus, \[\frac{r_2^{N/\langle y\rangle}(1)}{|N/\langle y\rangle|}>\frac{5}{8}\, .\]
    Thus, by Lemma~\ref{wallsquare-rootratio}, $N/\langle y\rangle$ is a Wall group either of type \rom{1} or \rom{3}.
\end{proof}
We are now ready to prove Theorem~\ref{lowerbound}.
\begin{proof}[Proof of Theorem~\ref{lowerbound}]
    Assume for the sake of a contradiction that there exists a group $G$ of order $|G|<8p^2$ satisfying the property $\mathcal{P}(2p;x,y)$. Following the notation used during the current subsection, we have, by Corollary~\ref{FinalstructureN/y}, that $N/\langle y\rangle$ is a Wall group of type~\rom{1} or~\rom{3}. We also have that $\pi(z)\in \mathcal{Z}(N/\langle y\rangle)$. Assume that $N/\langle y\rangle\cong E\times (A\rtimes \langle \tau\rangle)$ is a Wall group of type~\rom{1}, where $E$ is an elementary abelian $2$-group and $\tau$ is an element of order $2$ acting on the abelian group $A$ by conjugation (for simplicity of notation, we assume that the direct and semi-direct products are inner products). Let $\Psi\, \colon\,  N/\langle y\rangle\to E\times (A\rtimes \langle \tau\rangle) $ be an isomorphism, and denote $z':=\Psi(\pi(z))$. Since $z'$ is a square, we have that $z'\in A$ and there exists $h'\in A$ such that $(h')^2=z' $. Moreover, since $\Psi (\pi(h))$ is a square root of $z'$, we have $\Psi (\pi(h))=h'c$, where $c$ is a central element of order $2$. In this case, if $b\in N/\langle y\rangle$ is an element of order $2$, we have $\Psi(b\pi(h) b)\in \{\Psi(\pi(h)), \Psi(\pi(h^{-1}))\}$. This proves that for all $1\le i\le 2p$, we have \begin{equation}\label{commutarorhyi}
            \pi\bigl([h,y_i]\bigr)\in \{1,\pi(z)\}\, .
        \end{equation}
    If $N/\langle y\rangle$ is a Wall group of type~\rom{3}, its commutator group has order $2$ and is generated by the unique element of order $2$ that is a square. This implies that~\eqref{commutarorhyi} still holds in this case. Thus, in both cases, we have $[h,y_i]\in \{ 1, y, z ,zy\} $. Hence, for all $1\le i\le 2p$ we have $y_i h y_i^{-1}\in \{ h^{-1},yh^{-1},zh^{-1},zy h^{-1}\}$. Taking the squares of these elements, we deduce that for all $1\le i\le 2p$ we have \[y_i zy_i^{-1}=z\, .\]
    This contradicts the fact that $z\notin \mathcal{Z}(N)$.
\end{proof}
\subsection{Proof of Theorem~\ref{n3andn5}}
The following technical Lemma will be proved in the appendix. 
\begin{lem}\label{P(10)-generalcase}
        Let $G$ be a group satisfying $\mathcal{P}(10;x,y)$. Then, $|G|\notin\{240,270\}$.
\end{lem}
We now use Lemma~\ref{P(10)-generalcase} to prove Theorem~\ref{n3andn5}:
\begin{proof}[Proof of Theorem~\ref{n3andn5}]
        Let $G$ be a group satisfying the property $\mathcal{P}(6;x,y)$. If $\Ord(y)\ge 4$, then Lemma~\ref{lowerboundsmallp} shows that $|G|\ge 96$. If $\Ord(y)=3$, then Lemmas~\ref{primitive2proots} and~\ref{q-Sylow-divisibility} show that $6$ and $3^4=81$ divide $|G|$. This implies that $|G|\ge 162>96$. If $\Ord(y)=2$, combining Lemmas~\ref{primitive2proots} and~\ref{q-Sylow-divisibility}, we deduce that $48$ divides $|G|$. Moreover, Theorem~\ref{lowerbound} shows that $|G|\ge 8\times 3^2=72$. Thus $|G|\ge 48\times 2=96$.
        
        Let us assume now that $G$ is a group satisfying the property $\mathcal{P}(10;x,y)$. If $\Ord(y)\ge 8$, then by Lemma~\ref{lowerboundsmallp}, we have $|G|\ge 8\times 5\times 8=320$. If $\Ord(y)\in \{5,7\}$, then by Lemma~\ref{lowerboundsmallp}, we have that $|G|\ge 4\times 5\times (\Ord(y))^2\ge 20\times 25=500$. By Theorem~\ref{lowerbound}, we have $|G|\ge 8\times 25=200$. If $\Ord(y)\in\{2,4\}$, then by Lemma~\ref{q-Sylow-divisibility}, $|G|$ is a multiple of $5\times 16=80$. Since $|G|\ne 240$, we have that $|G|\ge 320$. If $\Ord(y)=3$, then $|G|$ is a multiple of $5\times 18=90$. Since $|G|\ne 270$, we have $|G|>360$. If $\Ord(y)=6$, by Lemma~\ref{primitive2proots}, we have that $60$ divides $|G|$. Since $|G|\ne 240$, we have $|G|\ge 300$. If $|G|=300$, then $25$ divides $|G|$. This contradicts Lemma~\ref{nonnormalpsylow}, which implies that $|G|\ge \min (16\times 25, 4\times 5^2 \times 6)\ge 400$. This finishes the proof of the theorem.
\end{proof}

\section{Concluding remarks and open questions}\label{sec:open_questions}

As proved in \S~\ref{ptlb}, restrictions on the order of a group $G$ satisfying the property $\mathcal{P}(2p;x,y)$ impose structural constraints on some of its subgroups and their quotients (such as those given by Wall's classification theorem). The proofs in \S~\ref{ptlb} suggest that a group of minimal order satisfying $\mathcal{P}(2p;x,y)$ must contain a subgroup $N$ whose quotient by an involution has a high proportion of involutions. These quotients, as implied by Wall's classification theorem, can be isomorphic to dihedral groups. Since the quotient of a generalized quaternion group by an involution is also a dihedral group, this structural similarity leads us to believe that for each $p$, there exists some group of minimal order satisfying $\mathcal{P}(2p;x,y)$ that contains a generalized quaternion group. The approach presented in Proposition~\ref{general-construction} provides a natural way to use generalized quaternion groups to build groups satisfying $\mathcal{P}(2p; x, y)$. For this reason, we expect that these constructions yield the minimal such structures.

Let $\mathcal{N}_p$ be the set of odd integers $n \ge 3$ such that $p \mid n(n-1)$. For any $n \in \mathcal{N}_p$, let $P_n$ denote the product of all odd primes strictly less than $p$ that divide $n(n-1)$ (with the convention that $P_n = 1$ if no such primes exist). We expect that the minimal Galois group order $n_p$ for an extension in $\mathcal{E}_p$ is given by:
\begin{equation}\label{expectednp}
n_p = \min_{n \in \mathcal{N}_p} 16n(n-1)P_n.
\end{equation}

This expected equality recovers the cases $p\in \{3,5\}$. If $\mathrm{rad}(p-1)=2$, meaning $p$ is a Fermat prime, the minimum is attained at $n=p$, yielding $n_p = 16p(p-1)$. If $p$ is a Sophie Germain prime, the minimum may shift to $n = 2p+1$. Evaluating this formula for the first few odd primes produces the following expected minimal orders:

\begin{table}[h!]
\centering
\begin{tabular}{|c|c|c|c|c|c|c|}
\hline
$p$ & 7 & 11 & 13 & 17 & 19 & 23 \\ \hline
Expected $n_p$ & 2016 & 8096 & 7488 & 4352 & 16416 & 34592 \\ \hline
\end{tabular}
\vspace{0.2cm}
\end{table}

\noindent \textbf{Question 1.} \emph{Does the expected equality~\eqref{expectednp} hold for all odd primes $p$? In particular, can one verify the values of $n_p$ for small primes, such as those listed in the table above?}

\noindent \textbf{Question 2.} \emph{Can one provide an explicit construction of groups satisfying $\mathcal{P}(d;x,y)$ for all squarefree integers $d\ge 2$? What is the minimal order of such a group? Are these minimal groups always solvable, and if so, can one relate the number of prime divisors of $d$ to the derived length of the group?}

\noindent \textbf{Question 3.} \emph{Let $d\ge 3$ be an odd squarefree integer. If a finite group $G$ satisfies $\mathcal{P}(2d;x,y)$, does there always exist an embedding of $G$ into some symmetric group $\mathfrak{S}_n$ such that the pair $(G,\mathfrak{S}_n)$ satisfies the property $\mathcal{Q}(2d;x,y)$?}

\section*{Appendix}
We prove Lemma~\ref{P(10)-generalcase}. It is natural to expect that proving that there exists no group $G$ of order $240$ satisfying the property $\mathcal{P}(10;x,y)$ should be technical since there are $208$ isomorphism classes of such groups. Let us first note that if $G$ is a group satisfying $\mathcal{P}(10;x,y)$ has order $|G|\in \{240,270\}$, then $\gcd(5,\Ord(y))=1$. Indeed, by Lemma~\ref{primitive2proots}, we have $5\Ord(y)$ divides $|G|$, which implies that $5$ does not divide $\Ord(y)$ since $25$ does not divide $|G|$. We start by proving the following partial result: 
\begin{lem}\label{P(10)-normalcase}
     Let $G$ be a group satisfying $\mathcal{P}(10;x,y)$ with a normal $5$-Sylow $P=\langle t\rangle$ of order $5$. Then, $|G|\ne 240$.
\end{lem}
\begin{proof}
    Assume for the sake of contradiction that $G$ has order $240$, and let $H$ be a complement of $P$ in $G$. By Lemma~\ref{SZapplication}, up to replacing $x$ and $y$ by their respective conjugates $x',y'\in H$, we may assume that $x,y\in H$. Since $|H|=240/5= 48=2^4\times 3$, the order of $y$ cannot be odd. Otherwise, we would have $\Ord(y)=3$, which implies, by Lemma~\ref{q-Sylow-divisibility}, that $9$ divides $|G|$, which is impossible. We deduce that $\Ord(y)$ is even. Since by Lemma~\ref{primitive2proots} we have $\kappa_{10}^G(y)\ge 2\times 5\times 4$, then by \eqref{2prootcountinH}, we deduce that $r_2^{H\cap \mathcal{C}_G(t)}(y)\ge 10$. Let us denote $N=H\cap \mathcal{C}_G(t)$. By Remark~\ref{noncommutativity-normalpsylow}, $x$ admits a square root in $H\setminus \mathcal{C}_G(t)$. The morphism given by the conjugation action of $H$ on $P$ shows that the index $(H\, \colon\,  N)$ is even (because its image contains the image of the square root of $x$ that does not commute with $t$). Moreover, $N$ contains $1$, $y$, $x$, and the elements of $\mathcal{R}_2^N(y)$, thus $|N|\ge 13$, which forces $|N|=24$. We now prove that $y$ admits exactly $12$ square roots in $N$ and $x$ has no square roots in $N$. This will provide the contradiction we are looking for. Indeed, by~\eqref{square-rootcountinH}, this implies that $5$ divides $r_2(x)$ whereas $r_2(y)\equiv 12\equiv2\pmod{5}$. 
    
    Firstly, since $\mathcal{R}_2^N(y)$ generates a group containing $1$, $y$, and at least $10$ square roots of $y$, then $\bigl|\bigl\langle\mathcal{R}_2^N(y)\bigr\rangle \bigr|\ge 12$. It cannot be of order $12$, because among the previously listed elements, there is none of order $3$. Thus, $\bigl\langle\mathcal{R}_2^N(y)\bigr\rangle=N$. In particular, $y$ is central in $N$. Let $\pi\, \colon\,  N\to N/\langle y\rangle$ be the canonical projection. We have that $\bigl|\pi(\mathcal{R}_2^N(y))\bigr|\ge 5$. Moreover, $\pi(x)\ne \pi(1)$ and $\{ \pi(1),\pi(x)\}\cap \pi\bigl(\mathcal{R}_2^N(y)\bigr)=\emptyset$. Thus, $|N/\langle y\rangle|\ge 7$. Furthermore, $|N/\langle y\rangle|$ divides $24$. Since $\Ord(y)$ is even, this implies that $|N/\langle y\rangle|=12$. Hence, $\Ord(y)=2$. Let $S$ be a $2$-Sylow of $N$. $S$ contains a conjugate $x'$ of $x$, and it also contains a square root $s$ of $y$ (because any conjugate of a square root of $y$ is itself a square root of $y$, since $y$ is central in $N$). Consequently, we can write $S=\langle s\rangle \cdot\langle x'\rangle$, which shows that either $S\cong D_8$ or $S\cong C_4\times C_2$. By Sylow's theorem, $N$ has at most $3$ $2$-Sylows. If $S\cong D_8$, $y$ would have $2$ square roots in $S$, and thus at most $6$ square roots in $N$, which contradicts that $r_2^N(y)\ge 10$. This proves that $S\cong C_4\times C_2$. Hence, $y$ has exactly $4$ square roots in each $2$-Sylow, and there must be $3$ $2$-Sylows.
    
    To see that $y$ admits exactly $12$ square roots in $N$, it suffices to prove that no two $2$-Sylows share the same square root of $y$. Let $S_1, S_2,$ and $S_3$ be the three $2$-Sylows of $N$. Because $G$ satisfies $\mathcal{P}(10;x,y)$ and $y$ commutes with an element of order $3$ in $N$, $x$ must also commute with an element of order $3$ in $G$. By the Schur--Zassenhaus theorem, all elements of order $3$ in $G$ lie in $N$ (because $N$ is a subgroup of all complements of $P$ in $G$, and its index is $2$, forcing it to contain all elements of order $3$). Let $z \in N$ be an element of order $3$ commuting with $x$. By Lemma~\ref{groupproductorder}, we have $N = S_1\langle z\rangle$. Thus, $zS_1z^{-1} \ne S_1$. 

    Assume that two $2$-Sylows, say $S_1$ and $S_2$, share a square root $s$ of $y$. Up to replacing $z$ with $z^2$, we may assume that $zS_1z^{-1}=S_2$. Thus, there exists $s' \in S_1$ such that $z s' z^{-1} = s$. Since $S_1 \cong C_4 \times C_2$, we have $\langle s', x\rangle = S_1$. Thus, \[z S_1 z^{-1} = \langle z s' z^{-1}, z x z^{-1}\rangle = \langle s, x\rangle=S_1\, .\] This is absurd.
    
    Finally, $x$ does not admit a square root in any $2$-Sylow of $N$, which implies that $x$ does not have any square root in $N$. This finishes the proof.
\end{proof}

The proof of Lemma~\ref{P(10)-generalcase} is organized as follows: 
\begin{enumerate}
    \item We prove that if $|G|\in \{240,270\}$ and $P$ is a $5$-Sylow of $G$, then $10$ divides $|N_G(P)|$.
    \item We prove that if $|G|=270$, then a $5$-Sylow $P$ satisfies $P\lhd G$. This implies that $4 \mid |G|$, which is impossible in this case. We may assume that $|G|=240$.
    \item We fix a $5$-Sylow $P=\langle t\rangle$ of $G$, where $t$ commutes with a square root $s$ of $y$. We prove that $|N_G(P)|=40$.
    \item We prove that $\Ord(y)=2$ and that $\langle s\rangle \lhd N_G(P)$. 
    \item We prove that there exists $z\in \mathcal{R}_{10}^G(y)\setminus N_G(P)$.
    \item We use the element $z$ to prove that $y\in \mathcal{Z}(G)$ and that $x$ commutes with all elements of order $5$ in $G$. This implies that $\mathcal{R}_2^{N_G(P)}(y)\subset \{s,s^{-1}, sx, s^{-1}x\}$.
    \item We conclude by obtaining a contradiction in both of the following two cases: the case where $s$ does not commute with any element of order $5$ outside $\langle t\rangle$, and the case where $s$ commutes with an element $t'\notin \langle t\rangle$ such that $t'$ has order $5$. 
\end{enumerate}
\begin{proof}[Proof of Lemma~\ref{P(10)-generalcase}]
    Assume for the sake of a contradiction that $|G|\in \{240, 270\}$. Note that since $25$ divides neither $240$ nor $270$, we have, by Sylow's theorem, that all elements of order $5$ in $G$ are in the same rational class. By Lemma~\ref{primitive2proots}, there exists an element of order $2\times 5\times \Ord(y)$ in $G$. Thus, any element of order $5$ in $G$ commutes with some element of order $2$. In particular, if $P$ is a $5$-Sylow of $G$, then the order of its normalizer $N_G(P)$ is a multiple of $10$.
        
    Assume that $|G|=270=5\times 2\times 3^3$. Since the index $(G\, \colon\, N_G(P))$ is exactly the number of $5$-Sylows of $G$, we have, by Sylow's theorem, that $(G\, \colon\,  N_G(P))\in\{1,6\} $, which implies that $(G\, \colon\,  N_G(P))=1$. Thus, $P$ is normal. By the Schur--Zassenhaus Theorem, there exists a complement $H$ of $P$. Let $\Phi\, \colon\, H\to \Aut(P)$ be the morphism given by the action of $H$ on $P$ by conjugation. Since, by Lemma~\ref{primitive2proots}, $\ker(\Phi)$ contains a non-trivial square root of $y$ and $\mathrm{Im}(\Phi)\setminus\{1\}$ contains the image of a non-trivial square root of $x$ (which has order $2$), we have that $4$ divides $|H|$, which gives a contradiction. 
        
    Let us assume now that $|G|=240=5\times 2^4\times 3$. Let $P$ be a $5$-Sylow generated by an element $t$ that commutes with a non-trivial square root $s$ of $y$. We have that $(G\, \colon\,  N_G(P))\in\{1,6, 16\}$. Since $|N_G(P)|$ is even, we have that $(G\, \colon\,  N_G(P))\in \{1,6\}$. Lemma~\ref{P(10)-normalcase} shows that $(G\, \colon\,  N_G(P))\ne 1$. Thus, $(G\, \colon\,N_G(P))=6$ and $|N_G(P)|=40$. 
        
    Let us note that since $s\in N_G(P)$ has order coprime to $5$, we have $\Ord(s)=2\Ord(y)\in\{4,8\}$. By Lemma~\ref{q-Sylow-divisibility}, this forces $\Ord(y)=2$. Let $S_0\subset N_G(P)$ be a $2$-Sylow of $N_G(P)$ containing $s$, and let $S\subset G$ be a $2$-Sylow of $G$ containing $S_0$. Since $S$ contains a conjugate of a non-trivial square root $\alpha$ of $x$, up to replacing $\alpha$ with this conjugate, we may assume $\alpha\in S$. Since $(S\, \colon\,  S_0)=2$, we have that $x\in S_0$ and $\alpha\notin S_0$ (because $S_0\lhd S$ and $S/S_0$ is a group of order $2$, thus for any $g\in S$ we have $g^2\in S_0$. Moreover, since $S=\langle s\rangle\langle \alpha \rangle$ and $s\in S_0$, we cannot have $\alpha \in S_0$). Since the subgroup $\langle t,s\rangle\subset N_G(P)$ has index $2$, and $\langle s\rangle $ is a characteristic subgroup of $\langle t,s\rangle$ (since it is the $2$-Sylow of the abelian group $\langle t,s\rangle$), we have that $\langle s\rangle$ is normal in $N_G(P)$, thus $N_G(P)\subset N_G(\langle s\rangle)$. 
        
    We now prove that inside $N_G(P)$ there are at most $20$ $10$-th roots of $y$: If $x$ does not commute with $s$, then $xsx=s^{-1}$, and hence $\langle s,x\rangle\cong D_8$ is a $2$-Sylow of $N_G(P)$. Since $\langle s\rangle$ is normal in $N_G(P)$, the unique square roots of $y$ in $N_G(P)$ are $s$ and $s^{-1}$. This means that $y$ admits at most $10$ $10$-th roots in $N_G(P)$ in the case where $x$ does not commute with $s$. Now assume that $x$ commutes with $s$. We distinguish two possibilities based on whether $x$ commutes with $t$:
    If $x$ commutes with $t$, then $N_G(P)\cong C_5 \times C_4\times C_2$. In this case, $y$ admits at most $4$ square roots (specifically $s, s^{-1}, sx,$ and $s^{-1}x$), meaning $y$ admits at most $20$ $10$-th roots in $N_G(P)$. If $x$ does not commute with $t$, then $xtx=t^{-1}$, which implies $N_G(P)\cong D_{10}\times C_4$. Here, the only square roots of $y$ that commute with $t$ are $s$ and $s^{-1}$, meaning $y$ admits at most $12$ $10$-th roots in $N_G(P)$. This proves the claim and implies, using Lemma~\ref{primitive2proots}, that there exists a primitive $10$-th root $z$ of $y$ such that $z\in G\setminus N_G(P)$. 
        
    Let us note that $z^{4}$ is an element of order $5$ that does not belong to $\langle t\rangle$. Otherwise, $\langle z^4\rangle=\langle t\rangle$, which implies that $z$ commutes with $t$, which is impossible since $z\notin N_G(P)$. This also proves that $z^4\notin N_G(P)$. The centralizer of $y$ contains $N_G(P)$ (since the normalizer of $s$ contains $N_G(P)$), and it also contains $z^4$. This forces $\bigl|\mathcal{C}_G(y)\bigr|\ge 5 |N_G(P)|$. Thus $y\in \mathcal{Z}(G)$. In particular, $y$ commutes with all elements of order $5$ in $G$, and since $r_5(x)=r_5(y)$, $x$ commutes with all elements of order $5$ in $G$. In particular, $x$ commutes with $t$, which, by the above, implies the two following possibilities: if $x$ commutes with $s$, then the only square roots of $y$ in $N_G(P)$ are $s,s^{-1}, sx,$ and $s^{-1}x$, and if $x$ does not commute with $s$, then the only square roots of $y$ in $N_G(P)$ are $s$ and $s^{-1}$. 
        
    To conclude we distinguish between two cases: 
    Case 1: $s$ does not commute with any element of order $5$ outside of $\langle t\rangle$. In this case, because $x$ commutes with all elements of order $5$, we have that square roots of $y$ in $N_G(P)$ do not commute with any element of order $5$ outside $\langle t\rangle$. We have that $z^5$ is a square root of $y$ that commutes with $z^4$. Since $z^4\notin N_G(P)$, we have that $\langle z\rangle \cap N_G(P)$ is a subgroup of $\langle z\rangle $ that does not contain $z^5$ and $z^4$ but contains $z^{10}=y$. Hence, $\langle z\rangle \cap N_G(P)=\langle y\rangle$. By Lemma~\ref{groupproductorder}, this implies that $|G|\ge \bigl| N_G(P)\,  \cdot\, \langle z\rangle\bigr|\ge \frac{40\times 20}{2}=400$, which contradicts $|G|=240$.
        
    Case 2: $s$ commutes with some element $t'\notin \langle t\rangle$ of order $5$. In this case, $N_G(\langle s\rangle)$ contains $N_G(P)$ and $\langle t'\rangle$, thus $N_G(\langle s\rangle)=G$. If $\alpha s\alpha ^{-1}=s$, then $x$ commutes with $s$. Thus, the centralizer of $s$ contains $N_G(P)$ and contains $t'$. Hence, in this case $s$ is central. If $\alpha s\alpha^{-1}=s^{-1}$, we have $\alpha s^k \alpha^{-1}=s^{-k}$ for all $k\in \Z$, in particular $\alpha s^{-1}\alpha^{-1}=s$. Hence $x sx=\alpha^2 s\alpha^{-2}=s$, which proves again that $x$ commutes with $s$ and thus the centralizer of $s$ contains $N_G(P)$ and $t'$, thus $s$ must be central contradicting that it does not commute with $\alpha$. This proves that $s\in \mathcal{Z}(G)$. In particular, $x$ commutes with $s$ and $t$, hence the centralizer of $x$ contains $N_G(P)$ and $t'$, thus $x\in \mathcal{Z}(G)$. But in this case any $2$-Sylow of $G$ is of the form $\langle s\rangle \cdot \langle d\rangle\cong C_4\times C_4$ where $d^2=x$. Thus, the only square roots of $y$ in $G$ are $s,s^{-1},sx, s^{-1}x$. Hence, in order to have $r_2(x)=r_2(y)$, this forces that there exists a unique $2$-Sylow that we denote $S$. Since the automorphism group of $C_4\times C_4$ has order $96$, we have that any element of order $5$ acts trivially on $C_4\times C_4$. This implies that all square roots of $x$ and $y$ commute with all elements of order $5$. Hence, $r_{10}(x)=r_{10}(y)$ contradicting the fact that $G$ satisfies $\mathcal{P}(10;x,y)$. \end{proof}

\end{document}